\documentclass[reqno, 11pt]{amsart}
\usepackage[utf8]{inputenc}
\usepackage[dvips]{graphicx,color}
\usepackage[all]{xy}
\usepackage{enumerate}
\usepackage[english]{babel}
\usepackage{graphicx,pstricks,pst-plot}

\usepackage{latexsym} % para el leadsto
\usepackage{xcolor} 
\usepackage{setspace}
\usepackage{cancel}

\usepackage[left=2.5cm, right=2.5cm, bottom=2cm]{geometry} 

\usepackage{amsmath, amssymb, amsthm, epsfig, mathtools}
\usepackage{hyperref, latexsym}
\usepackage{url} 
\usepackage{stmaryrd}

\newcommand{\Q}{\mathcal{Q}}
\renewcommand{\P}{\mathcal{P}}
\newcommand{\bC}{\mathbf{C}}
\newcommand{\bA}{\mathbf{A}}
\newcommand{\Ld}{\Lambda}
\newcommand{\al}{\alpha}
\newcommand{\ld}{\lambda}

\newcommand{\veps}{\varepsilon}
\newcommand{\wto}{\rightharpoonup}

\providecommand{\norm}[1]{\lVert#1\rVert}
\newcommand{\Norm}[1]{\left\lVert#1\right\rVert}
\providecommand{\inn}[1]{\langle#1\rangle}

\providecommand{\Int}[1]{\text{Int}(#1)}
\newcommand\nor[2]{\left\|#1\right\|_{#2}}%la norme
\newcommand\FL{\mathcal{F}_{n}}%flot lineaire

\usepackage{bbm} % characteristic function
\usepackage{mathtools}

\DeclarePairedDelimiter\floor{\lfloor}{\rfloor}

\def\today{\ifcase\month\or
  January\or February\or March\or April\or May\or June\or
  July\or August\or September\or October\or November\or December\fi
  \space\number\day, \number\year}
\DeclareMathOperator{\supp}{\mathrm{supp}}

\newtheorem{theorem}{Theorem}[section]
\newtheorem{lemma}[theorem]{Lemma}
\newtheorem{proposition}[theorem]{Proposition}
\newtheorem{corollary}[theorem]{Corollary}
\theoremstyle{definition}
\newtheorem{definition}[theorem]{Definition}

\theoremstyle{remark}
\newtheorem{remark}{Remark}[section]
\newtheorem{assum}{Assumption} %[section]
\newenvironment{assump}[2][]
  {\renewcommand\theassum{#2}\begin{assum}[#1]}
  {\end{assum}}

\newcommand{\mc}{\mathcal}
\newcommand{\A}{\mc{A}}
\newcommand{\B}{\mc{B}}

\newcommand{\E}{\mc{E}}

\def\H{\mathcal H}

\newcommand{\J}{\mc{J}}

\newcommand{\M}{\mc{M}}
\newcommand{\T}{\mc{T}}
\renewcommand{\O}{\mc{O}}

\newcommand{\X}{\mc{X}}
\newcommand{\Y}{\mc{Y}}
\newcommand{\C}{\boldsymbol{C}}
\newcommand{\R}{\mathbb{R}}
\newcommand{\N}{\mathbb{N}}

\newcommand{\intav}[1]{\mathchoice {\mathop{\vrule width 6pt height 3 pt depth  -2.5pt
\kern -8pt \intop}\nolimits_{\kern -6pt#1}} {\mathop{\vrule width
5pt height 3  pt depth -2.6pt \kern -6pt \intop}\nolimits_{#1}}
{\mathop{\vrule width 5pt height 3 pt depth -2.6pt \kern -6pt
\intop}\nolimits_{#1}} {\mathop{\vrule width 5pt height 3 pt depth
-2.6pt \kern -6pt \intop}\nolimits_{#1}}}

\begin{document}

\title{Global propagation of analyticity and unique continuation for semilinear waves}
\author[]{Camille Laurent, Crist\'obal Loyola}
\date{\today}
%\subjclass[2000]{XX-XXX}
%\keywords{XXX-XXX}
\address{CNRS UMR 7598 \& Sorbonne Universit\'e \\
Laboratoire Jacques-Louis Lions, F-75005, Paris, France}
\email{camille.laurent@sorbonne-universite.fr, cristobal.loyola@sorbonne-universite.fr}
\subjclass[2020]{35A20, 35B60,93B07,35L71,35L75} % 
%35A20 Analyticity in context of PDEs
%35B60 Continuation and prolongation of solutions to PDEs
%93B07 Observability
%35L71 Second-order semilinear hyperbolic equations
%35L75 - Higher-order nonlinear hyperbolic equations
%74K10 - Rods (beams, columns, shafts, arches, rings, etc.)
\keywords{propagation of analyticity, unique continuation, semilinear wave equation, semilinear plate equation}

\allowdisplaybreaks
\numberwithin{equation}{section}

\begin{abstract}
In this article, we develop a new method to prove both global propagation of analyticity and unique continuation in finite time for solutions of semilinear wave-type equations with analytic nonlinearity. It combines control theory techniques and Galerkin approximation, inspired by Hale-Raugel \cite{HR03}, to prove that analyticity in time can be propagated for the nonlinear equation from a zone where linear observability holds towards the full space.

For semilinear wave equations with Dirichlet boundary condition on a bounded domain, this implies that analyticity can be propagated to the entire domain from a subset $\omega$ that satisfies the geometric control condition. It also implies the unique continuation when the solution is assumed to be zero on $\omega$. When the nonlinearity is assumed to be subcritical and defocusing, we also obtain observability estimates in the optimal time of the geometric control condition.

For semilinear plate equations, similar propagation of analyticity is achieved by assuming the controllability of the linear Schr\"odinger equation.

\end{abstract}

\maketitle
\tableofcontents

\section{Introduction and main results}
 The aim of this article is to study how, for a certain class of evolution PDEs, properties observed from a subset $\omega\subset \M$ through time $(0, T)$ are propagated to the whole solution on $(0, T)\times \M$. Here, $\M$ can be, for instance, a compact Riemannian manifold with boundary. More precisely, we will study the following three properties:
\begin{enumerate}
    \item \textbf{Propagation of analyticity}: if the solution is analytic in time on $(0,T)\times \omega$, is the full solution  analytic in time on $(0,T)\times\M$?
    \item \textbf{Unique continuation}: if the solution is zero on $[0, T]\times \omega$, is the solution identically zero in $[0, T]\times\M$?
     \item \textbf{Observability inequality}: can we quantify the previous unique continuation property?
\end{enumerate}
The method that we develop is quite general for conservative equations and relies on observability estimates. In this paper, we will mainly give applications to nonlinear wave and nonlinear plate equations. We also provide an abstract result. We believe that it could be applied to several other systems and is amenable to generalizations. We begin by describing the results for nonlinear wave equations.
\subsection{Main results on semilinear wave equation} In this section we consider the semilinear wave equation
\begin{align}\label{eq:nlw-1}
    \left\{\begin{array}{rl}
        \partial_t^2 u-\Delta_g u+f(u)=0  &\ (t, x)\in [0, T]\times \text{Int}(\M),\\
        u_{|_{\partial\M}}=0            &\ (t, x)\in [0, T]\times\partial\M,\\
        (u, \partial_t u)(0)=(u_0, u_1) &\ x\in \M,
    \end{array}\right.
\end{align}
where $(\M, g)$ is a smooth compact connected Riemannian manifold with boundary of dimension $d$ less or equal to $3$\footnote{The restriction on the dimension is technical. The same results would likely hold in any dimension. One might modify the definition of subcriticality and require $f^{(k)}(0)=0$ for $0\leq k\leq k_d$ with $k_d$ chosen so that $f$ sends $H^{1+\sigma}_D$ to $H^{\sigma}_D$ for $1+\sigma>d/2$. Here, $H^{\sigma}_D$ is the Sobolev space adapted to the Laplace operator with Dirichlet boundary conditions.} and $\Delta_g$ is the Laplace-Beltrami operator equipped with Dirichlet boundary conditions. The nonlinearity $f: \R\to \R$ is assumed to be \emph{analytic} and  to satisfy $f(0)=0$.

We will make two kinds of assumptions concerning the regularity and the nonlinearity:
\begin{enumerate}
\item if there is no more assumption, we will choose the regularity index $\sigma\in (1/2,1]$ if $d=3$ and $\sigma\in (0,1/2)$ if $d\leq 2$ to ensure $H^{1+\sigma}\hookrightarrow L^{\infty}$.
    \item $f$ is energy subcritical: $f$ is assumed to be of polynomial type, in the sense that there exists $C>0$ such that
\begin{align}\label{hip:nonlinearity-hyp-1}
  \ |f(s)|\leq C(1+|s|)^p\ \text{and}\ |f'(s)|\leq C(1+|s|)^{p-1},
\end{align}
with $1\leq p<+\infty$ if $d\leq 2$ and $1\leq p<5$ for $d=3$. In that case, we choose $\sigma=0$ in what follows, corresponding to finite energy solutions.
\end{enumerate}
 
The observation set $\omega\subset \M$ will also always be assumed to be open. A key hypothesis will be that $\omega$ satisfies the Geometric Control Condition at time $T>0$:
\begin{assump}{GCC}\label{assumGCC}
    Every generalized geodesic of $\M$ traveling at speed $1$ meets $\omega$ in time $t\in (0, T)$.
\end{assump}
We will always assume that the Hamiltonian vector field of the wave operator does not have contact of infinite order with $\partial \M$ (see \cite{MS:78}). This ensures that the broken bicharacteristic flow is uniquely defined. We refer to Section \ref{s:geom} for details and references.

Our first main result concerns the propagation of analyticity in time.
\begin{theorem}\label{thm:analytic-prop}
    Let $\sigma\in (1/2, 1]$ if $d=3$ and $\sigma\in (0,1/2)$ if $d\leq 2$. Let $(u,\partial_t u)\in C^0([0, T], H^{1+\sigma}\cap H^1_0\times H^{\sigma}_0(\M))$ be a solution of \eqref{eq:nlw-1}. Assume that the above setting holds and assume moreover that:
    \begin{enumerate}
        \item $(\omega, T)$ satisfies the \ref{assumGCC}.
        \item $t\in (0, T)\mapsto \chi u(t,\cdot)\in H^{1+\sigma}(\M)\cap H_0^1(\M)$ is analytic for any cutoff function $\chi\in C_c^\infty(\M)$ whose support is contained in $\omega$.
        \item $s\mapsto f( s)$ is real analytic.
    \end{enumerate}
    Then $t\mapsto u(t, \cdot)$ is analytic from $(0, T)$ to $H^{2}(\M)\cap H_0^1(\M)$.
\end{theorem}

Note that, once the solution is known to be analytic in time, it is possible to get more regularity in space. For instance, this occurs when the metric is analytic.

\begin{corollary}
\label{coranalytspacetime}
Under the same assumptions as before, if, moreover, the metric and the boundary are analytic in $x$, then $u$ is analytic in all variables.
\end{corollary}

Note that the second assumption is written as an analyticity with a Banach valued application, see Section \ref{s:appcomplex}, but it can be satisfied if we have, for instance, the pointwise estimates $|\partial_t^{\alpha}\partial_x^{\beta}u(t,x)|\leq CR^{|\alpha|}\alpha!$ for $(t,x)\in (0,T)\times \omega$, $\alpha\in \N$, $|\beta|\leq 2$ together with the Dirichlet boundary condition.

These propagation results have applications to unique continuation problems. For this property, the shape of the nonlinearity dictates what type of equilibrium we will obtain.

\begin{theorem}{(Unique continuation)}\label{thm:unique-continuation-nlw}
    Assume that $(\omega, T)$ satisfies the \ref{assumGCC} and that $f$ is energy subcritical and real analytic. If one solution $U=(u,\partial_t u)\in C^0([0, T], H^1_0\times L^2(\M))$  with finite Strichartz norms \footnote{That means one of the admissible Strichartz exponent ensuring uniqueness is finite. We refer to Theorem \ref{thmStrichartz} for more precisions.} of \eqref{eq:nlw-1} satisfies $\partial_t u=0$ in $[0, T]\times \omega$, then $\partial_t u=0$ in $[0, T]\times \M$ and $u$ is an equilibrium point of \eqref{eq:nlw-1}, that is, solution of 
            \begin{align}\label{thm:eq:nlw-equilibrium}
                \left\{\begin{array}{rl}
                -\Delta_g u+f(u)=0  &\ x\in\textnormal{Int}(\M),\\
                u=0          &\ x\in \partial\M.
                \end{array}\right.
            \end{align}
    If, moreover, the nonlinearity satisfies 
    \begin{align}
    \label{defdefocusing}
        sf(s)\geq 0\quad  \textnormal{ if }\partial \M\neq \emptyset,\\
    \nonumber    sf(s)\geq \gamma s^2\quad \textnormal{ if }\partial \M= \emptyset,
    \end{align}
    for some $\gamma>0$ and for all $s\in \R$, then $u\equiv 0$.
\end{theorem}
Note that when unique continuation is not possible, we can still establish a result of finite determining modes, see Proposition \ref{propfinitedetwave}. In the first part of the theorem, $f$ can be focusing. Therefore, \eqref{eq:nlw-1} is not always globally well-posed and can lead to blow-up, that is why we make the a priori assumption of boundedness. In the case of defocusing nonlinearity, we are able to prove a quantitative version of the unique continuation, namely, an observability inequality.
\begin{theorem}
\label{thmobserintro}
 Assume that $(\omega, T)$ satisfies the \ref{assumGCC}. Assume that $f$ is analytic, energy subcritical and defocusing, that is, satisfying \eqref{hip:nonlinearity-hyp-1}, and \eqref{defdefocusing}. Then, for any $R_0>0$, there exists $C>0$ so that for any $(u_0,u_1)\in H^1_0\times L^2(\M)$, with $\nor{(u_0,u_1)}{H^1_0\times L^2}\leq R_0$, the unique solution of \eqref{eq:nlw-1} with finite Strichartz norms satisfies
    \begin{align}\label{obsevNLintro}
        C\norm{(u_0, u_1)}_{H_0^1(\M)\times L^2(\M)}^2\leq \int_0^T \norm{\mathbbm{1}_\omega\partial_t u(t)}_{L^2(\M)}^2 dt.
    \end{align}
\end{theorem}
For linear equations, this is the typical observability estimate of \cite{BLR92} that is equivalent to the controllability. Here, however, it is obtained for nonlinear equations. This type of inequality already appeared in the literature, often associated with stabilization results, but under stronger assumptions on $\omega$ and $T$ to ensure unique continuation. Indeed, the stabilization property for the damped equation is often proved thanks to an observability inequality as \eqref{obsevNLintro}. This topic has been studied extensively; for instance, see \cite{H:85}, \cite{Z:91} and \cite{D:01} for $p<3$, and for $p\in[3,5)$, our main reference is the work of Dehman, Lebeau and Zuazua \cite{DLZ03}. This work mainly addressed the stabilization problem previously described, in the Euclidean space $\R^3$ with flat metric and active damping outside of a ball, or on a bounded domain, but with damping close to the full boundary.

Our main purpose here is to extend observability to more general geometries where multiplier methods cannot be used or do not yield optimal results with respect to the geometry and optimal time. Other stabilization results for the nonlinear wave equation can be found in \cite{AIN:11} or \cite{P:24}, as well as the references therein. Some articles have addressed the more difficult critical case $p=5$; see \cite{L:11} for details. An observability estimate similar to \eqref{obsevNLintro} was also obtained by Joly and the first author in \cite{JL13} under the same Geometric Control Condition, but with a non-uniform time depending on the size of the initial data $R_0$. Their proof relied on some asymptotic regularization result of Hale-Raugel \cite{HR03}. Our scheme of proof here is inspired by their work, with the substantial modification that we always work in finite time, which requires adapting the original method.

We should also notice that, although the results of \cite{DLZ03} concern a specific geometry, by following their proof carefully, we can extract a rough statement of the type:
\begin{center}
 ``geometric control condition'' + ``unique continuation'' $\Longrightarrow$ ``observability''. 
\end{center}
Therefore, our proof will follow that line and will mainly focus on the unique continuation property.

\medskip

Note that we have assumed from the beginning that $\M$ is compact. Yet, several results in unbounded domains can be deduced from our result. We can, for example, get similar results for a compact perturbation of $\R^3$.

\begin{proposition}
\label{propR3obstacle}
Let $\Omega=\R^3\setminus \mathcal{O}$ where $\mathcal{O}$ is a bounded smooth domain (not necessarily connected). Assume that there exists $R>0$ so that $\R^3\setminus B(0,R)\subset \omega$ and that $(\omega,T)$ satisfies \ref{assumGCC}. Then, the same conclusion as Theorem \ref{thm:unique-continuation-nlw} holds.
\end{proposition}

The proof of the previous results are consequences of abstract results described in Section \ref{s:abstrIntro}. Our methods are inspired by Dynamical Systems techniques from Hale-Raugel \cite{HR03}, originally designed to study the regularity of "ancient solutions" for dissipative systems. It was noted in \cite{JL13} that such results could imply a unique continuation property with Geometric Control Condition on $\omega$, although, in infinite time or depending on the size of the data. To achieve finite-time results, the method has to be modified. The main idea in this article is to rely on observability properties instead of the decay of semigroup. This approach enables us to obtain some propagation results in finite time for solutions with zero observation. Additionally, the method is flexible enough to allow some source terms, leading to a full propagation result with optimal assumption on $\omega$ and the time $T$, from an observation analytic in time. For more details, we refer to Section \ref{s:abstrIntro}.

\subsubsection*{Literature overview on unique continuation and propagation} Let us now review some already known results and why our unique continuation result seems difficult to obtain from the already known methods and results in the literature. We also refer, for instance, to the survey \cite{LL:23Lecture} and the reference therein that contains an introduction to the unique continuation for wave type operators. For a more complete treatment of Carleman estimates and unique continuation, we also refer to \cite{LernerBook19}.

One initial approach for proving Theorem \ref{thm:unique-continuation-nlw} could involve using the general theory of unique continuation of H\"ormander \cite{H:63}. In that context, this would require considering the nonlinearity $f(u)$ as a potential term $Vu$ where $V$ has, at most, the same regularity as $u$. Achieving global unique continuation would then require iterating some local unique continuation results across a hypersurface $\{\Psi=0\}$. Yet, for a general configuration, the global geometric assumptions resulting from the use of  H\"ormander theorem for the unique continuation are not very natural and are stronger than \ref{assumGCC}. For instance, for a flat metric, the pseudoconvexity condition of a hypersurface $\{\Psi=0\}$ for the wave operator writes
\begin{align*}
X_{t}^{2}=|X_{x}|^{2} \quad \text{and}\quad d\Psi(x_0)(X)=0 \Longrightarrow Hess \Psi(x_0) (X , X) >0 \quad \text{for all } X=(X_{t},X_{x}) \in  \R^{1+d} \setminus \{0\}  .
\end{align*}
and imply a kind of convexity of the hypersurface. Typical geometric assumptions are often of ''multiplier type'' (or Morawetz type), meaning $\omega$ is a neighborhood of $\left\{x\in \partial\Omega \left|(x-x_0)\cdot n(x)>0\right.\right\}$ which are known to be stronger than the geometric control condition (see \cite{M:03} for a discussion about the links between these assumptions). Moreover, on curved spaces, this type of condition often needs to be checked by hand in each situation, which is mostly impossible in general. Additionally, concerning the unique continuation with even smooth $V$, the classical counterexamples of Alinhac-Bahouendi~\cite{AB:95}, refined by H\"ormander \cite{H:00}, are quite striking. They show that H\"ormander's pseudoconvexity condition is not far from being optimal for local unique continuation. For any $s>1$ and $d\geq 2$, they construct some $u$ and $V \in \mathcal{G}^s(B_{\R^{1+d}}(0,1), \mathbb{C})$ (Gevrey functions) so that 
\begin{align*}
\partial_{t}^{2}u-\Delta u&=V u \textnormal{ on }B_{\R^{1+d}}(0,1),\\
\supp(u)&=\left\{(t,x_{1},\dots, x_{d})\ |\  x_{1}\geq 0\right\}\cap B_{\R^{1+d}}(0,1).
\end{align*}
This suggests that in geometrical situations where the strong pseudoconvexity of the hypersurface is not satisfied, we cannot expect local unique continuation for potential a $V$ that is not analytic.

Note that quite surprisingly, even in $1$-dimensional linear hyperbolic systems of order $1$, a similar dichotomy seems to exist. It has been shown in this context by Coron-Nguyen \cite{CN:21} that for time-dependent coupling matrices, it is possible to construct counterexamples to unique continuation at some natural time while the property is true for coupling depending analytically on time.

Another counterexample that undermines a possible strategy to prove unique continuation in Theorem \ref{thm:unique-continuation-nlw} is provided by M\'etivier in \cite{M:93}. He proved that a nonlinear version of the Holmgren theorem fails in general. The operators for which it applies are not wave operators, but a nonlinear Holmgren theorem, even for a more specific class of operators has, up to our knowledge, never been obtained so far, except for scalar operators of order $1$.

Regarding the propagation of analyticity for nonlinear equations, several results date back to the 1980s and 1990s. Alinhac-M\'etivier proved in \cite{AM:84,AM:84b} that if $u$ is a regular enough solution of a general nonlinear PDE, the analyticity of $u$ propagates along any hypersurface for which the real characteristics of the linearized operator cross the hypersurface transversally. Subsequently, there has been an intense activity to understand what kind of singularities propagate for nonlinear waves. It was found that the situation is quite complicated since microlocal analytic singularities do not remain confined to bicharacteristics as in the linear case, but can give rise to nonlinear interactions. For more details, see Godin \cite{G:86} and G\'erard \cite{G:88}. 

Yet, in our geometric context, obtaining a global result from local propagation of singularities typically involves propagation from hypersurfaces of the form $S=\{\psi=0\}$ with $\psi(t,x)=\psi(x)$. In such cases, the operator is never hyperbolic with respect to $S$, and there can be some bicharacteristics transverse to $S$ as soon as $d\geq 2$. In particular, the results from \cite{AM:84,AM:84b} do not seem to apply.

The problem is better understood in the $C^{\infty}$ or $H^s$ context. For the linear equation, the propagation of $H^s$ regularity along rays for the Dirichlet problem is well understood since the work of Melrose-Sj\"ostrand \cite{MS:78}; see \cite{H:07} for a complete historical overview on the internal and boundary problem. Concerning the $H^s$ regularity of nonlinear equations, the microlocal propagation has been the object of several studies since the work of Bony \cite{B:81}. The global propagation from a set satisfying \ref{assumGCC} is proved in \cite{DLZ03} using a bootstrap argument and propagation of $H^s$ regularity for smoother source terms. It is unclear how to adapt such arguments in the analytic context. Even in the linear case $f=0$, local propagation of analytic regularity can be quite complicated, especially for glancing rays; see \cite{RS:81} for details. The propagation of analyticity that we prove may also be of interest in this context. It seems that, to achieve global propagation of analytic regularity using the propagation of the wavefront set, one would need to take into account analytic rays. They include rays arriving at the boundary in a glancing direction and potentially staying "stuck" at the boundary for a certain amount of time, even at diffractive points. It is uncertain to us whether the global propagation result can be obtained under the same assumptions using this analysis for the linear case.

Beyond H\"ormander's general result \cite{H:63} under pseudoconvexity, unique continuation for wave-type operators have been extensively studied. While it is impossible to cover them all here, we aim to present the variety of results and motivations that appear. A more precise historical overview can be found in \cite[Section 4]{LL:23Lecture}.

We begin with the unique continuation with partial analyticity, which will be crucial for obtaining Theorem \ref{thm:unique-continuation-nlw} from Theorem \ref{thm:analytic-prop}, see Section \ref{s:UCPword}. The history of this theory is quite long, with several breakthroughs and improvements that we do not detail here (see \cite[Section 4]{LL:23Lecture}). The proof of local uniqueness results across any noncharacteristic hypersurface for $\partial_t^2 - \Delta_g$ was achieved by Tataru in \cite{1995:tataru:ucp-general}, leading to the global unique continuation result in optimal time. Tataru's result is not restricted to the wave operator;
it holds for operators with coefficients that are analytic in part of the variables, interpolating between Holmgren's theorem and the H\"ormander's theorem. Technical assumptions of this article were successively removed by Robbiano-Zuily~\cite{RZ:98}, H\"ormander~\cite{Hor:97} and Tataru~\cite{Tataru:99}, leading to a very general local unique continuation result for operators with partially analytic coefficients (including as particular cases both Holmgren's and H\"ormander's theorems). This result (or some globalized version of it) will be used in our context; see Theorem \ref{thm:qucp} below. The key point in applying these results, which are linear, is that the coefficients need to be analytic in time. When applying this to nonlinear solutions such as \eqref{eq:nlw-1}, the nonlinearity $f(u)$ must be seen as a term $Vu$ with $V$ having the same regularity as $u$. This is why the propagation of analyticity in Theorem \ref{thm:analytic-prop} is crucial in order to apply this unique continuation result.

In regards to classical Carleman estimates, many authors explored the global assumptions needed to obtain them, as well as their consequences. Global Carleman estimates for waves were proved in~\cite[Chapter~4]{FI:96} and \cite{BDBE:13} with applications to controllability and inverse problems. Another line of investigations in a geometric context was taken in~\cite{DZZ:08,Shao:19,JS:21} to present geometric assumptions that would ensure the usual pseudoconvexity, with applications to observability estimates and null-controllability.

For treating nonlinear problems, admitting lower-order terms in unique continuation results is crucial. For instance, in the present work, a nonlinearity of the form $f(u)$ is treated as a term $V u$ with potential $V$. Here, a previous analysis allows to obtain that the nonlinear solution is actually very regular. Yet, having nonlinear problems in mind, it has sometimes been a goal to minimize the regularity of the admissible lower-order terms. A global unique continuation statement was proved in~\cite{Ruiz:92}, with application to energy decay for nonlinear waves. This led to some ``dispersive'' Carleman estimates with Strichartz-type spaces.  The literature is vast, and we refer, for instance, to~\cite{KRS:87,DDSF:05, KT:05}, and the references therein.

Unique continuation problems for wave operators also arise from mathematical general relativity. They have recently received a lot of attention, specifically in the context of the rigidity problem for stationary black holes. We refer to \cite{IK:15} for a precise overview of the problem and recent progress. Note also that the rigidity of stationary black holes is already known under the assumption of analyticity. So, obtaining some propagation of analyticity could be of great interest in this context.

\subsection{Main results on semilinear plate equations}
\label{s:plateintro}

As mentioned earlier, the results we obtain for the wave follow from more a general method and abstract framework that could apply to many other systems. Firstly, the result could be extended to systems of waves with the same leading order terms but coupled by lower order terms, provided that the observation is made across all components.  Additionally, we present a first application to nonlinear plate equations, focusing solely on the propagation of analyticity. However, this could represent a first step towards unique continuation in a more general setting. 
\begin{theorem}\label{thm:analytic-nlp}
    Let $(\M, g)$ be a compact connected manifold with (or without) boundary of dimension $d\leq 3$. Let $0<T<\widetilde{T}$ and $\omega\Subset \widetilde{\omega}$ two open subsets of $\M$ so that the Schr\"odinger equation is observable in $L^2$ from $\omega$ in time $T$ (see Section \ref{sssec:obs-schr} below for more precisions and examples). Assume that $f$ is real analytic with $f(0)=0$. Let $(u, \partial_t u)\in C^0([0, T], H^2\cap H_0^1\times L^2)$ be a solution of
    \begin{align}\label{eq:nlp-1intro}
        \left\{\begin{array}{rl}
            \partial_t^2 u+\Delta^2 u+f(u)=0  &\ (t, x)\in [0, T]\times \M,\\
            u_{|_{\partial\M}}=\Delta u_{|_{\partial\M}}=0  &\ (t, x)\in [0, T]\times\partial\M,
            \end{array}\right.
    \end{align}
so that for any cutoff function $\chi\in C_c^\infty(\M)$ whose support is contained in $\widetilde{\omega}$, then $t\in (0, T)\mapsto \chi u(t,\cdot)\in H^{3+\veps}(\M)\cap H_0^1(\M)$ is analytic for one $\veps>0$. Then $t\in (0, T)\mapsto \big(u(t, \cdot), \partial_t u(t,\cdot)\big)\in H^4\cap H_0^1\times H^2\cap H_0^1$ is analytic.
\end{theorem}

Note that the condition $\Delta u_{|_{\partial\M}}=0$ which does not make sense at this level of regularity is meant as an extension of the related semigroup. The nonlinear equation \eqref{eq:nlp-1intro} is meant in the sense of the Duhamel formula.

The sharp geometric condition on $\omega$ necessary for the observability of the Schr\"odinger equation remains an open question. Yet, it has been the object of many investigations. In the following situations, the observability of the Schr\"odinger equation is known, thereby allowing us to apply the previous theorem: 
\begin{enumerate}
    \item $(\M, g)$ is a compact Riemannian manifold with or without boundary and $\omega$ satisfies the \ref{assumGCC}. See Lebeau \cite{1992:lebeau:schrodinger-controle}.
    \item $(\M,g)=((0,1)^d,\text{Euclid})$ with $d\leq 3$, $\omega$ is any nonempty open set. This was first proved by Jaffard for $d=2$ and Komornik \cite{K:92} for other dimensions (actually directly for the beam equation). Other proofs have also been given later by Burq-Zworski \cite{BZ:04},  Anantharaman-Maci\`a \cite{AM:14}. Note here that the proofs are given for the torus $\mathbb{T}^d$, but an easy argument of symmetrization allows to recover the same result for the Dirichlet boundary condition.
    \item $(\M, g)=(\mathbb{D}, \text{Euclid})$ is the Euclidean closed disk in $\R^2$ and $\omega\cap \partial \M\neq\emptyset$. See Anantharaman, L\'eautaud, Maci\`a \cite[Theorem 1.2]{2016:anantharaman-leautaud-macia:wigner-schrodinger-disk},
    \item $(\M, g)$ is a compact, boundaryless connected Riemannian surface whose flow has the Anosov property, $\omega$ is any nonempty open set. See Dyatlov-Jin-Nonnenmacher \cite[Theorem 5]{DJN:22}.
    \item $(\M, g)$ is a compact, boundaryless Riemannian manifold of dimension $d$ and constant curvature $\equiv -1$, and the observation $\mc{C}\psi= a\psi$ is made through a smooth function $a$ on $M$ such that the set
    $\{\rho\in S^*\M\ |\ a^2(\phi_t(\rho))=0,\ \forall t\in \R\}$
    has Hausdorff dimension $<d$. Here $\phi_t$ is the bicharacteristic flow on $T^*\M$. See Anantharaman-Rivi\`ere \cite[Theorem 2.5]{AR:12}.
\end{enumerate}
We also provide an abstract result following Lebeau \cite{1992:lebeau:schrodinger-controle}, establishing a connection between the observability of the Schr\"odinger equation and the plate equation.

We expect that the unique continuation for the situation in Theorem \ref{thm:analytic-nlp} is true. Yet, we would need the unique continuation for the plate equation with lower order coefficients analytic in time. It is very likely to be true, but is not yet proved. We refer to \cite{Tataru:99,RZ:98,Hor:97} in analytic regularity and \cite{FLL:24} in Gevrey spaces for the closely related Schr\"odinger equation.

Systems of nonlinear plate equations as \eqref{eq:nlp-1intro} have been considered by Eller-Toundykov \cite{ET:15} for $d=2$ with $\M$ a bounded open subset of $\R^2$, where they have addressed the question of (semi-global) exact controllability for such system. It is known that establishing a unique continuation property for PDEs is a crucial step to achieve controllability results. They proved using Carleman estimates that unique continuation holds for \eqref{eq:nlp-1}, when $\omega$ is a subset of a collar neighborhood of the boundary $\partial \M$ (namely, a multiplier-type condition) and without an analyticity assumption on the coefficients, see \cite[Lemma 4.1]{ET:15}. We also refer to some recent results on nonlinear plates due to Bournissou-Ervedoza-Tucsnak \cite{TBE24}, where they show that the nonlinear system described by the von K\'arm\'an plate equation is locally exactly controllable around any stationary state defined by a real analytic function.

\begin{remark}
        Actually, in \cite{ET:15}, more general models than the one presented here have been considered. However, it was stressed by the authors that the study of \eqref{eq:nlp-1intro} presents a stepping stone to a further study of control-related questions for systems with even more complicated nonlinearities.
\end{remark}

\subsection{An abstract result}\label{s:abstrIntro} Even though we have presented several results related to the semilinear wave and plate equations, at the core of all of them lies an abstract result. This result essentially states that solutions to a nonlinear problem, with a skew-adjoint linear part, and a compact nonlinearity which is also analytic, are analytic in time, provided that the observed solution is zero.

We first need to introduce some notation and assumptions in order to state the aforementioned result. Let $T>0$ and let us consider the nonlinear observability system
\begin{align}\label{intro:eq-obs}
    \left\{\begin{array}{cc}
        \partial_t U=AU+F(U+H_1)+H_2, &\ t\in (0, T), \\
        \bC U(t)=0, &\ t\in (0, T),
    \end{array}\right.
\end{align}
on a suitable real Hilbert space $X$, where $A$ is a skew-adjoint operator on $X$, $F$ is a mapping on $X$, $H_1$ and $H_2$ are some parameters, and $\bC$ is a bounded observation operator in $X$. The first assumption dictates the class of PDEs we will be working with.
\begin{assump}{1}
\label{assumAA}$A$ is a skew-adjoint operator with domain $D(A)$ on a real separable Hilbert space $X$, so that $A^*A=-A^2$ has a compact resolvent. 
\end{assump}
This assumption allows us to introduce $X^\sigma$ as the interpolation space in between $D(A)$ and $X$ for $\sigma\in [0, 1]$, with $X^0=X$ and $X^1=D(A)$. From now on, we will work at a fixed level of regularity, for which we fix $\sigma\in [0, 1]$ and instead consider \eqref{intro:eq-obs} in $X^\sigma$. We will need to exert some control on the linear semigroup $t\in [0, T]\mapsto e^{At}\in \mc{L}(X^\sigma)$ generated by $A$ and even more, we will need some extra control at a slightly higher regularity. This will be accomplished through the following assumption.

\begin{assump}{2}
\label{assumCC}Let $\bC\in\mc{L}(X^\sigma, X^\sigma)$ be an observation operator. We assume that $t\mapsto e^{tA}$ is observable on $[0,T]$, namely, there exists a constant $\mathfrak{C}_{obs}>0$ such that
\begin{align}\label{observabstractintro}
    \norm{W_0}_{X^\sigma}^2\leq \mathfrak{C}_{\text{obs}}^2\int_0^T \norm{\bC e^{tA}W_0}_{X^\sigma}^2dt,\quad \forall W_0\in X^\sigma.
\end{align}
Furthermore, we assume that $\bC\in\mc{L}(X^{\sigma+\veps}, X^{\sigma+\veps})$ for some $\veps>0$ and that there exists (another) constant $\mathfrak{C}_{obs}>0$ such that
\begin{align}\label{observabstractvepsintro}
    \norm{W_0}_{X^{\sigma+\veps}}^2\leq \mathfrak{C}_{\text{obs}}^2\int_0^T \norm{\bC e^{tA}W_0}_{X^{\sigma+\veps}}^2dt,\quad \forall W_0\in X^{\sigma+\veps}.
\end{align}
\end{assump}

We now make the last assumption regarding the nonlinearity. For a given real Banach space $Y$, we denote the ball centered at $0$ of radius $M$ by
\begin{align*}
    \mathbb{B}_M(Y):=\{y\in Y\ |\ \norm{y}_Y\leq M\}.
\end{align*}
For a given interval $I\subset \R$, we denote the ball of radius $M$ of $C^0([0, T], Y)$ by
\begin{align*}
    \B_{M}^I(Y)=\{U\in C^0([0, T], Y)\ |\ \forall t\in I,\ \norm{U(t)}_Y\leq M\},
\end{align*}
Moreover, we introduce the canonical complexification $Y_\mathbb{C}$, defined as the set of elements $y_1+iy_2$, $y_j\in Y$, see \cite[Section 2]{BS:71-polynomials} for more details. We then introduce the notation for the \emph{cylinder} on $Y_\mathbb{C}$
\begin{align*}
    \mathbb{B}_{M, \delta}(Y)=\{y\in Y_\mathbb{C}\ |\ \norm{\Re(y)}_Y\leq M\ \text{and}\ \norm{\Im(y)}_Y\leq \delta\},
\end{align*}
and similarly on $C^0(I, Y_\mathbb{C})$
\begin{align*}
    \B_{M, \delta}^I(Y)=\{U\in C^0(I, Y_\mathbb{C})\ |\ \forall t\in I,\ \norm{\Re(U(t))}_Y\leq M\ \text{and}\ \norm{\Im(U(t))}_Y\leq \delta\}.
\end{align*}
The latter space is naturally endowed with the $L^\infty(I,Y_\mathbb{C})$-norm. When working on the complex plane $\mathbb{C}$, we will simply denote by $B_\mathbb{C}(z_0, r)$ a complex ball of center $z_0$ and radius $r>0$. With the previous notations, we will make the following assumption on $F$.
\begin{assump}{3}
\label{assumFholom}
$F$ is a nonlinear Lipschitz and bounded operator from $\mathbb{B}_{4R_{0}}(X^{\sigma})$ to $X^{\sigma+\veps}$ for some $R_{0}>0$, $\sigma>0$ and $\veps>0$\footnote{Here, we mean that $F$ is holomorphic on $\Int{\mathbb{B}_{4R_{0}}(X^{\sigma}})$ with uniform estimates and continuous extension up to the boundary. We will often make the same slight abuse of notation in the rest of the article.}. Furthermore, there exists $\delta>0$ so that $F$ has a holomorphic extension from $\mathbb{B}_{4R_{0},2\delta}(X^{\sigma})$ to $X^{\sigma+\veps}_{\mathbb{C}}$. Moreover, there exists $C>0$ so that 
\begin{align}
\label{boundFanalyticintro}
\nor{F(U_{0})}{X^{\sigma+\veps}_{\mathbb{C}}}\leq C,\quad \nor{F(U_{0})-F(V_{0})}{X^{\sigma+\veps}_{\mathbb{C}}}\leq C \nor{U_{0}-V_{0}}{X^{\sigma}_{\mathbb{C}}}
\end{align}
for any $U_{0},~V_{0}\in \mathbb{B}_{4R_{0},2\delta}(X^{\sigma})$.
\end{assump}
We will also need a technical assumption on the pair $(A, \bC)$ related to commutator estimates useful for the regularization.
\begin{assump}{4}
\label{assumcommu}
There exists $s>0$ so that $[(A^*A)^s,\bC]\in \mc{L}(X^{\sigma+2s}, X^{\sigma+\veps})$.
\end{assump}

We now state the main abstract result.

\begin{theorem}\label{thmabstractanalyticintro}
Let $R_{0}>0$ and $T>0$. Assume that Assumptions \ref{assumAA}, \ref{assumCC} (with $T$), \ref{assumFholom} (with $R_{0}$) and \ref{assumcommu} hold. Let $T^{*}>T$. Let $H_{1}\in \B_{R_{0}}^{[0,T^{*}]}(X^{\sigma})$ and $H_{2}\in C^{0}([0,T^{*}],X^{\sigma+\veps})$ that admit some extension in $C^{0}([0,T^{*}]+i[-\mu,\mu],X^{\sigma}_{\mathbb{C}})$, resp. $C^{0}([0,T^{*}]+i[-\mu,\mu],X^{\sigma+\veps}_{\mathbb{C}})$, with $\mu>0$, so that the application
\begin{align*}
    \left\{\begin{array}{rcl}
(0,T^{*})+i(-\mu,\mu) & \longrightarrow&X^{\sigma}_{\mathbb{C}}\\
  z&\longmapsto& H_{1}(z)
    \end{array}\right.
\end{align*}
is holomorphic. We assume the same for $H_{2}$ with value in $X^{\sigma+\veps}_{\mathbb{C}}$. We assume moreover that $\Re H_{1}(z)\in \mathbb{B}_{R_0}(X^{\sigma})$ for any $z\in [0,T^{*}]+i[-\mu,\mu]$.

Then, any solution $U\in C^{0}([0,T^{*}],X^{\sigma})$ and satisfying 
\begin{align}\label{UCabstractT*intro}
    \left\{\begin{array}{cr}
         \partial_t U=AU+F(U+H_{1})+H_{2}& \textnormal{ on } [0,T^{*}], \\
         \bC U(t)=0 &\textnormal{ for }t\in [0,T^{*}],
    \end{array}\right.
\end{align}
is real analytic in $t$ in $(0,T^*)$ with value in $X^{\sigma}$.
\end{theorem}

The previous result can still hold with some variants of the aforementioned assumptions (for instance, considering different assumptions on $A$). However, we have chosen to keep enough generality to showcase that the technique can be applied to other PDEs. This result can be seen as a sort of finite-time adaptation of an abstract result due to Hale-Raugel \cite[Theorem 2.5]{HR03} with the flexibility of adding analytic source terms.

Hale-Raugel were concerned with the regularity properties (including analyticity) of evolutionary equations whose solutions are defined on $t\in\R$ and lie on a compact invariant set. They prove that such solutions are as smooth (in time) as the nonlinearity, encompassing a wide range of PDEs, including dissipative hyperbolic equations, which, unlike parabolic ones, do not have smoothing properties. Their proof relies upon a generalized Galerkin procedure, already used by dynamicians to study the regularity of attractors in different contexts, see for instance \cite{FT:89,Go:00-torus,Go:18} and the references therein. The approach of Hale-Raugel was to find the high-frequency component as a fixed point of an associate adequate mapping, depending on the low-frequency component of the solution. Then, the problem is reduced to study the system associated with the low-frequency component and the corresponding fixed point map parameterized by it.  We can roughly say that there are two key hypotheses for the technique to work: the \emph{exponential decay} of the linear semigroup and some sort of \emph{compactness} on the nonlinearity. To adapt Hale-Raugel's technique to a finite-time setting, we replace the decay of the linear semigroup by its finite-time counterpart: the \emph{observability} of the linear semigroup. We are then led to solve a nonlinear observability system to find the high-frequency component. To this end, we will heavily rely on the observability properties of the linear semigroup generated by $A$ and the compactness of the nonlinearity $F$ to set up an appropriate fixed point. It is worth mentioning that, to stabilize some semilinear damped wave equations without \ref{assumGCC}, Joly and the first author \cite{2020:joly-laurent:decay-nlw-no-gcc} successfully adapted this technique when a weaker decay of the linear semigroup is assumed.

\subsection{Outline of the article}Section \ref{sec:abstract-construction} is devoted to the proof of the abstract Theorem \ref{thmabstractanalyticintro}. It also contains several preliminaries as the property of finite determining modes and the abstract propagation of regularity. Section \ref{SEC:wave-eq} contains the applications to the nonlinear wave equation. It contains the verification that the abstract Theorem \ref{thmabstractanalyticintro} can be applied for a sufficiently high regularity index $\sigma$. It also contains some propagation of regularity arguments that allow to reach this regularity $\sigma$ starting from the energy space. Section \ref{s:Plate} contains the applications to the plate equation. In the Appendix, we gathered some results about complex analysis in Banach spaces and some geometric facts about the generalized geodesic flow that are used in the rest of the article.

\subsection{Acknowledgement} The present article is certainly a consequence of all the earliest discussions between Romain Joly and both authors. We warmly thank him for everything it brought to this work. 

The second author has received funding from the European Union's Horizon 2020 research and innovation programme under the Marie Sk\l{}odowska-Curie grant agreement No 945332.

\section{Analytic reconstruction for nonlinear observability systems}\label{sec:abstract-construction}

The purpose of this section is to prove Theorem \ref{thmabstractanalyticintro}. For the reader's convenience, we will briefly outline its proof, following the Galerkin decomposition introduced in Hale-Raugel \cite{HR03}. The key replacement will be the observability estimate that allows to reconstruct the state from the observation, at least for the high-frequency part.

Let $T>0$ and fix $\sigma\in [0, 1]$. From Assumption \ref{assumAA}, we introduce the low and high-frequency projections $\P_n=\mathbbm{1}_{[0, n]}\big((AA^*)^{1/2}\big)$ and $\Q_n=I-\P_n$, respectively. Let $U=U(t)$ be a mild solution of \eqref{UCabstractT*intro} in $C^0([0, T], X^\sigma)$ and suppose $H_1=0$, $H_2=0$ for simplicity. Let us consider the splitting
\begin{align*}
    U(t)=\P_n U(t)+\Q_nU(t)=V(t)+W(t),
\end{align*}
where $\big(V(t), W(t)\big)$ solves the following system
\begin{align*}
    \left\{\begin{array}{rl}
         \partial_t V(t)&=AV(t)+\P_n F(V+W),  \\
         \partial_t W(t)&=AW(t)+\Q_n F(V+W), \\
         \bC V(t)&=-\bC W(t).
    \end{array}\right.
\end{align*}
By Duhamel's formula, the high-frequency component $W$ can be written as
\begin{align*}
    W(t)=e^{tA}W(0)+\int_0^t e^{A(t-s)}\Q_n F\big(V(s)+W(s)\big)ds.
\end{align*}
The observation condition $\bC V=-\bC W$ suggests that given $V$, we can reconstruct $W$ by considering the corresponding nonlinear observability system. Indeed, according to Assumption \ref{assumCC}, the observability of the linear semigroup $t\in [0, T]\mapsto e^{tA}$ enables us to construct an initial condition $W(0)$ solely in terms of an observation. Forgetting first about the source term given by the nonlinearity, that would allow to reconstruct $W(0)$ in terms of the observation of $W$, which is $\bC W=-\bC V$. Lemma \ref{lmCauchyobs} below provides a generalization of this reconstruction problem when source terms are present. In this context, the first part of Assumption \ref{assumFholom}, namely, that $F$ is a nonlinear map from bounded sets of $X^\sigma$ into $X^{\sigma+\veps}$, will allow, at frequency sufficiently large, to consider the nonlinearity as a perturbation and to complete this reconstruction procedure.  More precisely, this will imply the existence of a nonlinear map $\mc{N}$ such that $W(0)=\mc{N}(V)$. Consequently, we have the formula $W=\Phi_V(W)$, where $\Phi_V: C^0([0, T], \Q_nX^\sigma)\to C^0([0, T], \Q_nX^\sigma)$ is given by
\begin{align*}
    \Phi_V(W)(t)=\mc{N}V(\cdot)+\int_0^t e^{A(t-s)}\Q_n F\big(V(s)+W(s)\big)ds,\ t\in [0, T].
\end{align*}
This suggests that we can find the high-frequency component $W^*$ as a fixed point with the low-frequency component $V$ as an input. 

At this point, the solution $U$ can be represented as $U(t)=V(t)+W^*(V)(t)$, where $V$ solves
\begin{align}\label{intro:eq-low-split}
    \partial_t V(t)=AV+\P_n F(V+W^*(V)).
\end{align}
To demonstrate that $t\in (0, T)\mapsto U(t)\in X^\sigma$ is analytic, the second part of Assumption \ref{assumFholom}, namely, that $F$ admits a holomorphic extension, is essential. This will be achieved by establishing that $t\mapsto V(t)$ and $V\mapsto W^*(V)$ are both analytic maps. If instead, we consider \eqref{intro:eq-low-split} as a differential equation on the space Banach space $C^0([0, T], \P_n X^\sigma)$, classical ODEs theory imply that $t\mapsto V(t)$ is as smooth as $F$, and therefore analytic. The uniform contraction principle further ensures that $W^*$ depends analytically on $V$, from which the result follows.

In what follows, we develop these ideas towards the proof of the main theorem. Throughout the section, we will prove several intermediate results, some of which are of independent interest. For instance, we mention Proposition \ref{prop:finite-det-modes} (Finite determining modes) and Proposition \ref{prop:abs-prop-reg} (Propagation of regularity) below.

\subsection{Linear reconstruction, determining modes and propagation of regularity} In this section, we will revisit the assumptions outlined in the introduction, organizing them according to the different results we aim to establish in order to prove our main result.

From now on, we work under Assumption \ref{assumAA}. We will now list some consequences of such an assumption. That $A^*A=-A^2$ is non-negative self-adjoint, allows us to define the Hilbert space $X^{\sigma}=D((A^*A)^{\sigma/2})$ for any $\sigma\in\R$. Note that the assumptions imply $X^{\sigma+\veps}\hookrightarrow X^{\sigma} $ for any $\veps>0$. Unless specifically noticed, we will often omit the embedding $\iota: X^{\sigma+\veps}\to X^{\sigma}$. 

By the spectral theorem, and since $A^*A$ has a compact resolvent, thus, the spectrum of $A^*A$ is real and discrete, allowing us to construct an orthonormal basis of eigenvectors of $A^*A$ in $\mathcal{H}$, denoted by $(E_j)_{j\in\N}$ and associated to the nonnegative eigenvalues $(\lambda_j)_{j\in \N}$ (ranged increasingly) with $\lambda_j \underset{j\to +\infty}{\longrightarrow}+\infty$. We introduce the high-frequency projectors $\Q_n$ on the space $\overline{\text{Span}\{E_j\}_{j\geq n}}$ and then we set the low-frequency projection $\P_n=I-\Q_n$. Note that $A$ commutes with $\P_n$ and $A\P_n$ is a bounded operator of $X^{\sigma}$ to itself with norm $\inn{\ld_n}$.

The parameter $\sigma$ will be fixed from now on. We will use the notation $\P_n X^{\sigma}$ or $\Q_n X^{\sigma}$ for that related image of the Hilbert space endowed with the topology of $X^{\sigma}$. The restriction of the embedding $\iota$ to $\Q_n X^{\sigma+\veps}$(denoted with the same name) $\iota: \Q_n X^{\sigma+\veps}\to \Q_nX^{\sigma}$ has norm $\inn{\lambda_n}^{-\veps}$. 

By the spectral theorem, since the spectrum of $A$ is purely imaginary, we can define $e^{tA}$ for any $t\in \R$, together with the estimate $\norm{e^{tA}U_0}_{X^{\sigma}}=\norm{U_0}_{X^{\sigma}}$. Also, $e^{tA}$ commutes with $\P_n$ and $\Q_n$.

\subsubsection{Linear reconstruction} We will now introduce the following assumption regarding the observability of the semigroup generated by $A$.
\begin{assump}{2a}
\label{assumC}Let $\bC\in\mc{L}(X^\sigma, X^\sigma)$ be an observation operator. We assume that $t\mapsto e^{tA}$ is observable on $[0,T]$, namely, there exists a constant $\mathfrak{C}_{obs}>0$ such that
\begin{align}\label{observabstract}
    \norm{W_0}_{X^\sigma}^2\leq \mathfrak{C}_{\text{obs}}^2\int_0^T \norm{\bC e^{tA}W_0}_{X^\sigma}^2dt,\quad \forall W_0\in X^\sigma.
\end{align}
\end{assump}
Let $\mc{O}\in \mc{L}(X^\sigma, L^2([0, T], X^\sigma))$, defined by $\mc{O}:=\bC e^{\cdot A}$, be the observation operator of linear waves and $\mc{O}_n:=\mc{O}|_{\Q_nX^\sigma}$ the \emph{high-frequency} observation operator. The observability inequality \eqref{observabstract} can be written
\begin{align}\label{observabstractO}
    \norm{W_0}_{X^\sigma}&\leq \mathfrak{C}_{\text{obs}}\nor{\O W_{0}}{L^{2}([0,T],X^\sigma)},\quad \forall W_0\in X^\sigma,\\
 \label{observabstractOn}    \norm{W_0}_{X^\sigma}&\leq \mathfrak{C}_{\text{obs}}\nor{\O_{n} W_{0}}{L^{2}([0,T],X^\sigma)},\quad \forall W_0\in \Q_n X^\sigma,
    \end{align}
where the last inequality is uniform in $n\in \N$.
It implies that $\mc{O}$ is injective and it has a closed range. Moreover, since $\Q_nX^\sigma$ is closed in $X^\sigma$, then the \emph{high-frequency} observation operator $\mc{O}_n:=\mc{O}|_{\Q_nX^\sigma}$ has closed range as well, which allows us to define $\Pi_n$ as the orthogonal (according to the natural scalar product in $L^2([0, T], X^\sigma)$) projection onto its image $\text{Im}(\mc{O}_n)\subset L^2([0, T], X^\sigma)$. From now on, we equip $\Y_n:=\text{Im}(\O_n)$ with the induced topology from $L^2([0, T], X^\sigma)$ which makes it a Banach space. By \eqref{observabstractOn}, we know that $\O_n: \Q_n X^\sigma\to \Y_n$ is a bijection, and hence, $\Y_n$ is closed, a bounded reconstruction operator $\O_n^{-1}: \Y_n\to \Q_n X^\sigma$ exists.

By applying the observability inequality \eqref{observabstractOn}, consequence of Assumption \ref{assumC}, we get for any $y\in \Y_n \subset L^2([0, T], X^\sigma)$,
\begin{align}
\label{boundOn-1}
    \norm{\O_n^{-1}y}_{X^\sigma}&\leq \mathfrak{C}_{\text{obs}}\norm{\O_{n}\O_n^{-1}y}_{L^2([0, T], X^\sigma)}=\mathfrak{C}_{\text{obs}}\norm{y}_{L^2([0, T], X^\sigma)},
\end{align}
where again, the inequality is uniform in $n\in \N$.

To ease notation, we consider the operator $\mc{I}(t): g\mapsto \int_0^t e^{A(t-s)}g(s)ds$ and we denote $\mc{I}(\cdot)$ when the operator is seen with value in $C^{0}([0,T],Y)$ for a suitable Banach space. The above construction will enable us to solve an \emph{observability Cauchy problem}, which is the content of the following Lemma.

\begin{lemma}
\label{lmCauchyobs}
There exists $C(T,\mathfrak{C}_{\text{obs}},\nor{\bC}{\mc{L}(X^\sigma)})>0$ so that for any $n\in\N$, $H\in L^{1}([0,T],X^{\sigma})$ and $G\in L^{2}([0,T],X^{\sigma})$, there exists a unique $W\in C^{0}([0,T],\Q_{n}X^{\sigma})$ solution of 
\begin{align}\label{solprojectlin}
    \left\{\begin{array}{rl}
        \partial_t W(t)&=AW(t)+\Q_{n}H,  \\
   \Pi_n\bC W    &=\Pi_{n}G. 
    \end{array}\right.
\end{align}
It satisfies $W(0)= W_0:=\O_n^{-1}\Pi_n\left[G-\bC\mc{I}(\cdot)\Q_n H\right]$ and is given by $W(t)=e^{tA}W_{0}+\mc{I}(t)\Q_nH$. We denote $W:=\FL(G,H)$ the associated linear operator. Moreover, we have the estimate
\begin{align}
\label{estimFL}
\nor{\FL(G,H)}{C^{0}([0,T],\Q_{n}X^{\sigma})}\leq C \nor{\Pi_{n}G}{L^{2}([0,T],X^{\sigma})}+C\nor{\Q_{n} H}{L^{1}([0,T],X^{\sigma})}
\end{align}
\end{lemma}
\begin{proof}
To be solution of the first line of \eqref{solprojectlin}, it is equivalent to be written as a Duhamel formula $W(t)=e^{tA}W_{0}+\mc{I}(t)\Q_nH$ for some $W_{0}\in \Q_{n}X^{\sigma}$. So, we only need to compute $W_{0}$. With the previous formula, we have $W\in C^{0}([0,T],\Q_{n}X^{\sigma})\subset L^{2}([0,T],\Q_{n}X^{\sigma})$ and we can compute
 \begin{align*}
    \Pi_n\bC W=   \Pi_n\bC \left[e^{\cdot A}W_{0}+\mc{I}(\cdot)\Q_nH \right]=\Pi_n\bC \left[e^{\cdot A}W_{0} \right]+\Pi_n\bC \left[\mc{I}(\cdot)\Q_nH \right].
 \end{align*}
 Note that if $W_{0}\in \Q_{n}X^{\sigma}$, then $\bC\left[e^{\cdot A}W_{0}\right]=\O W_{0}=\O_{n} W_{0}$ and therefore, $\Pi_{n}\bC\left[e^{\cdot A}W_{0}\right]=\Pi_{n}\O_{n} W_{0}=\O_{n} W_{0}$ by definition of $\Pi_{n}$. 
 
 In particular, since both belong to $\mathcal{Y}_n$, we want $ \Pi_n\bC W=\Pi_{n}G$, we should have 
 \begin{align*}
\O_n^{-1}\Pi_{n}G= \O_n^{-1} \Pi_n\bC W =\O_n^{-1}\O_{n} W_{0}+\O_n^{-1}\Pi_n\bC \left[\mc{I}(\cdot)\Q_nH \right]=W_{0}+\O_n^{-1}\Pi_n\bC \left[\mc{I}(\cdot)\Q_nH \right].
 \end{align*}
 This gives the $W_{0}$ given. It indeed belongs to $\Q_{n}X^{\sigma}$ and therefore $W$, as defined, satisfies the second line of \eqref{solprojectlin} by reproducing the same computation backward, that is 
  \begin{align*}
\Pi_n\bC W&=\Pi_n\bC\left[e^{\cdot A}W_{0}+\mc{I}(\cdot)\Q_nH\right] =\Pi_{n} \O_{n}W_{0}+\Pi_n\bC\mc{I}(\cdot)\Q_nH\\
&=\Pi_{n} \O_{n}\O_n^{-1}\Pi_n\left[G-\bC\mc{I}(\cdot)\Q_n H\right]+\Pi_n\bC\mc{I}(\cdot)\Q_nH\\
&=\Pi_n\left[G-\bC\mc{I}(\cdot)\Q_n H\right]+\Pi_n\bC\mc{I}(\cdot)\Q_nH=\Pi_{n}G.
\end{align*}
 The uniqueness could actually be obtained from the unique definition of $W_{0}$ that we obtained, but we prefer to give a precise proof. We consider the difference $R=W_{1}-W_{2}\in C^{0}([0,T],\Q_{n}X^{\sigma})$ between two such solutions $W_{1}$ and $W_{2}$. It satisfies
 \begin{align*}
    \left\{\begin{array}{rl}
        \partial_t R(t)&=AR(t)   \\
   \Pi_n\bC R    &=0.  
    \end{array}\right.
\end{align*}
That is $R(t)=e^{tA}R(0)$ and $\Pi_n\bC R=\Pi_{n}\O R(0)=\Pi_{n}\O_{n}R(0)=\O_{n}R(0)$. In particular, $R(0)=0$ by injectivity of $\O_{n}$.

Concerning the estimates, since $A$ is skew-adjoint on $X^{\sigma}$, standard semigroup estimates give
\begin{align}
\nor{W}{C^{0}([0,T],\Q_{n}X^{\sigma})}\leq \nor{W_{0}}{X^{\sigma}}+\nor{\Q_{n} H}{L^{1}([0,T],X^{\sigma})}.
\end{align}
So, we need to estimate $W_{0}$. For any $\widetilde{G}\in L^2([0, T], X^\sigma)$, applying \eqref{boundOn-1} to $\Pi_n \widetilde{G}\in \Y_n$, we have 
\begin{align*}
    \norm{\O_n^{-1}\Pi_n \widetilde{G}}_{X^\sigma}&\leq \mathfrak{C}_{\text{obs}}\norm{\Pi_n\widetilde{G}}_{L^2([0, T], X^\sigma)}. 
\end{align*}
In particular, we can estimate $W_{0}$ by
\begin{align*}
\nor{W_{0}}{X^{\sigma}}= \nor{\O_n^{-1}\Pi_n\left[G-\bC\mc{I}(\cdot)\Q_n H\right]}{X^{\sigma}}\leq \mathfrak{C}_{\text{obs}}\norm{\Pi_nG}_{L^2([0, T], X^\sigma)}+\mathfrak{C}_{\text{obs}}\norm{\Pi_n\bC\mc{I}(\cdot)\Q_n H}_{L^2([0, T], X^\sigma)}.
\end{align*}
We can finally estimate by the unitarity of $\Pi_{n}$ and H\"older inequality in time
\begin{align*}
\norm{\Pi_n\bC\mc{I}(\cdot)\Q_n H}_{L^2([0, T], X^\sigma)}&\leq \norm{\bC\mc{I}(\cdot)\Q_n H}_{L^2([0, T], X^\sigma)}\leq \nor{\bC}{\mc{L}(X^\sigma)}\norm{\mc{I}(\cdot)\Q_n H}_{L^2([0, T], X^\sigma)}\\
&\leq T^{1/2}\nor{\bC}{\mc{L}(X^\sigma)}\norm{\mc{I}(\cdot)\Q_n H}_{L^{\infty}([0, T], X^\sigma)}\leq T^{1/2}\nor{\bC}{\mc{L}(X^\sigma)}\norm{\Q_n H}_{L^{1}([0, T], X^\sigma)}.
\end{align*}
Recollecting the previous estimates, we have finally proved
\begin{align*}
\nor{W}{C^{0}([0,T],\Q_{n}X^{\sigma})}\leq \mathfrak{C}_{\text{obs}} \nor{\Pi_{n}G}{L^{2}([0,T],X^{\sigma})}+\left(1+T^{1/2}\mathfrak{C}_{\text{obs}}\nor{\bC}{\mc{L}(X^\sigma)}\right)\nor{\Q_{n} H}{L^{1}([0,T],X^{\sigma})}.
\end{align*}
\end{proof}
\begin{remark}
It will be very important for what follows that the constant $C$ involved in the previous Lemma is independent of $n\in\N$.
\end{remark}

\subsubsection{Finite determining modes}
As a first direct consequence of the previous result, we can get a finite determining mode result: two solutions of a nonlinear equation with the same observation and the same low frequency are the same. This result will not be used directly later, but can be considered as an easier version of what will follow where we will actually construct the reconstruction operator and study its regularity. 

We make the following assumption on the nonlinearity $F$, akin to a compactness property.

\begin{assump}{3a}
\label{assumF}
$F$ is a nonlinear operator from $\mathbb{B}_{4R_{0}}(X^{\sigma})$ to $X^{\sigma+\veps}$ for some $R_{0}>0$, $\sigma>0$ and $\veps>0$. Moreover, there exists $C>0$ so that 
\begin{align}
\label{boundF}
\nor{F(U_{0})}{X^{\sigma+\veps}}\leq C,\quad \nor{F(U_{0})-F(V_{0})}{X^{\sigma+\veps}}\leq C \nor{U_{0}-V_{0}}{X^{\sigma}}
\end{align}
for any $U_{0},~V_{0}\in \mathbb{B}_{4R_{0}}(X^{\sigma})$.
\end{assump}

Note that the second bound implies the first one with another constant. We show the following property of finite determining modes.

\begin{proposition}\label{prop:finite-det-modes}
Let $R_{0}>0$. Under Assumptions \ref{assumAA}, \ref{assumC} and \ref{assumF} (with $R_{0}$), there exists $n\in \N$ such that the following holds. Let $H\in L^1([0,T],X^{\sigma})$ and $G\in L^2([0, T], X^\sigma)$. 

Let $U(t)$ and $\widetilde{U}(t)$ be two solutions on $(0,T)$ of
\begin{align*}
    \left\{\begin{array}{cl}
        \partial_t U=AU+F(U)+H, &\ \text{on } (0, T), \\
        \bC U(t)=G(t), & \text{for } t\in (0, T),
    \end{array}\right.
\end{align*}
such that $\norm{U(t)}_{X^\sigma}\leq R_{0}$ and $\norm{\widetilde{U}(t)}_{X^\sigma}\leq R_{0}$ for all $t\in [0, T]$. If $\P_n U(t)=\P_n \widetilde{U}(t)$ for all times $t\in [0, T]$, then $U(t)\equiv \widetilde{U}(t)$ for all $t\in [0, T]$.
\end{proposition}
\begin{proof}
    By assumption $\P_{n}U=\P_{n}\widetilde{U}$ as applications in $\B_{R_{0}}^{[0,T]}(X^{\sigma})$. Let us consider the difference of solutions $Z=U-\widetilde{U}$. It satisfies 
    \begin{align*}
    \left\{\begin{array}{cc}
        \partial_t Z=AZ+F(U)-F(\widetilde{U}), &  \\
        \bC Z=0. &
    \end{array}\right.
    \end{align*}
 Moreover, we have $\P_n Z=0$, that is $Z\in C^{0}([0,T],\Q_{n}X^{\sigma})$   and therefore, applying $\Q_n$, it also satisfies 
        \begin{align*}
    \left\{\begin{array}{cc}
        \partial_t Z=AZ+\Q_n \left(F(U)-F(\widetilde{U})\right), &  \\
       \Pi_{n}\bC Z=0. &
    \end{array}\right.
    \end{align*}
In particular, we are in the situation of Lemma \ref{lmCauchyobs} and $Z=\FL(0,\Q_n \big(F(U)-F(\widetilde{U})\big))$ so that estimate \eqref{estimFL} gives
\begin{align}
\nor{Z}{C^{0}([0,T],\Q_{n}X^{\sigma})}\leq C \nor{\Q_n \left(F(U)-F(\widetilde{U})\right)}{L^{1}([0,T],X^{\sigma})}\leq \dfrac{C}{\inn{\ld_n}^\veps} \nor{F(U)-F(\widetilde{U})}{L^{1}([0,T],X^{\sigma+\veps})}.
\end{align}
    
    By hypothesis $U$, $\widetilde{U}\in \B_{R_0}^{[0, T]}(X^\sigma)$ and Assumption \ref{assumF} implies that $F(U)$, $F(\widetilde{U})\in \B_{C}^{[0,T]}(X^{\sigma+\veps})$ with the following Lipschitz estimate 
\begin{align}
\nor{Z}{C^{0}([0,T],X^{\sigma})}\leq \dfrac{C}{\inn{\ld_n}^\veps} \nor{U-\widetilde{U}}{L^{1}([0,T],X^{\sigma})}\leq \dfrac{CT}{\inn{\ld_n}^\veps} \nor{Z}{C^{0}([0,T],X^{\sigma})}.
\end{align}
Adjusting $n$ if necessary so that $\tfrac{CT}{\inn{\ld_n}^\veps}<1$, we conclude that $Z=0$, which is $U= \widetilde{U}$.
\end{proof}

\subsubsection{Propagation of regularity} We now turn our attention to the regularity of the nonlinear system, for which we prove a Propagation of Regularity result under suitable assumptions. This result will later be related to a sort of uniformity for the splitting in the Galerkin procedure. 

Building upon Assumption \ref{assumC}, we make the following assumption, which will allow us to control the semigroup generated by $A$ in a slightly more regular space:
\begin{align}\label{observabstractveps}
    \norm{W_0}_{X^{\sigma+\veps}}^2\leq \mathfrak{C}_{\text{obs}}^2\int_0^T \norm{\bC e^{tA}W_0}_{X^{\sigma+\veps}}^2dt,\quad \forall W_0\in X^{\sigma+\veps}.
\end{align}
In that context Assumption \ref{assumCC} is the conjunction of Assumption \ref{assumC} and  \eqref{observabstractveps}, that is the observability at both levels or regularity $X^{\sigma}$ and $X^{\sigma+\veps}$.

\begin{proposition}\label{prop:abs-prop-reg}
Let $R_{0}>0$, $R_{1}>0$ and $A$, $\bC$ and $F$ satisfying Assumptions \ref{assumAA}, \ref{assumCC} and \ref{assumF}, respectively. Moreover, assume that the pair $(A, \bC)$ satisfies Assumption \ref{assumcommu}. Then, there exists $R_{2}>0$ so that for any 
$U\in \B_{R_{0}}^{[0,T]}(X^{\sigma})$, $H_{1}\in \B_{R_{0}}^{[0,T]}(X^{\sigma})$ and $H_{2}\in \B_{R_{1}}^{[0,T]}(X^{\sigma+\veps})$ that satisfy \eqref{UCabstractT*intro} on $[0,T]$, we have $U\in  \B_{R_{2}}^{[0,T]}(X^{\sigma+\veps})$.
\end{proposition}
\begin{proof}

We know that $(AA^*)^{1/2}: X^1\to X$ admits a unique restriction so that $(AA^*)^{1/2}: X^{1+\sigma}\to X^{\sigma}$ is a linear continuous operator. Let us call such extension $\A_\sigma$. Furthermore:
\begin{itemize}
    \item From Assumption \ref{assumAA}, $\A_\sigma$ has compact resolvent.
    \item $\A_\sigma$ and $(AA^*)^{1/2}$ have the same spectrum, hence the same resolvent set.
\end{itemize}
For simplicity, we keep the same notation $\A$ for the same operator acting on different spaces. Since $\A$ is non-negative, its resolvent set contains $\R_-$ and for $n\in \N^*$, for any $s>0$, we have a well defined smoothing operator $\mc{J}_n=(I+\frac{1}{n}\A^s)^{-1}\in \mc{L}(X^\sigma, X^{\sigma+s})$ with the uniform bound $\norm{\mc{J}_n}_{\mc{L}(X^\sigma, X^{\sigma+s})}\leq n$. We also have the uniform bound $\norm{\mc{J}_n}_{\mc{L}(X^\sigma)}\leq 1$ and the same estimate holds in $\mc{L}(X^{\sigma+\veps})$. Following Proposition \cite[Proposition 2.3.4]{TW:09}, we can see that $\J_n \phi\underset{n\to+\infty}{\longrightarrow} \phi$ for any $\phi\in X^\sigma$.

Let us consider $U^n(t)=\J_nU(t)$ and observe that $U^n$ is uniformly bounded in $X^\sigma$ by some constant $C>0$, which is independent of $n\in \N$.  By Duhamel's formula, let us split the solution into its linear and nonlinear part as follows
    \begin{align*}
        U^n(t)=e^{tA}U_0^n+\int_0^t e^{A(t-s)}\big(\J_nF(U(s)+H_1(s))+\J_nH_2(s)\big)ds:=U_{lin}^n(t)+U_{Nlin}^n(t),
    \end{align*}
where $U_0^n:=\J_n U_0$. Observe that we have used that $e^{tA}$ and $\J_n$ commute. Since $H_2\in \B_{R_{1}}^{[0, T]}(X^{\sigma+\veps})$ and $F$ satisfies Assumption \ref{assumF}, we readily get that $U_{Nlin}^n$ is uniformly bounded in $L^\infty([0, T], X^{\sigma+\veps})$ by some constants depending on $R_{1}$ and the constant $C$ in Assumption \ref{assumF}. The latter property and Assumption \ref{assumCC}, allows us to treat the linear part by employing the observability inequality followed by the triangle inequality
    \begin{align*}
        \norm{U_{lin}^n(t)}_{X^{\sigma+\veps}}^2&\leq \norm{e^{tA}}_{\mc{L}(X^{\sigma+\veps})}^2 \norm{U_0^n}_{X^{\sigma+\veps}}^2\\
        &\leq \mathfrak{C}_{\text{obs}}^2\int_0^T \norm{\bC U_{lin}^n(t)}_{X^{\sigma+\veps}}^2dt\\
        &\leq 2\mathfrak{C}_{\text{obs}}^2\int_0^T \norm{\bC U^n(t)}_{X^{\sigma+\veps}}^2dt+2\mathfrak{C}_{\text{obs}}^2\int_0^T \norm{\bC U_{Nlin}^n(t)}_{X^{\sigma+\veps}}^2dt.
    \end{align*}
  Since $\bC U\equiv 0$ on $[0, T]\times \M$, we get that $\bC U^n=\bC\mc{J}_n U=\mc{J}_n\bC U+[\mc{J}_n,\bC]U=[\mc{J}_n,\bC]U$. Moreover, we have $[\mc{J}_n,\bC]=\frac{1}{n}\mc{J}_n[\A,\bC]\mc{J}_n$. We have seen that $\norm{\mc{J}_n}_{\mc{L}(X^\sigma)}\leq 1$ and $\norm{\frac{1}{n}\mc{J}_n}_{\mc{L}(X^{\sigma}, X^{\sigma+s})}\leq 1$, so we get, uniformly in $n$
    \begin{align*}
      \norm{\bC U^n}_{L^\infty([0, T], X^{\sigma+\veps})}= \norm{\mc{J}_n[\A^s,\bC]\frac{\mc{J}_n}{n} U}_{L^\infty([0, T], X^{\sigma+\veps})}\leq R_0\norm{[\A^s,\bC]}_{\mc{L}(X^{\sigma+s}, X^{\sigma+\veps})}.
    \end{align*}
     Therefore, in view of Assumption \ref{assumcommu}, the term $\norm{[\A^s,\bC]}_{\mc{L}(X^{\sigma+s}, X^{\sigma+\veps})}$ is bounded and so $U_{lin}^n$ is uniformly bounded in $L^\infty([0, T], X^{\sigma+\veps})$. It follows that $U^n$ is uniformly bounded in $L^\infty([0, T], X^{\sigma+\veps})$ by some constant $C>0$. Moreover, due to the fact that $\norm{U-U^n}_{L^\infty([0, T], X^\sigma)}\to 0$ and leveraging that $X^\sigma$ is a Hilbert space, we get that $U(t)\in X^{\sigma+\veps}$ and 
    \begin{align*}
        \norm{U(t)}_{X^{\sigma+\veps}}\leq \liminf \norm{U^n(t)}_{X^{\sigma+\veps}}\leq C,
    \end{align*}
    for any $t\in [0, T]$, showing that $U$ is uniformly bounded in $L^\infty([0, T], X^{\sigma+\veps})$.
\end{proof}
\begin{remark}\label{rk:abs-prop-reg}
    Looking at the above proof, in regards to the nonlinearity, we only need to ensure that $t\mapsto \int_0^t e^{A(t-s)}F(Z(s))ds$ defines a bounded map in $L^\infty([0, T], X^{\sigma+\veps})$. For instance, this was achieved here through the sole hypothesis that $F(Z)$ is bounded in $L^\infty([0, T], X^{\sigma+\veps})$ for $Z$ bounded in $L^\infty([0, T], X^\sigma)$.
\end{remark}

\subsection{An abstract frequency-based reconstruction operator} The objective of this section is to prove that it is possible to reconstruct the high-frequency component for the solutions of our nonlinear system, with the low-frequency component as an input. Furthermore, this reconstruction can be made in a holomorphic way.

Building upon Assumption \ref{assumF}, we further assume that there exists $\delta>0$ so that $F$ has a holomorphic extension from $\mathbb{B}_{4R_{0},2\delta}(X^{\sigma})$ to $X^{\sigma+\veps}_{\mathbb{C}}$ for some $R_{0}>0$, $\sigma\geq 0$ and $\delta>0$. Moreover, there exists $C>0$ so that 
\begin{align}
\label{boundFanalytic}
\nor{F(U_{0})}{X^{\sigma+\veps}_{\mathbb{C}}}\leq C,\quad \nor{F(U_{0})-F(V_{0})}{X^{\sigma+\veps}_{\mathbb{C}}}\leq C \nor{U_{0}-V_{0}}{X^{\sigma}_{\mathbb{C}}}
\end{align}
for any $U_{0},~V_{0}\in \mathbb{B}_{4R_{0},2\delta}(X^{\sigma})$. In this context, Assumption \ref{assumFholom} is the conjunction of Assumption \ref{assumF} and the previous assumption.

The main result of this section is the following.

\begin{theorem}
\label{thmabstract}Let $R_{0}$ and $R_{1}>0$. Under Assumptions \ref{assumAA}, \ref{assumCC},\ref{assumF} (with $R_{0}$) and \ref{assumcommu}, there exists $n\in \N$ and a nonlinear Lipschitz (reconstruction) operator $\mathcal{R}$ 
\begin{align}
\label{Rspace}
\mathcal{R}:  \quad \B_{R_{0}}^{[0,T]}(\P_n X^{\sigma}) \times  \B_{R_{0}}^{[0,T]}(X^{\sigma}) \times  \B_{R_{1}}^{[0,T]}(X^{\sigma+\veps})  \longrightarrow \B_{R_{0}}^{[0,T]}(\Q_n X^{\sigma})
\end{align}
 so that for any $U\in \B_{R_{0}}^{[0,T]}(X^{\sigma})$, $H_{1}\in \B_{R_{0}}^{[0,T]}(X^{\sigma})$ and $H_{2}\in \B_{R_{1}}^{[0,T]}(X^{\sigma+\veps})$ that satisfy
\begin{align}\label{UCabstract}
    \left\{\begin{array}{cl}
         \partial_t U=AU+F(U+H_{1})+H_{2}, &\textnormal{ on } [0,T], \\
         \bC U(t)=0, &\textnormal{ for }t\in [0,T],
    \end{array}\right.
\end{align}
then, $\Q_n  U=\mathcal{R}(\P_{n}U,H_{1},H_{2})$.

Moreover, if additionally, $F$ satisfies Assumption \ref{assumFholom} then, there exist $\eta,\eta_1>0$ so that for any $U\in \B_{R_{0}}^{[0,T]}(X^{\sigma})$ that satisfies \eqref{UCabstract}, then $\mathcal{R}$ extends holomorphically as 
\begin{align*}
\mathcal{R}: \left( \P_nU+ \B_{\eta,\eta}^{[0,T]}(\P_n X^{\sigma})\right) \times  \B_{R_{0},\eta}^{[0,T]}(X^{\sigma}) \times  \B_{R_{1},R_{1}}^{[0,T]}(X^{\sigma+\veps})  \longrightarrow \B_{R_{0},\eta_1}^{[0,T]}(\Q_n X^{\sigma}).
\end{align*}
\end{theorem}

The main idea is to consider the splitting
\begin{align}\label{eq:freq-split}
    U=\P_n U+\Q_n U:=V+W
\end{align}
so the problem \eqref{UCabstract} translates into the following nonlinear observability system
\begin{align}\label{eq:freq-split-ode}
    \left\{\begin{array}{rl}
         \partial_t V(t)&=AV(t)+\P_n F( V+W+H_{1})+\P_{n}H_{2},  \\
         \partial_t W(t)&=AW(t)+\Q_n F(V+W+H_{1})+\Q_{n}H_{2}, \\
         \bC V(t)&=-\bC W(t).
    \end{array}\right.
\end{align}
For a given bounded function $V\in C^0([0, T], \P_n X^\sigma)$, we are interested in solving the nonlinear observability problem
\begin{align}\label{eq:high-freq-nlw-obs}
    \left\{\begin{array}{rl}
        \partial_t W(t)&=AW(t)+\Q_n F(V+W+H_{1})+\Q_{n}H_{2},   \\
   \Pi_n\bC W    &=- \Pi_n\bC V.  
    \end{array}\right.
\end{align}
Note that the operator $\Pi_{n}$ is nonlocal in time in $[0,T]$.
\begin{proposition}\label{propabstract} 
Let $R_{0}$ and $R_{1}>0$. Under Assumptions \ref{assumAA}, \ref{assumC} and \ref{assumF} (with $R_{0}$), there exists $n_{0}\in \N$ and $\eta>0$ so that for any $n\geq n_{0}$, for any $H_{1}\in  \B_{5R_{0}/2}^{[0,T]}(X^{\sigma})$, $H_{2}\in  \B_{R_{1}}^{[0,T]}(X^{\sigma+\veps})$ and $G\in \mathbb{B}_{\eta}(L^2([0, T], X^{\sigma}))$, there exists a unique solution $W\in \B_{R_{0}}^{[0,T]}(\Q_n X^{\sigma})$ to 
\begin{align}\label{eqWG}
    \left\{\begin{array}{rl}
        \partial_t W&=AW+\Q_n F(W+H_{1})+\Q_{n}H_{2}, \\
        \Pi_n\bC W(t)    &=\Pi_nG. 
    \end{array}\right.
\end{align}
This defines a nonlinear Lipschitz operator $\mathcal{\widetilde{R}}$ 
\begin{align}
    \left\{\begin{array}{rcl}
 \B_{5R_{0}/2}^{[0,T]}(X^{\sigma}) \times  \B_{R_{1}}^{[0,T]}(X^{\sigma+\veps}) \times  \mathbb{B}_{\eta}(L^{2}([0,T],X^{\sigma})) & \longrightarrow& C^0([0, T], \Q_{n }X^{\sigma})\\
    (H_{1},H_{2},G)&\longmapsto& W:=\mathcal{\widetilde{R}}(H_{1},H_{2},G).
    \end{array}\right.
\end{align}
Moreover, if additionally, $F$ satisfies Assumption \ref{assumFholom}, then, there exists $\eta>0$ so that $\mathcal{\widetilde{R}}$ extends holomorphically in 
\begin{align*}
\mathcal{\widetilde{R}}:  \B_{5R_{0}/2,\eta}^{[0,T]}(X^{\sigma})  \times \B_{R_{1}, R_{1}}^{[0,T]}(X^{\sigma+\veps})\times  \mathbb{B}_{\eta,\eta}(L^{2}([0,T],X^{\sigma}))  \longrightarrow C^0([0, T], \Q_{n }X^{\sigma}_{\mathbb{C}}).
\end{align*}
\end{proposition}
\begin{proof}[Proof of Theorem \ref{thmabstract} from Proposition \ref{propabstract}]Let $\chi\in C^{\infty}_{0}(\R,[0,1])$ supported in $[-1,1]$ so that $\chi(s)=1$ for $s\in [-1/2,1/2]$. We define the operator $\mathcal{R}$ by the formula
 \begin{align}\label{defR}
 \mathcal{R}(V,H_{1},H_{2})=\mathcal{\widetilde{R}}(V+H_{1},H_{2},- \chi(\eta^{-1} \nor{\bC V}{L^{2}([0,T],X^{\sigma})}) \bC V).
 \end{align}
where $\eta$ is the one in Proposition \ref{propabstract}.
 If $V$ is complex-valued, we mean $ \chi(\eta^{-1} \nor{\bC V}{L^{2}([0,T],X^{\sigma}_{\mathbb{C}})})$.
 
  We prove that it is well defined and satisfies the requirements of Theorem \ref{thmabstract}. 

We first notice that $\chi(\eta^{-1} \nor{\bC V}{L^{2}([0,T],X^{\sigma})}) =0$ if $\nor{ \bC V}{L^{2}([0,T],X^{\sigma})}\geq \eta$, so that we always have $\nor{ \chi(\eta^{-1} \nor{ \bC V}{L^{2}([0,T],X^{\sigma})}) \bC V}{L^{2}([0,T],X^{\sigma})}\leq \eta$. In particular, $\mathcal{R}$ is well defined from the spaces precised in \eqref{Rspace}. 
 
 Now, we check that it satisfies the required properties. Let $U$ be as in the Theorem and a solution of \eqref{UCabstract}. We want to check that for $n$ large enough (only depending on the parameters),  we can impose $\Q_n U$ to be small in $C^{0}([0,T],X^{\sigma})$ and $\bC \P_n U$ to be small in $X^{\sigma}$. We can apply Proposition \ref{prop:abs-prop-reg} to $U$ to propagate regularity, resulting in $\norm{U}_{C^0([0, T], X^{\sigma+\veps})}\leq R_{2}$ for some $R_{2}>0$. We have
 \begin{align}
 \label{unifboundU}
\nor{ \Q_n U}{C^{0}([0,T],X^{\sigma})} &\leq \frac{1}{\inn{\lambda_n}^{\veps}} \nor{ \Q_n U}{C^{0}([0,T],X^{\sigma+\veps})}\leq \frac{R_{2}}{\inn{\lambda_n}^{\veps}} .
 \end{align}
  For such a solution satisfying \eqref{UCabstract}, we have $\bC U=0$ and in particular, $\bC \Q_n U=-\bC \P_n U$, so that 
 \begin{align*}
  \nor{\bC \P_n U}{L^{2}([0,T],X^{\sigma})}&= \nor{\bC \Q_n U}{L^{2}([0,T],X^{\sigma})}\leq T^{1/2}\nor{\bC}{\mc{L}(X^\sigma)}\nor{ \Q_n U}{C^{0}([0,T],X^{\sigma})}\\
  &\leq \frac{T^{1/2}\nor{\bC}{\mc{L}(X^\sigma)}R_{2}}{\inn{\lambda_n}^{\veps}} .
 \end{align*}
  Since $\lambda_n\to +\infty$ as $n\to +\infty$, we can find $n\geq n_0$ so that it can be made smaller that $\eta/4$. In particular, defining $W:= \mathcal{R}(\P_{n}U,H_{1},H_{2})$, we have by definition of $\mathcal{R}$,
 \begin{align*}
W= \mathcal{\widetilde{R}}(\P_{n}U+H_{1},H_{2},- \bC \P_{n}U)
  \end{align*}
  is the unique solution in $\B_{R_{0}}^{[0,T]}(\Q_n X^{\sigma})$ of
  \begin{align*}
    \left\{\begin{array}{rl}
        \partial_t W&=AW+\Q_n F(W+\P_{n}U+H_{1})+\Q_{n}H_{2}   \\
   \Pi_n\bC W    &=-\Pi_n \bC \P_{n}U. 
    \end{array}\right.
\end{align*}
Also, since $\bC \P_{n}U=- \bC \Q_{n}U $ we notice that $\Q_n U$ is solution of 
 \begin{align*}
    \left\{\begin{array}{rl}
        \partial_t \Q_n U&=A\Q_n U+\Q_n F(\Q_n U+\P_n U+H_{1})+\Q_{n}H_{2}   \\
   \Pi_n\bC \Q_n U    &=-\Pi_n \bC \P_{n}U. 
    \end{array}\right.
\end{align*}
In particular, since $  \nor{\bC \P_n U}{L^{2}([0,T],X^{\sigma})}\leq \eta/4$ and by the various bounds, we have $\P_n U+H_{1}\in  \B_{2R_{0}}^{[0,T]}(X^{\sigma})$, $H_{2}\in  \B_{R_{1}}^{[0,T]}(X^{\sigma+\veps})$ and $-\bC \P_{n}U\in \mathbb{B}_{\eta}(L^2([0, T], X^{\sigma}))$. By definition of $\widetilde{R}$, this implies $\Q_n U= \widetilde{R}(\P_n U+H_{1},H_{2},- \bC \P_{n}U)$. So we conclude $\Q_{n}U=W=\mathcal{R}(\P_{n}U,H_{1},H_{2})$ as expected.

Concerning the holomorphic extension, since $\mathcal{\widetilde{R}}$ has a holomorphic extension, as proved in Proposition \ref{propabstract}, using formula \eqref{defR} and composition of function, it is enough to prove that $V\mapsto - \chi(\eta^{-1} \nor{\bC V}{L^{2}([0,T],X^{\sigma}_{\mathbb{C}})})$ is constant equal to $1$ and therefore holomorphic in some neighborhood $\P_nU+ \B_{\eta_{1},\eta_{1}}^{[0,T]}(\P_n X^{\sigma})$ that would be included in $\B_{3R_0/2,\eta/2}^{[0,T]}(X^{\sigma})$. This can be obtained for $\eta_{1}$ small enough since $  \nor{\bC \P_n U}{L^{2}([0,T],X^{\sigma})}\leq \eta/4$. This gives the result up to renaming $\eta_{1}$ by $\eta$.
\end{proof}

\begin{proof}[Proof of Proposition \ref{propabstract}]
Using Lemma \ref{lmCauchyobs} (recall that $\FL$ denotes a linear flow map), we are looking for $W$ solution of 
\begin{align*}
    W=\FL (G,F(W+H_{1})+H_{2}).
\end{align*}
So, it is natural to define the nonlinear operator $\Phi_{n,G,H_{1},H_{2}}$ defined by
\begin{align}\label{defPhi}
\Phi_{n,G,H_{1},H_{2}}(W):=\FL (G,F(W+H_{1})+H_{2}) 
\end{align}
 that we will prove later to be well defined from suitable balls of $C^0([0, T], \Q_n X^\sigma)$ to $C^0([0, T], \Q_n X^\sigma)$. To keep notations reasonable, we will write $\Phi=\Phi_{n,V,G,H_{1},H_{2}}$ keeping in mind all the dependence. The goal will be to find a fixed point of the operator $\Phi$ in a small ball that will also satisfy \eqref{eqWG}. 
 
Following Hale-Raugel \cite[Theorem 2.14]{HR03} and Joly-Laurent \cite[Theorem 10.1]{2020:joly-laurent:decay-nlw-no-gcc}, we divide the proof in three steps.\\

\noindent\emph{Step 1. High-frequency fixed point: real case.} For this part of the proof, we consider all the functions involved to be real-valued, in the sense that we consider the real vector space $X^{\sigma}$.

We fix $H_{1}\in  \B_{5R_{0}/2}^{[0,T]}(X^{\sigma})$, $H_{2}\in  \B_{R_{0}}^{[0,T]}(X^{\sigma+\veps})$ and $G\in \mathbb{B}_{\eta}(L^2([0, T], X^{\sigma}))$, but we will make estimates uniform when these functions are in these sets.

We want to find a fixed point in $W\in \B_{R_{0}}^{[0,T]}(\Q_n X^{\sigma})$ by the fixed point theorem of Banach. We prove that $\Phi$ is a contraction in this set.

Let $W\in \B_{R_{0}}^{[0,T]}(\Q_n X^{\sigma})$. Estimate \eqref{estimFL} of Lemma \ref{lmCauchyobs} gives for some constants uniform in $n\in \N$
\begin{align*}
\nor{\Phi(W)}{C^{0}([0,T],\Q_{n}X^{\sigma})}&=\nor{\FL(G,F(W+H_{1})+H_{2})}{C^{0}([0,T],\Q_{n}X^{\sigma})}\\
&\leq C \nor{\Pi_{n}G}{L^{2}([0,T],X^{\sigma})}+C\nor{\Q_{n}\left[ F(W+H_{1})+H_{2}\right]}{L^{1}([0,T],X^{\sigma})}.
\end{align*}
Since $\Pi_{n}$ is a projection on $L^2([0,T],X^{\sigma})$, we have $\nor{\Pi_{n}G}{L^{2}([0,T],X^{\sigma})}\leq \nor{G}{L^{2}([0,T],X^{\sigma})}\leq \eta$ by assumption. So, it remains to estimate the second term. Since $\nor{W+H_{1}}{C^{0}([0,T],X^{\sigma})}\leq 3R_{0}$, we can use \eqref{boundF} and get
\begin{align}
\label{estimFinPhi}
\nor{\Q_{n}\left[ F(W+H_{1})+H_{2}\right]}{L^{1}([0,T],X^{\sigma})}&\leq \frac{1}{\inn{\lambda_n}^{\veps}} \nor{F(W+H_{1})+H_{2}}{L^{1}([0,T],X^{\sigma+\veps})}\\
&\leq  \frac{T}{\inn{\lambda_n}^{\veps}} \left(C+R_{1}\right).
\end{align}
Summing up the previous estimates gives
\begin{align*}
\nor{\Phi(W)}{C^{0}([0,T],\Q_{n}X^{\sigma})}\leq C\eta+ \frac{1}{\inn{\lambda_n}^{\veps}} \left(C+R_{1}\right).
\end{align*}
Concerning the difference, similar estimates give for $W, W'\in \B_{R_{0}}^{[0,T]}(\Q_n X^{\sigma})$, using instead the second estimate of \eqref{boundF}
\begin{align}
\nonumber\nor{\Phi(W)-\Phi(W')}{C^{0}([0,T],\Q_{n}X^{\sigma})}&=\nor{\FL(0,F(W+H_{1})-F(W'+H_{1}))}{C^{0}([0,T],\Q_{n}X^{\sigma})}\\
\nonumber&\leq C\nor{\Q_{n}\left[ F(W+H_{1})-F(W'+H_{1}))\right]}{L^{1}([0,T],X^{\sigma})}\\
\nonumber&\leq  \frac{C}{\inn{\lambda_n}^{\veps}} \nor{F(W+H_{1})-F(W'+H_{1}))}{L^{1}([0,T],X^{\sigma+\veps})}\\
\label{bounddiffW}&\leq  \frac{CCT}{\inn{\lambda_n}^{\veps}} \nor{W-W'}{C^{0}([0,T],X^{\sigma})}.
\end{align}
In particular, if $\eta$ is chosen small enough and $n_{0}$ is large enough, then $\Phi$ reproduces $ \B_{R_{0}}^{[0,T]}(\Q_n X^{\sigma})$ and is contracting. 

\bigskip

\noindent\emph{Step 2. Extension to the complex case.} The proof is very similar to the real case. We define $\Phi$ with the same formula as \eqref{defPhi} but for 
\begin{align*}
W\in \B_{R_{0},\eta_{1}}^{[0,T]}(\Q_n X^{\sigma}), H_{1}\in  \B_{5R_{0}/2,\eta}^{[0,T]}(X^{\sigma}) , H_{2}\in   \B_{R_{1}, R_{1}}^{[0,T]}(X^{\sigma+\veps}), G\in   \mathbb{B}_{\eta,\eta}(L^{2}([0,T],X^{\sigma}))
\end{align*}
where $\eta_{1}>0$ small is to be fixed later on.

The flow map $\FL$ is linear, so the extension to complex-valued vectors is clear, see \cite[Theorem 3]{BS:71-polynomials}. So, we only need to check that the valuation $F(W+H_{1})$ makes sense and the contraction in $\B_{R_{0},\eta_{1}}^{[0,T]}(\Q_n X^{\sigma})$ of $\Phi$ is still true with these parameters, up to making $\eta$ small and $n$ large. 

If $W\in \B_{R_{0},\eta_{1}}^{[0,T]}(\Q_n X^{\sigma}), H_{1}\in  \B_{5R_{0}/2,\eta}^{[0,T]}(X^{\sigma})$, then $W+H_{1}\in \B_{7R_{0}/2,\eta_{1}+\eta}^{[0,T]}(X^{\sigma})\subset \B_{4R_{0},\delta}^{[0,T]}(X^{\sigma})$ if we have $\eta_{1}+\eta\leq \delta$, so that  $F(W+H_{1})$ is well defined and satisfies similar estimates as in the real case, thanks to \eqref{boundFanalytic} in Assumption \ref{assumFholom}. It remains to check $\Phi(W)\in \B_{R_{0},\eta}^{[0,T]}(\Q_n X^{\sigma})$. We still have
\begin{align*}
 \nor{\Pi_{n}G}{L^{2}([0,T],X_\mathbb{C}^{\sigma})}\leq 2\eta,
\end{align*}
when $G\in \mathbb{B}_{\eta,\eta}(L^{2}([0,T],X^{\sigma}))$, while a similar estimate as in \eqref{estimFinPhi} gives 
\begin{align*}
\nor{\Q_{n}\left[ F(W+H_{1})+H_{2}\right]}{L^{1}([0,T],X_\mathbb{C}^{\sigma})}&\leq  \frac{1}{\inn{\lambda_n}^{\veps}} \left(C+2R_{1}\right).
\end{align*}
Finally, the linear estimate \eqref{estimFL} of Lemma \ref{lmCauchyobs} still holds in the complex-valued case, so, we get 
\begin{align*}
\nor{\Phi(W)}{C^{0}([0,T],\Q_{n}X_\mathbb{C}^{\sigma})}\leq 2\eta+ \frac{1}{\inn{\lambda_n}^{\veps}} \left(C+2R_{1}\right).
\end{align*}
The bound \eqref{bounddiffW} holds in the complex case without modification. So, we obtained that if $\eta_{1}+\eta\leq \delta$, $4\eta\leq \eta_{1}\leq R_{0}$ and $n$ is large enough, $\Phi$ is a contraction on $\B_{R_{0},\eta_{1}}^{[0,T]}(\Q_n X^{\sigma})$. This gives a unique fixed point in the complex case. 

\bigskip

\noindent\emph{Step 3. Regularity of the fixed point.} We prove that for fixed $n$, the map $\widehat{\Phi}: (W, G, H_1, H_2)\mapsto \Phi_{n, G, H_1, H_2}(W)$ is Lispchitz of its arguments under Assumption \ref{assumF} and holomorphic when Assumption \ref{assumFholom} is made.
Recalling Definition \eqref{defPhi} and that $\FL$ is linear in its arguments, we see that, by composition, it will be enough to verify that $(W,H_1)\mapsto F(W+H_{1})$ is Lipschitz or holomorphic, depending on the assumption made.

First of all, observe that the same argument to get estimates \eqref{bounddiffW}, shows that $(W,H_1)\mapsto F(W+H_{1})$ is Lipschitz under Assumption \ref{assumF}, and by linearity of the flow map $\mc{F}_n$, so is the map 
\begin{align*}
    \widehat{\Phi}: \B_{R_0}^{[0, T]}(\Q_n X^{\sigma})\times \mathbb{B}_{\eta}(L^{2}([0,T],X^{\sigma}))\times \B_{2R_{0}}^{[0,T]}(X^{\sigma})  \times  \B_{R_{1}}^{[0,T]}(X^{\sigma+\veps}) \longrightarrow C^0([0, T], \Q_{n }X^{\sigma}).
\end{align*}
Now, under Assumption \ref{assumFholom}, the same argument as before shows that the complex extension, denoted by the same letter $\widehat{\Phi}$,
\begin{align*}
    \left\{\begin{array}{rcl}
    \B_{R_0, \eta}^{[0, T]}(\Q_n X^{\sigma})\times \mathbb{B}_{\eta,\eta}(L^{2}([0,T],X^{\sigma}))\times \B_{2R_{0},\eta}^{[0,T]}(X^{\sigma})  \times  \B_{R_{1}, R_{1}}^{[0,T]}(X^{\sigma+\veps})  & \longrightarrow& C^0([0, T], \Q_{n }X_\mathbb{C}^{\sigma})\\
    (W, G, H_1, H_2)&\longmapsto& \Phi_{n, G, H_1, H_2}(W)
    \end{array}\right.
\end{align*}
is Lipschitz, hence continuous. In view of Theorem \ref{appendix:thm:holom-several-variables} and that $\widehat{\Phi}$ is linear in the variables $G$ and $H_2$ separately, the main task is to show that the map 
\begin{align}
    \left\{\begin{array}{rcl}
    \B_{4R_{0}, \eta}^{[0,T]}(X_{\mathbb{C}}^{\sigma}) \times  \B_{2R_{0}, \eta}^{[0,T]}(X^{\sigma}) & \longrightarrow& C^0([0, T], X_\mathbb{C}^{\sigma+\veps})\\
    (W, H_1 )&\longmapsto& F(W+H_1)
    \end{array}\right.
\end{align}
is holomorphic on each variable separately. Since $F$ is holomorphic when defined from some balls to $X^\sigma$ to $X^{\sigma+\veps}$, the only point to check is that it implies the same result as an application on time-dependent functions in $C^0([0,T],X^\sigma)$. 

Fix $W\in \text{Int}(\B_{R_0, \eta}^{[0, T]}(\Q_nX^\sigma))$ and $H_1\in \text{Int}(\B_{5R_0/2, \eta}^{[0, T]}(X^\sigma))$ and $\ell>0$ so that $B_{C^0([0,T],X_{\C}^\sigma))}(W+H_1,\ell)\subset \B_{4R_0, 2\eta}^{[0, T]}(X^\sigma))$. Let $C$ be the bound on $F$ in Assumption \ref{boundFanalytic}. For all $K_W\in C^0([0, T], \Q_n X_\mathbb{C}^\sigma)$ and $K_H\in C^0([0, T], X_\mathbb{C}^\sigma)$ with $\norm{K_W+K_H}_{C^0([0, T], X_\mathbb{C}^\sigma)}\leq \ell/4$ and $t\in [0,T]$, we have $(W+H_1+K_W+K_H)(t) \in \mathbb{B}_{4R_{0},2\delta}(X^{\sigma})$ (recall $\eta<\delta$). Being $F$ holomorphic, by applying a Taylor expansion (see Theorem \ref{appendix:prop:cauchy-form}) together with Cauchy estimates (see Theorem \ref{appendix:prop:cauchy-est}), we have the bound, uniform in $t\in [0,T]$,
\begin{multline}\label{prop:proof:frechet-1tfixed}
    \norm{F(W(t)+H_1(t)+K_W(t)+K_H(t))-F(W(t)+H_1(t))-\delta F(W(t)+H_1(t); K_W(t)+K_H(t))}_{ X_\mathbb{C}^{\sigma+\veps}}\\
    \leq \dfrac{4C\norm{K_W(t)+K_H(t)}_{X_\mathbb{C}^\sigma}^2}{\ell(\ell-2\norm{K_W(t)+K_H(t)}_{X_\mathbb{C}^\sigma})}   \leq \dfrac{4C\norm{K_W+K_H}_{C^0([0, T], X_\mathbb{C}^\sigma)}^2}{\ell(\ell-2\norm{K_W+K_H}_{C^0([0, T], X_\mathbb{C}^\sigma)})}\leq \dfrac{8C\norm{K_W+K_H}_{C^0([0, T], X_\mathbb{C}^\sigma)}^2}{\ell^2},
\end{multline}
where
\begin{align*}
    \delta F(Z; K):=\lim_{s\to 0}\dfrac{1}{s}(F(Z+sK)-F(Z)).
\end{align*}
The differential of $F$ can be easily extended from $C^0([0, T], X_\mathbb{C}^\sigma)$ into $C^0([0, T], X_\mathbb{C}^{\sigma+\veps})$. Since, the previous estimate is uniform in $t\in [0,T]$, we can write
\begin{multline}\label{prop:proof:frechet-1}
    \norm{F(W+H_1+K_W+K_H)-F(W+H_1)-\delta F(W+H_1; K_W+K_H)}_{C^0([0, T], X_\mathbb{C}^{\sigma+\veps})}\\
\leq \dfrac{8C\norm{K_W+K_H}_{C^0([0, T], X_\mathbb{C}^\sigma)}^2}{\ell^2}.
\end{multline}

Actually, the previous estimate \eqref{prop:proof:frechet-1} holds uniformly in a ball around $(W, H_1)$ of radius $\ell/4$. Consequently, we establish that the map
\begin{align}
    \left\{\begin{array}{rcl}
    \B_{R_{0}, \eta}^{[0,T]}(X^{\sigma}) \times  \B_{5/2R_{0}, \eta}^{[0,T]}(X^{\sigma}) & \longrightarrow& C^0([0, T], X_\mathbb{C}^{\sigma+\veps})\\
    (W, H_1 )&\longmapsto& F(W+H_1)
    \end{array}\right.
\end{align}
is Fr\'echet differentiable (in the complex sense), hence holomorphic as an application of Theorem \ref{appendix:thm:equiv-holom}. The estimates in Lemma \ref{lmCauchyobs} show that the application $\FL$ is bilinear continuous with respect to each of its parameters and therefore has an obvious holomorphic extension. Since the embedding $\iota: C^0([0,T],X_{\mathbb{C}}^{\sigma+\veps})\to L^1([0,T],X_{\mathbb{C}}^{\sigma})$ is linear, an application of the chain rule (see Theorem \ref{appendix:thm:chain-rule}) shows that the map
\begin{align}
    \left\{\begin{array}{rcl}
\B_{R_{0}/2}^{[0,T]}(X^{\sigma}) \times  \mathbb{B}_{\eta}(L^{2}([0,T],X^{\sigma})) \times  \B_{5R_{0}/2}^{[0,T]}(X^{\sigma}) \times  \B_{R_{1}}^{[0,T]}(X^{\sigma+\veps})& \longrightarrow& C^0([0, T], \Q_{n }X^{\sigma})\\
  (W,G,H_1,H_2)&\longmapsto& \Phi_{n,G,H_{1},H_{2}}(W)
    \end{array}\right.
\end{align}
is holomorphic since it is defined by $\Phi_{n,G,H_{1},H_{2}}(W):=\FL (G,\iota\left(F(W+H_{1})+H_{2}\right)) $. The conclusion that the nonlinear reconstruction operator $\mathcal{\widetilde{R}}$ is holomorphic follows as a direct application of Theorem \ref{app:thm:uniffixedpoint} of regularity of fixed points with respect to parameters.
\end{proof}

\subsection{Analyticity in time of the observed solution} Here we prove the main theorem of this section, that is the abstract Theorem \ref{thmabstractanalyticintro} of analytic regularity stated in the introduction.
\begin{proof}[Proof of Theorem \ref{thmabstractanalyticintro}]
First, since $y\mapsto \max_{t\in [0,T^{*}]}\nor{\Im H_{1}(t+iy)}{X^{\sigma}}$ is a continuous function on $[-\mu,\mu]$ that is equal to zero at $y=0$, there exists $0<\mu'<\mu$ so that $\max_{t\in [0,T]}\left|\Im H_{1}(t+iy)\right| \leq \eta$ for $y\in [-\mu',\mu']$, where $\eta$ is given by Theorem \ref{thmabstract}. We can do the same for $H_{2}$. We denote $R_{1}=\max_{z\in [0,T^{*}]+i[-\mu,\mu]}\nor{H_{2}(z)}{X^{\sigma+\veps}}$. 

Let $n\in \N$, $\eta>0$ and $\mathcal{R}$ be given by Theorem \ref{thmabstract}, with the chosen $R_0$ and $R_{1}$. Let $0<\nu<T^*-T$. 

With these two choices, if we define for $s\in [-\nu, T^*-T-\nu]$, $H_{1}^{s}$ as the application $t\mapsto H_{1}(t+\nu+s)$, we have $H_{1}^{s} \in \B_{R_{0},\eta}^{[0,T]}(X^{\sigma})$ for any $s\in [-\nu, T^*-T-\nu]$. Moreover, the application $s\mapsto H_1^s$ is continuous. Indeed, since $t\in [0, T^*]\mapsto H_1(t)\in X^\sigma$ is a continuous map defined on a compact set, it is uniformly continuous. Hence, for every $\veps>0$ and any $t\in [0, T]$, there exists $\delta>0$ so that for any $s$, $s_0\in [-\nu, T^*-T-\nu]$ with $|(s+t)-(s_0+t)|\leq \delta$, then $\norm{H_1(t+\nu+s)-H_1(t+\nu+s_0)}_{X^\sigma}\leq \veps$. Since the later property does not depend on $t$, we can take $\sup_{t\in [0, T]}$ in the last inequality to obtain that $\norm{H_1^s-H_1^{s_0}}_{C^0([0, T], X^\sigma)}\leq \veps$. We do the same for $H_2$, so we have $H_2^s\in \B_{R_{1},R_{1}}^{[0,T]}(X^{\sigma+\veps})$ for any $s\in [-\nu, T^*-T-\nu]$ and moreover $s\mapsto H_2^s$ is continuous.

For $U$, we can define similarly $U^s\in \B_{R_{0},\eta}^{[0,T]}(X^{\sigma})$ for any $s\in [-\nu, T^*-T-\nu]$. We can also decompose $U=\P_n U+\Q_n U=:V+W$ and $U^s=V^s+W^s$. For any fixed $s\in [-\nu, T^*-T-\nu]$, the assumption \eqref{UCabstractT*intro} implies 
\begin{align}\label{UCabstractwiths}
    \left\{\begin{array}{cr}
         \partial_t U^s=AU^s+F(U^s+H_{1}^s)+H_{2}^s& \textnormal{ on } [0,T],\\
         \bC U^s(t)=0 &\textnormal{ for }t\in [0,T].
    \end{array}\right.
\end{align}
In particular, for any fixed $s\in [-\nu, T^*-T-\nu]$, Theorem \ref{thmabstract} gives 
\begin{align}
\label{equalWs}
W^s=\Q_n  U^s=\mathcal{R}(\P_{n}U^s,H_{1}^s,H_{2}^s)=\mathcal{R}(V^s,H_{1}^s,H_{2}^s),
\end{align} the equality being meant in $C^0([0,T], \Q_n X^{\sigma})$. 

But if we denote $G=\P_n F(U+H_{1})+\P_n H_{2}=\P_n F(V+W+H_{1})+\P_n H_{2}\in C^0([0,T^*],\P_n X^{\sigma}) $, then $V\in C^0([0,T^*],\P_n X^{\sigma})$ is solution of 
\begin{align*}
    \partial_t V=\P_n AV+G& \textnormal{ on } [0,T^{*}].
\end{align*}
With the related notations, Lemma \ref{lmtranslat} implies 
\begin{align}
\label{eqVG}
    \partial_s V^s=\P_n AV^s+G^s& \textnormal{ on } [-\nu, T^*-T-\nu].
\end{align}
But for any fixed $s\in [-\nu, T^*-T-\nu]$, we have, using \eqref{equalWs},
\begin{align*}
G^s=\P_n F(V^s+W^s+H_{1}^s)+\P_n H_{2}^s=\P_n F(V^s+\mathcal{R}(V^s,H_{1}^s,H_{2}^s)+H_{1}^s)+\P_n H_{2}^s
\end{align*}
with equality in $C^0([0,T], \Q_n X^{\sigma})$. 

It is then natural to define the application
\begin{align}
    \left\{\begin{array}{rcl}
    \Tilde{F}: [-\nu, T^*-T-\nu]\times \B_{R_{0}}^{[0,T]}(\P_n X^{\sigma})&\to& C^0([0, T], \P_n X^\sigma)\\
    (s,\Tilde{V})&\mapsto& \P_n F(\Tilde{V}+\mathcal{R}(\Tilde{V},H_{1}^s,H_{2}^s)+H_{1}^s)+\P_n H_{2}^s.
    \end{array}\right.
\end{align}
This application is well-defined, continuous and locally Lipschitz (with respect to the second variable). Observe that $H_1^s\in \B_{R_0}^{[0, T]}(X^\sigma)$, $H_2^s\in \B_{R_{1}}^{[0, T]}(X^{\sigma+\veps})$ for all $s\in [-\nu, T^*-T-\nu]$ and $\Tilde{V}\in \B_{R_{0}}^{[0,T]}(\P_n X^{\sigma})$ implies, by construction, that $\Tilde{V}+\mathcal{R}(\Tilde{V},H_{1}^s,H_{2}^s)+H_{1}^s\in \B_{3R_0}^{[0, T]}(X^\sigma)$. Therefore $F(\Tilde{V}+\mathcal{R}(\Tilde{V},H_{1}^s,H_{2}^s)+H_{1}^s)$ is well defined and so is $\Tilde{F}$. Recall that $\mc{R}$ is Lipschitz on its variables, that is, there exists $L>0$ such that
\begin{multline*}
    \norm{\mc{R}(\Tilde{V}, H_{1}^{s}, H_2^{s})-\mc{R}(\Tilde{V}^*, H_{1}^{s^*}, H_2^{s^*})}_{C^0([0, T], X^\sigma)}\leq L\big( \norm{\Tilde{V}-\Tilde{V}^*}_{C^0([0, T], X^\sigma)}\\+\norm{H_1^{s}-H_1^{s^*}}_{C^0([0, T], X^\sigma)}+\norm{H_2^{s}-H_2^{s^*}}_{C^0([0, T], X^{\sigma+\veps})} \big),
\end{multline*}
for all $\Tilde{V}$, $\Tilde{V}^*\in \B_{R_{0}}^{[0,T]}(\P_n X^{\sigma})$ and $s$, $s^*\in [-\nu, T^*-T-\nu]$. The continuity of $\Tilde{F}$ then follows by the continuity of $s\mapsto (H_1^s, H_2^s)$, Assumption \ref{assumFholom} and by algebra of continuous maps. The same inequality along with Assumption \ref{assumFholom} shows that
\begin{align*}
    \norm{\Tilde{F}(s, \Tilde{V})-\Tilde{F}(s, \Tilde{V}^*)}_{C^0([0, T], \P_n X^\sigma)}\leq C\norm{\Tilde{V}-\Tilde{V}^0}_{C^0([0, T], \P_n X^\sigma)},
\end{align*}
for any $\Tilde{V}$, $\Tilde{V}^*\in \B_{R_{0}}^{[0,T]}(\P_n X^{\sigma})$ and $s\in [-\nu, T^*-T-\nu]$.

Moreover, we claim that for any $s_0\in (-\nu,T^*-T-\nu)$, there exists $\rho>0$ so that it admits a holomorphic extension (in both variables)
\begin{align*}
    \left\{\begin{array}{rcl}
    \Tilde{F}: B_{\mathbb{C}}(s_0, \rho)\times \left[V^{s_0}+\B_{\eta,\eta}^{[0,T]}(\P_n X^{\sigma})\right]&\to& C^0([0, T], \P_n X_\mathbb{C}^\sigma)\\
    (z,\Tilde{V})&\mapsto& \P_n F(\Tilde{V}+\mathcal{R}(\Tilde{V},H_{1}^z,H_{2}^z)+H_{1}^z)+\P_n H_{2}^z.
    \end{array}\right.
\end{align*}
By Theorem \ref{thmabstract}, around $V^{s_0}=\P_n U^{s_0}$, the map $\mc{R}$ extends holomorphically as
\begin{align*}
\mathcal{R}: \left[ V^{s_0}+ \B_{\eta,\eta}^{[0,T]}(\P_n X^{\sigma})\right] \times  \B_{R_{0},\eta}^{[0,T]}(X^{\sigma}) \times  \B_{R_{1},R_{1}}^{[0,T]}(X^{\sigma+\veps})  \longrightarrow C^0([0, T], \Q_{n }X_\mathbb{C}^{\sigma}).
\end{align*}
We need to argue that, the application $z\mapsto H_1^z$ is holomorphic from $B_{\mathbb{C}}(s_0, \rho)$ to $C^0([0, T], X_\mathbb{C}^\sigma)$ for $\rho>0$ small enough, and that the same holds for $z\mapsto H_2^z$ from $B_{\mathbb{C}}(s_0, \rho)$ with values in $C^0([0, T], X_\mathbb{C}^{\sigma+\veps})$. Indeed, since
\begin{align*}
    z\in (0, T^*)+i(-\mu', \mu')\mapsto H_1(z)\in X_\mathbb{C}^\sigma
\end{align*}
is holomorphic, for each $s_0\in [-\nu, T^*-T-\nu]$ we can find $\rho>0$ such that, for any $t\in [0, T]$, the application
\begin{align*}
    z\in B_\mathbb{C}(s_0, \rho)\longmapsto H_1(z+t+\nu)\in X_\mathbb{C}^\sigma
\end{align*}
is holomorphic. As we did in the proof of Proposition \ref{propabstract}, by using Cauchy estimates and the bound on $H_1(z)$, for each $t\in [0, T]$ and any $z_0\in B_\mathbb{C}(s_0, \rho)$, we can find $\ell>0$ such that
\begin{align*}
    \norm{H_1(z+h+t+\nu)-H_1(z+t+\nu)-\delta H_1(z+t+\nu, h)}_{X_\mathbb{C}^\sigma}\leq \dfrac{4(R_0+\eta)|h|^2}{\ell(\ell-2|h|)}\leq \dfrac{8(R_0+\eta)|h|^2}{\ell^2} 
\end{align*}
holds uniformly for all $z\in B_\mathbb{C}(z_0, \ell/4)$ and $B_\mathbb{C}(0, \ell/4)$. Moreover, the last estimate is independent of $t\in [0, T]$, so we actually have
\begin{align*}
    \norm{H_1^{z+h}-H_1^z-\delta H_1(z+\cdot+\nu, h)}_{C^0([0, T], X_\mathbb{C}^\sigma)}\leq \dfrac{8(R_0+\eta)|h|^2}{\ell^2},
\end{align*}
uniformly for all $z\in B_\mathbb{C}(z_0, \ell/4)$ and $|h|\leq \ell/4$. The previous estimate shows that the application
\begin{align*}
    z\in B_\mathbb{C}(s_0, \rho)\longmapsto H_1^z\in \B_{R_0, \eta}^{[0, T]}(X^\sigma)
\end{align*}
is holomorphic. The argument works similarly to show that $z\mapsto H_2^z$ is holomorphic.

By composition of holomorphic functions and Theorem \ref{appendix:thm:holom-several-variables} we get that
\begin{align*}
    \left\{\begin{array}{rcl}
    B_{\mathbb{C}}(s_0, \rho)\times \left[V^{s_0}+\B_{\eta,\eta}^{[0,T]}(\P_n X^{\sigma})\right]&\to& C^0([0, T], \Q_n X_\mathbb{C}^\sigma)\\
    (z,\Tilde{V})&\mapsto& \mathcal{R}(\Tilde{V},H_{1}^z,H_{2}^z)
    \end{array}\right.
\end{align*}
is a well-defined and holomorphic map. That the extension $\Tilde{F}$ is holomorphic, follows from Assumption \ref{assumFholom}, algebra of holomorphic maps, and that $\P_n$ is a linear bounded operator.

Since we have proved that for any $s\in [-\nu, T^*-T-\nu]$, $G^s=\Tilde{F}(s,V^s)$, we obtain that the equation \eqref{eqVG} verified by $V^s$ can be written as
\begin{align*}
    \partial_s V^s=\P_n AV^s+\Tilde{F}(s,V^s)& \textnormal{ on } [-\nu, T^*-T-\nu].
\end{align*}
We can then consider the following ODE, defined in the Banach space $C^0([0, T], \P_n X_\mathbb{C}^\sigma)$,
\begin{align}\label{thm:proof:eq:low-freq-evo-nlw}
    \left\{\begin{array}{rl}
        \partial_s \xi(s)&=\P_nA\xi(s)+\Tilde{F}(s, \xi(s))   \\
        \xi(s_0)&=V^{s_0}.
    \end{array}\right.
\end{align}
Notice that $C^0([0,T], \P_n X^\sigma)$ is of infinite dimension, but we can check that $\P_n A$ is a linear bounded operator on it. Lemma \ref{lmDuhamelCauchy} shows that $V^s$, which is a solution of \eqref{thm:proof:eq:low-freq-evo-nlw} in the sense of semi-group, is also a solution in the usual sense of ODE in Banach space.

Notice that $\P_n A+\Tilde{F}$ is holomorphic from $ B_{\mathbb{C}}(s_0, \rho)\times \left[V^{s_0}+ \B_{\eta,\eta}^{[0,T]}(\P_n X^{\sigma})\right]$ into $C^0([0, T], \P_n X_{\mathbb{C}}^\sigma)$. Therefore, by the classical theory of ODEs in Banach spaces \cite[Theorem 10.4.5]{1969:dieudonne:modern-analysis}, we obtain a unique classical solution of \eqref{thm:proof:eq:low-freq-evo-nlw}
\begin{align*}
    \xi:B_{\mathbb{C}}(s_0, \rho')\longmapsto \Tilde{V}(z)\in C^0([0,T], \P_n X_{\mathbb{C}}^\sigma)
\end{align*}
for some $\rho'\in (0, \rho]$. Moreover, such a solution map inherits the regularity of the right-hand side of the ODE and so it is a holomorphic map. By uniqueness of the solutions, we have $\xi(s)=V^s$ for $s\in (s_0-\rho',s_0+\rho')$. 

Moreover, the maps $z\mapsto \xi(z)$ and 
\begin{align*}
    \left\{\begin{array}{rcl}
    B_{\mathbb{C}}(s_0, \rho)\times \left[V^{s_0}+ \B_{\eta,\eta}^{[0,T]}(\P_n X^{\sigma})\right]&\to& C^0([0, T], \P_n X_{\mathbb{C}}^\sigma)\\
    (z,\Tilde{V})&\mapsto&\Tilde{V}+ \mathcal{R}(\Tilde{V},H_{1}^z,H_{2}^z)
    \end{array}\right.
\end{align*}
are both holomorphic in their respective spaces, and so is its composition $z\mapsto \xi(z)+\mathcal{R}(\xi(z),H_{1}^z,H_{2}^z)$.

But for $s\in (s_0-\rho',s_0+\rho')$, $\xi(s)=V^s$, so that $\xi(s)+\mathcal{R}(\xi(s),H_{1}^s,H_{2}^s)=V^s+\mathcal{R}(V^s,H_{1}^s,H_{2}^s)=V^s+W^s=U^s$ where we have used \eqref{equalWs} and $(s_0-\rho',s_0+\rho')\subset (-\nu,T^*-T-\nu)$ for $\rho'$ small enough. In particular, the map $s\in (s_0-\rho',s_0+\rho')\mapsto U^s\in C^0([0, T], X^\sigma)$ is the restriction to a real interval of a holomorphic map, and hence is real analytic.

Since for any $t_0\in [0,T]$, the trace application $C^0([0, T], X^\sigma)\to X^\sigma$ is linear continuous, then, we obtain by composition that the application  
\begin{align*}
    \left\{\begin{array}{rcl}
   (s_0-\rho',s_0+\rho')&\to& X^\sigma\\
   s&\mapsto&U^s(t_0)=U(t_0+\nu+s)
    \end{array}\right.
\end{align*}
is real analytic. Recall that $\rho'$ depends on all the other parameters, but $\nu$ is an arbitrary number with $0<\nu<T^*-T$, $s_0$ is any number with $s_0\in (-\nu,T^*-T-\nu)$ while $t_0$ is arbitrary in $[0,T]$. It means that $t\mapsto U(t)$, (which is well defined for $t\in [0,T^*]$), is real analytic in a neighborhood of any $t_1$ of the form $t_1=t_0+\nu+s_0$ with $t_0$, $\nu$ and $s_0$ as before. It is not hard to see that it implies that $t\mapsto U(t)$ is real analytic on $(0,T^*)$, as expected.
\end{proof}

\section{Application to the wave equation}\label{SEC:wave-eq}

The aim of this section is to prove the main results of propagation of analyticity and unique continuation related to the semilinear wave equation.

\subsection{Notation and preliminaries}\label{s:not_wave} Let $\beta\geq 0$. Set $X=H_0^1(\M)\times L^2(\M)$ and introduce the operator 
\begin{align*}
    A=\begin{pmatrix}
    0 & I \\
    \Delta_g-\beta & 0
    \end{pmatrix}\ \text{ with }\ D(A)=\big(H^2(\M)\cap H_0^1(\M)\big)\times H_0^1(\M).
\end{align*}
For the standard scalar product on $X$ (with $\norm{u}_{H^1_0}^2=\norm{\nabla u}_{L^2}^2+\beta \norm{u}_{L^2}^2$), we can compute
\begin{align*}
    A^*=-A\ \text{ and }\ A^*A=-A^2=\begin{pmatrix}
    -\Delta_g+\beta& 0 \\
    0 & -\Delta_g+\beta
    \end{pmatrix},
\end{align*}
where we used the same notation for the Dirichlet Laplacian defined on $H^1_0(\M)$ and $L^2(\M)$ with their natural respective domains.
Observe that $A$ satisfies Assumption \ref{assumAA} on the real Hilbert space $X$. For $\sigma\geq 0$, $H^\sigma_D$ and  $X^\sigma$ denote the spaces 
\begin{align}
\label{defSob}
H^\sigma_D&=D((-\Delta_g)^{\sigma/2}),\\
 X^\sigma&= D((A^*A)^{\sigma/2})=  D((-\Delta_g)^{(1+\sigma)/2})\times D((-\Delta_g)^{\sigma/2})=H^{1+\sigma}_D\times H^{\sigma}_D.
\end{align}
Here, $\Delta_g$ denotes the Dirichlet Laplacian defined on $L^2$.
\begin{remark}
    For $\sigma \in [0,1/2)$, we have $ X^\sigma=(H^{1+\sigma}(\M)\cap H_0^1(\M))\times H^\sigma(\M)$, while for  $\sigma \in (1/2,1]$, we have $ X^\sigma=(H^{1+\sigma}(\M)\cap H_0^1(\M))\times H_0^\sigma(\M)$, see \cite{G:67}. Notice that $X^0=X$ and $X^1=D(A)$. When $\partial\M=\emptyset$, we still write $H^1_D=H_0^1=H^1$ and assume $\beta>0$ so that Poincar\'e-like inequalities hold.
\end{remark}
By the spectral theorem, $-\Delta_g$ has a compact resolvent and thus, we can construct an orthonormal basis of eigenfunctions of $-\Delta_g$ in $L^2(\M)$, denoted by $(e_j)_{j\in \N}$ and associated to the eigenvalues $(\lambda_j)_{j\in \N}$. We introduce the high-frequency projectors $Q_n$ on the space $\overline{\text{Span}\{e_j\}_{j\geq n}}$ and then we set the low-frequency projection $P_n=I-Q_n$. In the same direction we then set $\Q_n=(Q_n, Q_n)$ and $\P_n=(P_n, P_n)$ on $X$, fitting on the abstract framework given in Section \ref{sec:abstract-construction}.

In regards to the nonlinearity, for some $\widetilde{\chi}\in C^{\infty}(\M)$ to be chosen later, we use the notation
\begin{align}
\label{defFwave}
    F: (u, v)\longmapsto (0, -\widetilde{\chi} f(u)+\beta u),
\end{align}
where $f: \R \to \R$ will be either $C^4$ or analytic. By setting $U(t):=(u(t), \partial_t u(t))$, if $\widetilde{\chi}=1$, we can write the associated Cauchy problem to \eqref{eq:nlw-1} as
\begin{align}\label{eq:nlw-1-cauchy}
\left\{\begin{array}{lc}
    \partial_t U=AU+F(U) \\
    U(0)=U_0.
    \end{array}\right.
\end{align}
We will later consider some slightly different problems, which is why we introduce the cutoff $\widetilde{\chi}$. Since $A$ is skew-adjoint, by Stone's theorem it generates a unitary group $t\mapsto e^{tA}$ on $X$ and on $D(A)$. In particular,
\begin{align*}
    \forall t\in [0, T],\  \norm{e^{tA}}_{\mc{L}(X)}= 1\ \text{ and }\ \norm{e^{tA}}_{\mc{L}(D(A))}=1.
\end{align*}
By linear interpolation, the same holds on $X^\sigma$ for any $\sigma\in (0, 1)$.

Now we establish some regularity properties for $F$. First of all, we recall a version of a result that can be found in Alinhac-G\'erard \cite[Proposition 2.2]{1991:alinhac-gerard:psido-nash-moser}, in relation to the regularity of a composition.
\begin{lemma}\label{lem:comp-reg}
    Let $g: \R\to \R$ be a $C^3(\R, \R)$ function, with $g(0)=0$. If $u\in L^\infty(\M)\cap H^s(\M)$, with $s\in (0, 2)$, then $g(u)\in L^\infty(\M)\cap H^s(\M)$ and $\norm{g(u)}_{H^s}\leq C\norm{u}_{H^s}$, where $C$ only depends on $g$ and $\norm{u}_{L^{\infty}}$.
\end{lemma}

\begin{remark}
    The previous Lemma is actually written in \cite{1991:alinhac-gerard:psido-nash-moser} for function in $H^s(\R^d)$ and $f\in C^{\infty}$. Yet, the same result holds for functions in $H^s(\M)$ when $\M$ is a compact manifold with boundary using the definition of the norm of $H^s(\M)$ by partition of unity and sum of the norm in $H^s(\R^d)$ of the functions in local coordinates and with extension. Concerning the requirement $g\in C^3$, we only notice that the Meyer multiplier Lemma \cite[Lemma 2.2.]{1991:alinhac-gerard:psido-nash-moser} requires estimates of the derivatives of the multiplier up to $\floor{s}+1\leq 2$. Since it is applied to $g'$ in \cite[Proposition 2.2]{1991:alinhac-gerard:psido-nash-moser}, it requires $g'\in C^2$. Note that, as proved in \cite{G:67}, for functions satisfying the appropriate boundary condition, (that is $f=0$ on $\partial \M$ for $1/2<s<5/2$ and no condition if $0\leq s<1/2$), the norms $H^s(\M)$ (defined as restrictions of functions in $H^s(\R^d)$ in local coordinates) and the norms $H^s_D(\M)$ defined by spectral theory are equivalent. 
\end{remark}

We now prove that $F$ satisfies Assumption \ref{assumF} and \ref{assumFholom} under suitable hypothesis on the function $f$.

\begin{proposition}\label{prop:f-assumptions}
 Let $\sigma\in (1/2, 1]$ if $d=3$ and $\sigma\in (0,1]$ if $d\leq 2$. Assume that $f\in C^4(\R,\R)$ and $\widetilde{\chi}\in C^{\infty}(\M)$. Then $F$, defined in \eqref{defFwave}, satisfies Assumption \ref{assumF} for any $\veps\in (0, 1]$ and $R_0>0$. If in addition, $f$ is real analytic, then $F$ satisfies Assumption \ref{assumFholom}.
\end{proposition}
\begin{proof} We will use the notation introduced in \eqref{defSob} for the Sobolev spaces. In what follows, $C$ will denote a generic constant, that may change from line to line, but we will specify its dependency on the parameters involved in the estimates. We consider $\beta=0$ for simplifying the statements since this term is easier to treat. Thanks to the choice of $\sigma$, we have the embedding $H_D^{1+\sigma}\hookrightarrow L^\infty$ with constant denoted $\kappa$. Since $f\in C^4$, we can apply Lemma \ref{lem:comp-reg}, so that for any $v\in H_D^{1+\sigma}$, we have that both $f(v)$ and $Df(v)-Df(0)$ are well defined in $H_D^{1+\sigma}$ and moreover $\norm{Df(v)-Df(0)}_{H_D^{1+\sigma}}\leq C\norm{v}_{H_D^{1+\sigma}}$, where $C$ depends on $Df$ and the $L^{\infty}$-norm of $v$. Therefore, using that $H^{1+\sigma}$ is an algebra, we get, for any $v$, $v'\in H_D^{1+\sigma}$,
\begin{align}
    \norm{f(v)-f(v')}_{H_D^{1+\sigma}}&=\Norm{\int_0^1 Df(v'+\tau(v-v'))(v-v')d\tau}_{H^{1+\sigma}}\\ 
    &\leq C\left(1+\Norm{\int_0^1 \big(Df(v'+\tau(v-v')-Df(0)\big)d\tau}_{H^{1+\sigma}}\right)\norm{v-v'}_{H_D^{1+\sigma}}\\
    \label{fineq1} &\leq C\left(1+\norm{v}_{H_D^{1+\sigma}}+\norm{v'}_{H_D^{1+\sigma}}\right)\norm{v-v'}_{H_D^{1+\sigma}},
\end{align}
where $C$ is a constant depending only on $Df$ and the $L^2$-norm of both $v$ and $v'$. Note that we have used $f(0)=0$ to get $\norm{f(v)-f(v')}_{H_D^{1+\sigma}}=\norm{f(v)-f(v')}_{H^{1+\sigma}}$, where we consider one norm of $H^{1+\sigma}$ defined in local coordinates so that Lemma \ref{lem:comp-reg} can be applied.

We now establish that $F$ satisfies Assumption \ref{assumF}. Let $V$, $V'\in \mathbb{B}_{4R_0}(X^\sigma)$. Observe that the first component of $V$ and $V'$, denoted by $v$ and $v'$, respectively, always stay smaller than $4\kappa R_0$ in $L^{\infty}$. This implies that $F$ is a well defined map on $\mathbb{B}_{4R_0}(X^\sigma)$. For $\veps\in (0, 1]$, we have $\sigma+\veps\leq 1+\sigma$ and thus the continuous embedding $H_D^{1+\sigma}\hookrightarrow H_D^{\sigma+\veps}$. Therefore, 
\begin{align}\label{fineq2}
    \norm{F(V)-F(V')}_{X^{\sigma+\veps}}=\norm{\widetilde{\chi}(f(v)-f(v'))}_{H_D^{\sigma+\veps}}&\leq C\norm{f(v)-f(v')}_{H_D^{1+\sigma}},
\end{align}
where $C$ is a constant depending on $\chi$. Combining inequalities \eqref{fineq1} and \eqref{fineq2}, we get
\begin{align}\label{ineq:Flip1}
    \norm{F(V)-F(V')}_{X^{\sigma+\veps}}&\leq C(1+8R_0)\norm{V-V'}_{X^\sigma},
\end{align}
where $C$ depends only $\widetilde{\chi}$, $f$ and $R_0$. Inequality \eqref{ineq:Flip1} proves the claim.

Now, we will show that $F$ satisfies Assumption \ref{assumFholom} if we assume that $f$ is real analytic. By compactness, there exists $\delta>0$ small such that $f$ extends holomorphically into the interior of the complex strip
\begin{align*}
    \mathbb{S}_{R_0, \delta}:=\{z_1+iz_2\in \mathbb{C}\ |\ |z_1|\leq 4\kappa R_0\ \text{and}\ |z_2|\leq 2\kappa \delta\}.
\end{align*}
Moreover, this extension is continuous up to the boundary and there exists a constant $M>0$ so that $|f(z)|\leq M$ for all $z\in \mathbb{S}_{R_0, \delta}$. We still denote by $f$ such an extension. A slight variant of Lemma \ref{lem:comp-reg} allows to consider the composition of smooth functions defined on domains of $\mathbb{C}$ by functions in $H^{1+\sigma}$, assuming that the composition makes sense. Since $\kappa$ is the constant of the embedding $H_D^{1+\sigma}\hookrightarrow L^\infty$, we see that $f(v)$ is well defined in $\mathbb{B}_{4R_0, 2\delta}(H^{1+\sigma}_D)$. In particular, $F$ is well defined in $\mathbb{B}_{4R_0, 2\delta}(X^\sigma)$ and satisfies the same estimate as  \eqref{fineq1}. 

Now, we will prove that this application $F$ is $\mathbb{C}$-differentiable. We write for $z\in \Int{\mathbb{S}_{R_0, \delta}}$ and $h\in\mathbb{C}$ small, $f(z+h)=f(z)+f'(z)h+h^2\int_0^1f^{(2)}(z+th)(1-t)dt$. So, for $v\in \Int{\mathbb{B}_{4R_0, 2\delta}(H^{1+\sigma}_D)} $, we can write 
\begin{align}\label{eq:f-exp}
    f(v(x)+r(x))=f(v(x))+f'(v(x))r(x)+r(x)^2\int_0^1f^{(2)}(v(x)+tr(x))(1-t)dt,
\end{align}
for $r\in \mathbb{B}_{\eta, \eta}(H^{1+\sigma}_D)$ (for some $\eta$ depending on $v$). Since $f$, $f'$ and $f^{(2)}$ are smooth functions on the range of $v+tr$ for any $t\in [0,1]$, as before, we get that all terms in \eqref{eq:f-exp} are well defined and bounded in $H^{1+\sigma}$. The Dirichlet boundary condition being satisfied for each of the three terms, we conclude that $f$ is $\mathbb{C}$-differentiable at $v$, as in Definition \ref{def-Kdifferentiable}, with derivative $r\mapsto f'(v)r$ which is continuous $\mathbb{C}$-linear from $H^{1+\sigma}_D+iH^{1+\sigma}_D$ into $H^{1+\sigma}_D+iH^{1+\sigma}_D$ and therefore into $H^{\sigma+\veps}_D+iH^{\sigma+\veps}_D$. Therefore, $f$ can be extended to admit a bounded holomorphic extension from $\mathbb{B}_{4R_0, 2\delta}(H_D^{1+\sigma})$ into $H^{\sigma+\veps}_D+iH^{\sigma+\veps}_D$ and the required extension holds for $F$ after composition by linear bounded functions.
\end{proof}

\subsubsection{Well posedness theory} For $d\leq 2$, and with assumption \eqref{hip:nonlinearity-hyp-1} for $1\leq p<+\infty$, the well-posedness theory in $H^1\times L^2$ can be performed with Sobolev embedding, so we omit the details. Assume that $\M$ is of dimension $d=3$. Later on, we will need to use some results related to the global existence and uniqueness of solutions of the semilinear wave equation in the subcritical case. We will briefly recall them.

A central argument to handle the subcritical case in dimension $d=3$ is the use of Strichartz estimates. They have a long history. We only quote the results that we use and refer to the references therein for historical background. For general domains with boundary, Strichartz estimates were proved by Burq, Lebeau and Planchon \cite{BLP:08} and later, the  range of admissible exponents was extended by Blair, Smith and Sogge \cite[Corollary 1.2]{BSS:09} leading to the following theorem.
\begin{theorem}{(Strichartz estimates)}\label{thmStrichartz}
    Let $T>0$ and $(q, r)$ satisfying
    \begin{align}\label{thm:strichartz-exponents}
        \dfrac{1}{q}+\dfrac{3}{r}=\dfrac{1}{2},\ \ \ q\in [7/2, +\infty].
    \end{align}
    There exists $C=C(T, q)>0$ such that for every $G\in L^1([0, T], L^2(\M))$ and every $(u_0, u_1)\in X$, the mild solution $u$ of
    \begin{align*}
        \left\{\begin{array}{lc}
            \partial_t^2 u-\Delta_g u=G(t) &  \\
            u_{|_{\partial\M}}=0 & \\
            (u, \partial_t u)(0)=(u_0, u_1) & 
        \end{array}\right.
    \end{align*}
    satisfies the estimate
    \begin{align*}
        \norm{u}_{L^q([0, T], L^r(\M))}\leq C\big(\norm{(u_0, u_1)}_{H^1(\M)\times L^2(\M)}+\norm{G}_{L^1([0, T], L^2(\M))}\big).
    \end{align*}
\end{theorem}

Without loss of generality, it can be assumed that $p\in (3, 5)$ for the bound on $f$ in \eqref{hip:nonlinearity-hyp-1}. The exponent $p=3$ is the exponent where Strichartz are no longer necessary and so the cases $p\in [1, 3]$ can be managed using appropriate Sobolev embeddings. Observe that it is enough to consider the pair of exponents $(q, r)=(\tfrac{2p}{p-3}, 2p)$, since they give $u^p\in L^{\frac{2}{p-3}}([0, T], L^2(\M))\subset L^1([0, T], L^2(\M))$ because $1<\tfrac{2}{p-3}<+\infty$. Once the Strichartz estimates are obtained, the global well-posedness is classical for subcritical nonlinearities. It first appeared in \cite{BLP:08} for general bounded domains. We will need the following well-posedness result. 
\begin{theorem}{(Cauchy problem)}\label{thm:cauchy-problem}
    Let $f$ satisfies \eqref{hip:nonlinearity-hyp-1}  and \eqref{defdefocusing}. For any $T>0$ and for any $(u_0, u_1)\in X=H_0^1(\M)\times L^2(\M)$ there exists a unique solution $U(t)=(u(t), \partial_t u(t))\in C^0([0, T], X)$ with finite Strichartz norm of \eqref{eq:nlw-1-cauchy}. Moreover, for any $E_0\geq 0$ and $(q, r)$ satisfying \eqref{thm:strichartz-exponents}, there exists a constant $C>0$ such that, if $U=(u, \partial_t u)$ is solution of \eqref{eq:nlw-1-cauchy} with $E(u(0))\leq E_0$, then
    \begin{align*}
        \norm{u}_{L^q([0, T], L^r(\M))}\leq C\norm{(u_0, u_1)}_{H^1(\M)\times L^2(\M)}.
    \end{align*}
    In addition, there exists a constant $C>0$ such that, if $U$ and $\Tilde{U}$ are two solutions of \eqref{eq:nlw-1} with $E(u(0))\leq E_0$ and $E(v(0))\leq E_0$, then
    \begin{align*}
        \sup_{[-T, T]}\norm{(u, \partial_t u)(t)-(\Tilde{u}, \partial_t \Tilde{u})(t)}_X\leq C\norm{(u, \partial_t u)(0)-(\Tilde{u}, \partial_t \Tilde{u})(0)}_X.
    \end{align*}
\end{theorem}
 The exact previous statement result including Lipschitz stability can be found in \cite[Theorem 2.2]{JL13}.

\subsubsection{Observability of linear waves at higher regularity} A crucial tool in the present section will be an observability inequality of linear waves, but at an appropriate regularity. We start by recalling the following classical observability result from Bardos-Lebeau-Rauch \cite{BLR:88} in the case $\partial\M\neq\emptyset$.

\begin{theorem}{\cite{BLR:88}}\label{thm:blr-wave}
    Assume that $(\omega, T)$ satisfies \ref{assumGCC}. Then, there exists $C>0$ such that for any $(v_0, v_1)\in H_0^1(\M)\times L^2(\M)$ and associated solution $v$ of 
    \begin{align}\label{eq:linear-waves}
    \left\{\begin{array}{rll}
        \partial_t^2 v-\Delta_g v=&0  &\ (t, x)\in [0, T]\times\M,\\
        v_{|_{\partial\M}}=&0  &\ (t, x)\in [0, T]\times\partial\M,\\
        (v, \partial_t v)(0)=&(v_0, v_1) &\ x\in \M,
    \end{array}\right.
\end{align}
    we have
    \begin{align}\label{thm:ineq:linear-waves-obs}
        C\norm{(v_0, v_1)}_{H_0^1(\M)\times L^2(\M)}^2\leq \int_0^T \norm{\mathbbm{1}_\omega\partial_t v(t)}_{L^2(\M)}^2 dt.
    \end{align}
\end{theorem}
We are now going to give several variants of this observability inequality. It is certainly folklore in the field that all these variations might be obtained from \ref{assumGCC}, but we did not find the precise statements that we need. At the completion of this work, we realized that Perrin \cite{P:23} already considered similar changes of regularity, especially for boundary observability, but it does not seem to be exactly the required statement, even if very close. We decided to keep our short proofs starting from results explicitly written in the literature.
From now on we assume that $(\omega, T)$ satisfies the \ref{assumGCC} and $b_\omega$ is a smooth function defined on $\M$ so that 
\begin{align}
\label{lowerb}
    b_{\omega}(x)\geq 1 \textnormal{ for }x\in \omega.
\end{align}
Therefore, \eqref{thm:ineq:linear-waves-obs} implies
\begin{align}\label{ineq:linear-waves-obs-l2-weak}
        \norm{(v_0, v_1)}_{H_0^1(\M)\times L^2(\M)}^2\leq C_2\int_0^T \norm{b_\omega\partial_t v(t)}_{L^2(\M)}^2 dt,
\end{align}
for any $(z_0, z_1)\in H_0^1(\M)\times L^2(\M)$.

Let $\beta\geq 0$ and let us consider $z$ solution of
\begin{align}\label{eq:linwavesbeta}
    \left\{\begin{array}{rll}
        \partial_t^2 z-\Delta_g z+\beta z=&0  &\ (t, x)\in [0, T]\times\M,\\
        z_{|_{\partial\M}}=&0            &\ (t, x)\in [0, T]\times\partial\M\ \text{ if } \partial\M\neq\emptyset,\\
        (z, \partial_t z)(0)=&(z_0, z_1) &\ x\in \M,
    \end{array}\right.
\end{align}  
with $(z_0, z_1)\in H_0^1(\M)\times L^2(\M)$. The associated energy to this system is given by
\begin{align}
    \E(z, \partial_t z)=\dfrac{1}{2}\int_\M \big(|\partial_t z|^2+|\nabla z|^2+\beta|z|^2\big)dx,
\end{align}
which is known to be conserved through time, namely, that $\E(z(t), \partial_t z(t))=\E(z_0, z_1)$ for all $t\in (0, T)$. In the case $\partial\M=\emptyset$, from the work of Rauch-Taylor \cite{RT74}, it is understood that the observability inequality \eqref{eq:linear-waves} still holds for \eqref{eq:linwavesbeta} under the \ref{assumGCC}. Yet, they proved such an inequality for a damped wave equation. While the same idea applies in our case, for the convenience of the reader, we will briefly outline how this inequality can be obtained for system \eqref{eq:linwavesbeta}, following the work of Laurent-L\'eautaud \cite{LL16}. In particular, the observation will be made through a smooth function $b_\omega$ supported on $\omega$. In what follows, we assume that $\beta>0$ if $\partial \M=\emptyset$ or $\beta\geq 0$ otherwise.

\begin{proposition}\label{thm:ineq:linear-waves-obs-beta}
    Assume that $(\omega, T)$ satisfies \ref{assumGCC} and $b_{\omega}$ satisfies \eqref{lowerb}. Then, there exists $C>0$ such that for any $(z_0, z_1)\in H_0^1(\M)\times L^2(\M)$ and associated solution $z$ of \eqref{eq:linwavesbeta}, we have
    \begin{align}\label{thm:ineq:linear-waves-obs-2}
        C\norm{(z_0, z_1)}_{H_0^1(\M)\times L^2(\M)}^2\leq \int_0^T \norm{b_\omega\partial_t z(t)}_{L^2(\M)}^2 dt.
    \end{align}
\end{proposition}
\begin{proof}
Let us first assume that $\partial\M=\emptyset$ and $\beta>0$. In order to observe the component $\partial_t z$, with a slight modification of \cite[Proposition 2.2]{LL16} along with Garding's inequality \cite[Theorem A.9]{LL16} (see also \cite[Theorem 1.3.]{CLW:20} for an equivalence with ODE along bicharacteristics, also valid for systems), we obtain the weak observability inequality
    \begin{align*}
        C\E(z_0, z_1)\leq \int_0^T \norm{b_\omega \partial_t z}_{L^2(\M)}^2dt+\norm{(z_0, z_1)}_{H^{1/2}\times H^{-1/2}(\M)}^2.
    \end{align*}
    Now, by employing a nowadays classical compactness-uniqueness argument (as in Proposition \ref{prop:obs-schr-plates} below), we get rid of the compact term in the right-hand side of the above inequality, which leads us to \eqref{thm:ineq:linear-waves-obs-2}.
    
Let us now treat the case $\partial\M\neq\emptyset$, with $\beta\geq 0$. It follows the same argument as \cite[Theorem 1.5]{LL16}. Let us consider $v$ solution of \eqref{eq:linear-waves}, but with the same initial data as $z$, namely, $(v, \partial_t v)(0)=(z_0, z_1)$. Therefore, $w=v-z$ solves
    \begin{align*}
    \left\{\begin{array}{rl}
        \partial_t^2 w-\Delta_g w=&\beta z  \ \\
        w_{|_{\partial\M}}=&0           \ \\
        (w, \partial_t w)(0)=&(0, 0). 
    \end{array}\right.
    \end{align*}  
    Now, energy estimates for the above equation lead us to
    \begin{align*}
        \int_0^T \norm{b_\omega \partial w}_{L^2(\M)}^2dt\leq C\norm{\partial_t w}_{L^\infty([0, T], L^2(\M))}^2\leq C\norm{z}_{L^2([0, T], L^2(\M))}^2.
    \end{align*}
    By applying the observability estimate to $v$, we get
    \begin{align*}
        C_1\norm{(z_0, z_1)}_{H_0^1(\M)\times L^2(\M)}^2\leq 2\int_0^T \norm{b_\omega \partial_t z(t)}_{L^2(\M)}^2 dt+2\int_0^T \norm{b_\omega w(t)}_{L^2(\M)}^2 dt.
    \end{align*}
    Combining the two previous inequalities, we obtain
    \begin{align*}
        C_1\norm{(z_0, z_1)}_{H_0^1(\M)\times L^2(\M)}^2\leq 2\int_0^T \norm{b_\omega \partial_t z(t)}_{L^2(\M)}^2 dt+2C\norm{z}_{L^2([0, T], L^2(\M))}^2.
    \end{align*}
    Once again, a compactness-uniqueness argument leads us to \eqref{thm:ineq:linear-waves-obs-2}.
\end{proof}

The following variant with observation in $H^1$ is a direct consequence of the previous result, but, apparently, is not written this way in the literature. We provide a proof for the convenience of the reader.
\begin{proposition}
    Assume that $(\omega, T)$ satisfies \ref{assumGCC} and $b_{\omega}$ satisfies \eqref{lowerb}. Then, there exists $C>0$ such that for any $(z_0, z_1)\in H_0^1(\M)\times L^2(\M)$ and associated solution $z$ of \eqref{eq:linwavesbeta} we have
    \begin{align}\label{thm:ineq:linear-waves-obs-1st-component}
        C\norm{(z_0, z_1)}_{H_0^1(\M)\times L^2(\M)}^2\leq \int_0^T \norm{b_\omega z(t)}_{H_0^1(\M)}^2 dt.
    \end{align}
\end{proposition}
\begin{proof}
    By Lemma \ref{lemma:gcc-smaller-subset}, we can find $(\Tilde{\omega}, \Tilde{T})$ satisfying the \ref{assumGCC} and $b_{\Tilde{\omega}}$ supported in $\omega$ with $b_{\Tilde{\omega}}=1$ on $\Tilde{\omega}$ and $\Tilde{T}<T$. Let $T'$, $T'' \in (0,T)$ and $\rho \in C^\infty_c(\mathbb{R},\mathbb{R}^+)$ be such that $T''-T'>\widetilde{T}$, $\rho \equiv 1$ on $(T',T'')$ and $\text{supp}(\rho) \subset (0,T)$.

Since the observability estimate \eqref{thm:ineq:linear-waves-obs-2} holds for the pair $(\Tilde{\omega}, \Tilde{T})$ with $b_{\Tilde{\omega}}$, we can employ it to find a constant $\Tilde{C_1}>0$ such that
    \begin{align}\label{ineq:linear-waves-obs-smaller}
            \norm{(z_0, z_1)}_{H_0^1(\M)\times L^2(\M)}^2\leq \Tilde{C_1}\int_0^{\Tilde{T}} \norm{b_{\Tilde{\omega}}\partial_t z(t)}_{L^2(\M)}^2 dt.
    \end{align}
    Noticing that there exists $\alpha>0$ so that $|b_{\Tilde{\omega}}|\leq \alpha |b_{\omega}|$, we have
    \begin{align}\label{eq:obs-h1-1}
        \norm{(z_0, z_1)}_{H_0^1(\M)\times L^2(\M)}^2\leq \Tilde{C_1}\int_0^{\Tilde{T}} \norm{b_{\Tilde{\omega}} \partial_t z(t)}_{L^2(\M)}^2 dt
        \leq C_1\int_0^T \rho(t) \norm{b_\omega \partial_t z(t)}_{L^2(\M)}^2 dt.
    \end{align}
    We aim to find a proper bound of the right-hand side in the above inequality. Let us multiply the equation $\partial_t^2 z-\Delta_gz+\beta z=0$ by the multiplier $\rho b_\omega^2 z$ and then perform integration by parts both in time and space on $[0, T]\times \M$ to obtain the identity
    \begin{multline*}
        \int_0^T\rho(t)\norm{b_\omega \partial_t z}_{L^2(\M)}^2dt=\int_0^T\rho(t)\norm{b_\omega \nabla z}_{L^2(\M)}^2dt+\beta\int_0^T \rho(t)\norm{b_\omega z}_{L^2(\M)}^2dt\\+2\iint_{[0, T]\times\M} \rho(t)b_\omega z \nabla b_\omega\cdot\nabla z dxdt+\iint_{[0, T]\times \M}\rho'(t) b_\omega^2z\partial_t z dt.
    \end{multline*}
    On the right-hand side, we can complete the square to obtain
    \begin{multline}\label{eq:obs-h1-2}
        \int_0^T\rho(t)\norm{b_\omega \partial_t z}_{L^2(\M)}^2dt=\int_0^T\rho(t)\norm{\nabla (b_\omega z)}_{L^2(\M)}^2dt+\beta\int_0^T \rho(t)\norm{b_\omega z}_{L^2(\M)}^2dt\\
        +\iint_{[0, T]\times \M}\rho'(t)b_\omega^2z\partial_t z dt-\int_0^T \rho(t)\norm{z \nabla b_\omega }_{L^2(\M)}^2.
    \end{multline}
    By using Cauchy-Schwarz and then energy estimates, we get, for any $\veps>0,$
    \begin{align*}
        \left|\iint_{[0, T]\times \M}\rho'(t)b_\omega^2z\partial_t z dt\right|&\leq \int_0^T |\rho'(t)|\norm{b_\omega z}_{L^2(\M)}\norm{b_\omega \partial_tz}_{L^2(\M)}dt\\
        &\leq \dfrac{1}{2\veps}\int_0^T \norm{b_\omega z}_{L^2(\M)}^2dt+\dfrac{\veps}{2}\int_0^T |\rho'(t)|^2\norm{b_\omega\partial_t z}_{L^2(\M)}^2dt\\
        &\leq \dfrac{1}{2\veps}\int_0^T \norm{b_\omega z}_{L^2(\M)}^2dt+\dfrac{C\veps}{2}\norm{(z_0, z_1)}_{H^1(\M)\times L^2(\M)}^2.
    \end{align*}
    By plugging the above estimate in \eqref{eq:obs-h1-2}, we get
    \begin{multline}\label{eq:obs-h1-3}
        \int_0^T\rho(t)\norm{b_\omega \partial_t z}_{L^2(\M)}^2dt\leq C\left(\int_0^T \norm{b_\omega z}_{L^2(\M)}^2dt+\int_0^T\rho(t)\norm{b_\omega z}_{H_0^1(\M)}^2dt\right)\\+C\veps\norm{(z_0, z_1)}_{H^1(\M)\times L^2(\M)}^2.
    \end{multline}
    Observe that, on the right-hand side, the first term in parenthesis can be bounded by above by $L^2([0, T], H_0^1(\M))$-norm of $b_\omega z$. By choosing $\veps>0$ small enough, we conclude by plugging \eqref{eq:obs-h1-3} into \eqref{eq:obs-h1-1}.
\end{proof}

Assume that $(z_0, z_1)\in (H^2(\M)\cap H_0^1(\M))\times H_0^1(\M)$. Take $y=\partial_t z$ and observe that it solves
\begin{align}\label{eq:linear-waves-2}
    \left\{\begin{array}{rl}
        \partial_t^2 y-\Delta_g y+\beta y=0  &\ \\
        y_{|_{\partial\M}}=0            &\ \text{if } \partial\M\neq\emptyset\\
        (y, \partial_t y)(0)=(z_1, (\Delta_g-\beta) z_0),
    \end{array}\right.
\end{align}
with $(z_1, (\Delta_g-\beta) z_0)\in H_0^1(\M)\times L^2(\M)$. Applying the observability inequality \eqref{thm:ineq:linear-waves-obs-1st-component} with $b_\omega y$ as observation and then going back to the $z$ variable, we get
\begin{align*}
    \norm{(z_0, z_1)}_{(H^2(\M)\cap H_0^1(\M))\times H_0^1(\M)}^2\leq C_2\int_0^T \norm{b_\omega \partial_t z(t)}_{H_0^1(\M)}^2dt.
\end{align*}
Let us consider the observation operator $\bC\in \mc{L}(X^\sigma, X^\sigma)$ given by
\begin{align}\label{eq:observation-operator-wave}
    \bC(\phi, \psi)=(0, b_\omega\psi).
\end{align}
With the operator's notation of Section \ref{s:not_wave}, we have obtained the inequality
\begin{align}\label{ineq:linear-waves-obs-h2h1}
    \norm{Z_0}_{D(A)}^2\leq C\int_0^T \norm{\bC e^{tA}
    Z_0}_{D(A)}^2 dt.
\end{align}
By linear interpolation in between inequalities \eqref{ineq:linear-waves-obs-l2-weak} and \eqref{ineq:linear-waves-obs-h2h1}, we obtain the following result. We also refer to \cite{P:23} for the link between observability at different levels of regularity.
\begin{proposition}\label{prop:higher-reg-obs}
    Assume that $(\omega, T)$ satisfies \ref{assumGCC} and $b_{\omega}$ satisfies \eqref{lowerb}. Then, for $\bC$ defined by \eqref{eq:observation-operator-wave}, there exists $\mathfrak{C}_{\text{obs}}>0$ such that for any $\sigma\in [0,1]$ and $Z_0\in X^\sigma$ we have
    \begin{align}\label{thm:ineq:linear-waves-obs-high}
        \norm{Z_0}_{X^\sigma}^2\leq \mathfrak{C}_{\text{obs}}^2\int_0^T \norm{\bC e^{tA}Z_0}_{X^\sigma}^2 dt.
    \end{align}
\end{proposition}
The previous result shows that, if $(\omega, T)$ satisfies the \ref{assumGCC}, then $t\mapsto e^{tA}$ satisfies Assumption \ref{assumC} with $\sigma\in [0,1]$, $T>0$ and observation operator $\bC$ given by \eqref{eq:observation-operator-wave}. 

We also show that the pair $(A, \bC)$ satisfies Assumption \ref{assumcommu}.
\begin{proposition}\label{propcommutwave}
    Let $\sigma\in [0, 1]\setminus \{1/2\}$. Then, if $\bC$ is given by \eqref{eq:observation-operator-wave} with $b_{\omega}$ smooth satisfying $\partial_{\vec n}b_{\omega}=0$, then Assumption \ref{assumcommu} is fulfilled with $s=1$ as long as $\veps\leq 1$.
\end{proposition}
\begin{proof}
    We compute $[A^*A,\bC]=\begin{psmallmatrix}  0 & 0 \\  0 & [b_{\omega},\Delta_g]\end{psmallmatrix}$ so that the result is true as long as $[b_{\omega},\Delta_g]=-2\nabla_g b_{\omega}\cdot \nabla_g -\Delta_g b_{\omega}$ sends $H^{\sigma+2}_D$ into $H^{\sigma+\veps}_D$. We claim that $[b_{\omega},\Delta_g]$ sends $H^{\sigma+2}_D$ into $H^{\sigma+1}_D$ when $\sigma\in [0, 1]\setminus \{1/2\}$, which will give the result since $\veps\leq 1$.
    
    Since $b_{\omega}$ is smooth and $\sigma\in [0, 1]\setminus \{1/2\}$, we only need to verify that the term $\nabla b_{\omega}\cdot \nabla u$ is indeed zero on $\partial\M$. Decomposing $\nabla_g =\partial_{\vec n}+\nabla_{T}$ where $\nabla_{T}$ is the tangential derivative, we can write $\nabla b_{\omega}\cdot \nabla u=\partial_{\vec n}b_{\omega} \partial_{\vec n}u+\nabla_T b_{\omega}\cdot \nabla_T u$. The first term is zero on the boundary since we assumed $\partial_{\vec n}b_{\omega}=0$ while the second term $\nabla_T u$ cancels since $u=0$ on $\partial \M$. This proves the claimed result.
\end{proof}

\subsection{Propagation of analyticity}
Here we prove Theorem \ref{thm:analytic-prop} and then Corollary \ref{coranalytspacetime}.

\begin{proof}[Proof of Theorem \ref{thm:analytic-prop}] According to Lemma \ref{lemma:gcc-smaller-subset}, there exist $\chi\in C^{\infty}_c(\omega)$ with non negative values so that 
    \begin{itemize}
        \item there exists a nonempty open set $\Tilde{\omega}$ which is compactly contained in $\omega$ and a time $0<\Tilde{T}<T$ such that $(\Tilde{\omega}, \Tilde{T})$ satisfies the \ref{assumGCC} and $\chi =1$ on $\Tilde{\omega}$,
        \item $\partial_{\vec n} \chi=0$ on $\partial \M$ where $\partial_{\vec n}$ is the normal derivative to the boundary.
    \end{itemize}
   
    Let $\chi_2 \in C^{\infty}_c(\omega)$ so that $\chi_2 =1$ on $\supp(\chi)$. Since $\chi_2$ is a cutoff function whose support is contained on $\omega$, by hypothesis, we get that $t\mapsto \chi_2u(t)$ is analytic with value in $ H^{1+\sigma}\cap H_0^1(\M)$. Note that up to exchanging $(0,T)$ with a compact subinterval so that the geometric control condition is still satisfied, we can assume without loss of generality that the analyticity holds in a neighborhood of $(0,T)$. In the same way as done in the proof of Proposition \ref{propcommutwave}, we verify that the application $[\Delta_g, \chi]$ maps $ H^{1+\sigma}\cap H_0^1(\M)$ into $H^{\sigma}_0(\M)$ when $\sigma\in [0, 1]\setminus \{1/2\}$. By composition of analytic maps, and noticing that $[\Delta_g, \chi]=\chi_2[\Delta_g, \chi]=[\Delta_g, \chi]\chi_2$, we get that the map
    \begin{align*}
        t\in (0, T)\mapsto [\Delta_g, \chi] u\in H_0^{\sigma}(\M)
    \end{align*}
    is analytic.

    By writing $u=\chi u+(1-\chi)u$, it remains to show that $t\mapsto (1-\chi)u$ is analytic. Set $\Tilde{\chi}=(1-\chi)$ and consider the new variable $z=\Tilde{\chi} u$. We then have
    \begin{align*}
        \partial_t^2 z-\Delta_g z+\beta z&=\Tilde{\chi}(\partial_t^2 u-\Delta_g u)+\beta z+[\Delta_g, \Tilde{\chi}]u=-\Tilde{\chi} f(u)+\beta z-[\Delta_g, \chi]u\\
        &=-\Tilde{\chi} f(z+\chi u)+\beta z-[\Delta_g, \chi]u=-\Tilde{\chi} f(z+h_1)+\beta (z+h_1)+h_2-\beta h_1
    \end{align*}
    where we have defined the functions $h_1=\chi u$ and $h_2=-[\Delta_g, \chi]u$.

    Let $\sigma_*\in (1/2, \sigma)$ if $d=3$ and $\sigma_*\in (0,\sigma)$ if $d\leq 2$. Now, we define 
    \begin{align*}
        \left\{\begin{array}{ccl}
            t\in [0, T] & \longmapsto & H_1(t)=(\chi u(t), 0)\in X^{\sigma_*},  \\
            t\in [0, T] & \longmapsto & H_2(t)=(0, -[\Delta_g, \chi]u(t)-\beta \chi u(t))\in X^\sigma.
        \end{array}\right.
    \end{align*}
Thanks to the previous discussion, $H_1\in C^0([0, T], X^{\sigma_*})$ and $H_2\in C^0([0, T], X^\sigma)$. Since $\norm{U}_{C^0([0, T], X^\sigma)}\leq M$ for some $M>0$, it follows that
    \begin{align*}
        H_1\in \B_{CM}^{[0, T]}(X^{\sigma_*})\ \text{ and }\ H_2\in \B_{CM}^{[0, T]}(X^\sigma),
    \end{align*}
    for some $C=C(\chi,\chi_2)=C(\omega)>0$. Since $t\in (0, T)\mapsto \chi u\in H^{1+\sigma_*}(\M)\cap H_0^1(\M)$ and $t\in (0, T)\mapsto [\Delta, \chi] u\in H^{\sigma}_0(\M)$ are analytic, an application of Theorem \ref{app:thm:h-ext} and compactness, gives the existence of $\mu>0$ so that $H_1$ and $H_2$ can be extended holomorphically as
    \begin{align*}
        \left\{\begin{array}{ccl}
            z\in (0, T)+i(-\mu, \mu) & \longmapsto & H_1(z)=(\chi u(z), 0)\in X_\mathbb{C}^{\sigma_*},  \\
            z\in (0, T)+i(-\mu, \mu) & \longmapsto & H_2(z)=(0, -[\Delta_g, \chi]u(z)-\beta \chi u(z))\in X_\mathbb{C}^\sigma.
        \end{array}\right.
    \end{align*}
    
    Moreover, since $\Re(H_1(z))\in \mathbb{B}_{CM}(X^{\sigma_*})$ for $z\in[0,T]+i\{0\}$, by shrinking $\mu>0$ if necessary, by continuity and compactness, we can assume that $\Re(H_1(z))\in \mathbb{B}_{2CM}(X^{\sigma_*})$ for every $z\in [0, T]+i[-\mu, \mu]$.

    Applying Lemma \ref{lemma:gcc-smaller-subset} again to $(\Tilde{\omega},\Tilde{T})$, we can find $(\widetilde{\omega}_1,\Tilde{T}_1)$ satisfying \ref{assumGCC} and $b\in C^{\infty}_c(\Tilde{\omega})$ so that $b=1$ on $\widetilde{\omega}_1$ and $0<\widetilde{T}_1<\Tilde{T}<T$. Defining $\bC$ by $\bC(\phi, \psi)=(0, b\psi)$ as in \eqref{eq:observation-operator-wave}, Proposition \ref{prop:higher-reg-obs} gives the following observability estimate 
   \begin{align}\label{thm:ineq:linear-waves-obs-highTstar}
        \norm{Z_0}_{X^{\sigma_*}}^2\leq \mathfrak{C}_{\text{obs}}^2\int_0^{\widetilde{T}_1} \norm{\bC e^{tA}Z_0}_{X^{\sigma_*}}^2 dt
    \end{align}
    and the same observability at the level of regularity $X^{\sigma}$.

    Observe that $\chi=1$ on $\Tilde{\omega}$, then $z=(1-\chi)u=0$ on $\widetilde{\omega}$. Then, since $b\in C^{\infty}_c(\Tilde{\omega})$, it implies $b\partial_t z=0$ on $[0,T]\times \M$. 
    Finally, we see that $Z=(z, \partial_t z)$ satisfies the system
    \begin{align}\label{thm:proof:eq:}
    \left\{\begin{array}{cl}
         \partial_t Z=AZ+F(Z+H_{1})+H_{2}&\textnormal{ on } [0,T] \\
         \bC Z(t)=0 &\textnormal{ for }t\in [0,T]
    \end{array}\right.
    \end{align}
    where $F$ is as in \eqref{defFwave}. We are then in the framework of Theorem \ref{thmabstractanalyticintro} with $(T^*, T)$, $(\sigma, \veps)$ and $R_0$ of the Theorem replaced by $(T, \widetilde{T}_1)$, $(\sigma_*, \sigma-\sigma_*)$  and $2CM$, respectively. Assumption \ref{assumFholom} is fulfilled thanks to Proposition \ref{prop:f-assumptions}. The observability estimate \eqref{thm:ineq:linear-waves-obs-highTstar} for $\sigma$ and $\sigma_*$ ensures Assumption \ref{assumCC}. We deduce that $t\in (0, T)\mapsto Z(t)\in X^{\sigma_*}$ is real analytic, hence $t\in (0, T)\mapsto U(t, \cdot)\in X^{\sigma_*}$ is real analytic as well.

  Observe that $t\in (0, T)\mapsto \partial_t U(t, \cdot)\in X^{\sigma_*}$ is an analytic map (see Proposition \ref{prop:differential-analytic}), and so is $t\in (0, T)\mapsto \partial_t U(t, \cdot)-F(U(t,\cdot))\in X^{\sigma_*}$ where this time, $F(u,v)=(0,-f(u))$ (that is as in \eqref{defFwave} with $\widetilde{\chi}=1$ and $\beta=0$). We readily get that the map $t\in (0, T)\mapsto A U(t, \cdot)\in X^{\sigma_*}$ is analytic. This implies that $t\in (0, T)\mapsto U(t, \cdot)\in X^{1+\sigma_*}$ is analytic as well.

    For the last statement, we can use local holomorphic extension and prove that it is weakly holomorphic, which is sufficient (see Theorem \ref{appendix:thm:equiv-holom}). Indeed, a continuous linear form $L$ on $X^{1+\sigma_*}$ can be written, for some $V\in X^{1+\sigma_*}$, $L(U)=\left<V,U\right>_{X^{1+\sigma_*}}=\left<(A^*A)^{1/2}V,(A^*A)^{1/2}U\right>_{X^{\sigma_*}}+\left<V,U\right>_{X^{\sigma_*}}=\left<AV,AU\right>_{X^{\sigma_*}}+\left<V,U\right>_{X^{\sigma_*}}$. In particular, since $AV\in X^{\sigma_*}$, the extension $z\mapsto L(U(z))$ is well defined and holomorphic.
\end{proof}

\begin{proof}[Proof of Corollary \ref{coranalytspacetime}]
We know that $u$ admits a holomorphic extension to $[0,T]+i[-\delta,\delta]$ with value in $H^{1+\sigma}(\M)\cap H^1_0(\M)$ with $\sigma\in (1/2,1]$. Let $\varphi\in C^{\infty}_c(\M)$. For any $z=(t,s)\in [0,T]+i[-\delta,\delta]$, we consider the well defined quantity $m(z)=\left<\partial_z^2u(z)-f(u(z)),\varphi\right>_{L^2(\M)}-\left<u(z),\Delta_g \varphi\right>_{L^2(\M)}$. Note that $m$ is holomorphic by composition. Since $u$ is holomorphic, it satisfies $\partial_{\bar z}u(z)=0$ for $z\in (0,T)+i(-\delta,\delta)$ where $\partial_{\bar z}=\frac{1}{2}\left(\partial_t+i\partial_s\right)$. In particular, $\partial_{z}u(z)=\partial_t u(z)=\frac{1}{i}\partial_s u(z)$ with equality meant in $H^{1+\sigma}\cap H^1_0$. Moreover, if we restrict $m$ to the real interval $(0,T)$, we know that it is zero since $u_{\left|(0,T)\right.}$ is solution of \eqref{eq:nlw-1}. By analytic continuation, we conclude that $m(z)=0$ for $z\in (0,T)+i(-\delta,\delta)$. That means that $\partial_z^2 \left<u(z),\varphi\right>_{L^2(\M)}=\left<f(u(z)),\varphi\right>_{L^2(\M)}+\left<u(z),\Delta_g \varphi\right>_{L^2(\M)}$ for any $\varphi\in C^{\infty}_c(\M)$ and $z\in (0,T)+i(-\delta,\delta)$.

Now, for any $t_0\in (0,T)$, we consider the function $v\in C^{\infty}((-\delta,\delta),H^{1+\sigma}\cap H^1_0(\M))$ defined by $s\in (-\delta,\delta)\mapsto v(s):=u(t_0+is)\in H^{1+\sigma}\cap H^1_0(\M)$. For any $\varphi\in C^{\infty}_c(\M)$, we compute 
\begin{align*}
    \partial_s^2 \left<v(s),\varphi\right>_{L^2(\M)}&=\left<\partial_s^2 v(s),\varphi\right>_{L^2(\M)}=-\left<\partial_z^2 u(t_0+is),\varphi\right>_{L^2(\M)}\\
    &=-\left<f(u(t_0+is)),\varphi\right>_{L^2(\M)}-\left<u(t_0+is),\Delta_g \varphi\right>_{L^2(\M)}\\
    &=-\left<f(v(s)),\varphi\right>_{L^2(\M)}-\left<v(s),\Delta_g \varphi\right>_{L^2(\M)}.
\end{align*}
In particular, $v$ is solution, in the distributional sense of $\partial_s^2v+\Delta_g v=-f(v)$ on $(-\delta,\delta)\times \M$ with Dirichlet boundary condition on $\partial \M$. Since $v\in C^{\infty}((-\delta, \delta),H^{1+\sigma}\cap H^1_0(\M))\subset C^0((-\delta,\delta)\times \M)$, standard elliptic regularity states that $v$ is smooth and therefore analytic, see for instance  Friedman \cite[Theorem 5]{Friedman:58}. In particular, it gives that for any $x_0\in \M$, in some charts around $x_0$, there exists $R>0$ and $C>0$ so that 
\begin{align*}
   \left| \partial_s^{\alpha}\partial_x^{\beta}v(0,x_0)\right|\leq C R^{\alpha+|\beta|}\alpha!\beta!.
\end{align*}
By definition of $u$ and holomorphy with value in $H^{1+\sigma}\cap H^1_0$ we have $\partial_s^{\alpha}v(0)=(i\partial_t)^{\alpha}u(t_0)$ with equality in $H^{1+\sigma}\cap H^1_0$. Since the valuation at $x_0\in \M$ is continuous on $H^{1+\sigma}\cap H^1_0$, we get $\partial_s^{\alpha}v(0,x_0)=(i\partial_t)^{\alpha}u(t_0,x_0)$ for any $x_0\in \M$. Taking derivative in $x_0$ now, we have $\partial_s^{\alpha}\partial_x^{\beta}v(0,x_0)=(i\partial_t)^{\alpha}\partial_x^{\beta}u(t_0,x_0)$. So, we obtain
\begin{align*}
   \left| \partial_t^{\alpha}\partial_x^{\beta}u(t_0,x_0)\right|\leq C R^{\alpha+|\beta|}\alpha!\beta!.
\end{align*}
This is the analyticity close to $(t_0,x_0)$ and gives the result since $(t_0,x_0)\in (0,T)\times \M$ are arbitrary.
\end{proof}

\subsection{Finite determining modes} Now we show that the property of finite determining modes holds the observed problem \eqref{eq:nlw-1}.
\begin{proposition}
\label{propfinitedetwave}
 Let $\sigma\in (1/2, 1]$ if $d=3$ and $\sigma\in (0,1/2)$ if $d\leq 2$. With the notations of Section \ref{s:not_wave}, assume $(\omega, T)$ satisfies \ref{assumGCC} and that $f\in C^4(\R, \R)$. For any $R_0>0$, there exists $n\in \N$ such that the following holds. Let $h\in  C^0([0,T], H_0^{\sigma}(\M))$ and $g\in L^2([0, T], H_0^\sigma(\M))$. Let $U(t)=(u(t),\partial_t u(t))$ and $\widetilde{U}(t)=(\tilde{u}(t), \partial_t\tilde{u}(t))$ be two solutions on $(0,T)$ of 
  \begin{align}\label{eq:nlw-uc_source}
    \left\{\begin{array}{rl}
        \partial_t^2 u-\Delta_g u+f(u)=h(t, x)  &\ (t, x)\in [0, T]\times \text{Int}(\M),\\
        u_{|_{\partial\M}}=0            &\ (t, x)\in [0, T]\times\partial\M,\\
        \partial_t u=g &\ (t, x)\in [0, T]\times \omega,
    \end{array}\right.
\end{align}
such that $\norm{U(t)}_{X^\sigma}\leq R_{0}$ and $\norm{\widetilde{U}(t)}_{X^\sigma}\leq R_{0}$ for all $t\in [0, T]$. If $\P_n U(t)=\P_n \widetilde{U}(t)$ for all times $t\in [0, T]$, then $U(t)\equiv \widetilde{U}(t)$ for all $t\in [0, T]$.
\end{proposition}
\begin{proof}
We have already established in Section \ref{s:not_wave} that $A$ satisfies Assumption \ref{assumAA}. Under the hypothesis that $(\omega, T)$ satisfies the \ref{assumGCC}, Proposition \ref{prop:higher-reg-obs} implies Assumption \ref{assumC}. Finally, from Proposition \ref{prop:f-assumptions}, Assumption \ref{assumF} is satisfied for some $\veps\in (0, 1]$ where $F$ is defined by \eqref{defFwave} with $\widetilde{\chi}=1$. Then, the result follows as a direct application of the abstract Proposition \ref{prop:finite-det-modes}.
\end{proof}

\subsection{Unique continuation and equilibrium points} Here we will consider $U=(u, \partial_t u)\in C^0([0, T], X)$ solution of the system
\begin{align}\label{eq:nlw-uc}
    \left\{\begin{array}{rl}
        \partial_t^2 u-\Delta_g u+f(u)=0  &\ (t, x)\in [0, T]\times \text{Int}(\M),\\
        u_{|_{\partial\M}}=0            &\ (t, x)\in [0, T]\times\partial\M,\\
        \partial_t u=0 &\ (t, x)\in [0, T]\times \omega.
    \end{array}\right.
\end{align}
The purpose of this section is to prove Theorem \ref{thm:unique-continuation-nlw}.

\subsubsection{Propagation of regularity} Here we aim to prove a regularization effect for functions satisfying \eqref{eq:nlw-uc} in the subcritical case. When $\M$ is of dimension $d=2$ and $f$ has polynomial growth with $p\in [1, +\infty)$, it is possible to prove a gain of regularity in the nonlinearity by means of appropriate Sobolev embedding. The same holds when $d=3$ and $p\in [1, 3]$. However, when $d=3$ and $p\in (3, 5)$, this gain of regularity needs to be handled by means of Strichartz estimates. Dehman-Lebeau-Zuazua \cite[Theorem 8]{DLZ03} proved that the nonlinearity is more regular than it seems to be and they use this fact in a key fashion to stabilize the semilinear wave equation in an unbounded domain. We recall such regularity result from \cite[Corollary 4.2]{JL13}, which in particular fits in our geometrical setting.

\begin{proposition}\label{prop:gain-reg-nonlinearity}
    Assume $d=3$. Let $R>0$ and $T>0$. Let $s\in [0, 1)$ and $\veps=\min\{1-s, (5-p)/2, (17-3p)/14\}>0$ with $p$ as in \eqref{hip:nonlinearity-hyp-1}. There exist $(q, r)$ satisfying \eqref{thm:strichartz-exponents} and $C>0$ such that the following property holds. If $v\in L^\infty([0, T], H^{1+s}(\M)\cap H_0^1(\M))$ is a function with finite Strichartz norms $\norm{v}_{L^q([0, T], L^r(\M))}\leq R$, then $f(v)\in L^1([0, T], H_0^{s+\veps}(\M))$ and moreover
    \begin{align*}
        \norm{f(v)}_{L^1([0, T], H_0^{s+\veps}(\M))}\leq C\norm{v}_{L^\infty([0, T], H^{1+s}(\M)\cap H_0^1(\M))}.
    \end{align*}
    The constant $C$ depends only on $\M$, $(q, r)$, $R$ and the constant in estimate \eqref{hip:nonlinearity-hyp-1}.
\end{proposition}
\begin{remark}\label{rk:gain-reg-cpct-pert}
    According to the geometric framework considered in \cite{JL13}, the previous result remains true when $\M$ is a compact perturbation of $\R^3$, that is, $\R^3\setminus\O$ where $\O$ is a bounded smooth domain, endowed with a smooth metric equal to the euclidean one outside of a ball.
\end{remark}

With this gain of regularity at hand along with the assumption that $(\omega, T)$ satisfies the \ref{assumGCC}, we are able to propagate the regularity through the observability estimate \eqref{thm:ineq:linear-waves-obs-high}.

\begin{proposition}\label{prop:propagation-regularity}
    Let $U\in C^0([0, T], X)$ be a mild solution of the system \eqref{eq:nlw-uc} with finite Strichartz norm. Then
    \begin{align*}
        U\in C^0([0, T], X^{\upsilon})
    \end{align*}
    for all $\upsilon\in [0, 1)$. In particular $u\in L^\infty([0, T]\times \M)$.
\end{proposition}

\begin{proof}
    The proof is the same as Proposition \ref{prop:abs-prop-reg}, except that the nonlinear term is not bounded, yet the Duhamel term is well defined thanks to Sobolev embedding or Strichartz estimates depending on the case; see Remark \ref{rk:abs-prop-reg}. Therefore, we only need to check that $\T: t\mapsto \int_0^t e^{A(t-s)}F(U(s))ds$ defines a bounded map from $L^\infty([0, T], X)$ into $L^\infty([0, T], X^\veps)$ for some $\veps>0$.
 
    From \cite[Proposition 3.2]{2020:joly-laurent:decay-nlw-no-gcc}, when $d=2$, we know that $F$ maps bounded sets of $X$ into bounded sets of $X^\veps$ for any $\veps\in [0, 1)$. When $d=3$ and $p\in [1, 3]$, the same holds with $\veps\in [0, (3-p)/2)$. When $d=3$ and $p\in (3, 5)$, since $u\in H_0^1(\M)$ and $f$ is subcritical, Proposition \ref{prop:gain-reg-nonlinearity} implies that for a given $\veps>0$ depending on $p$, $f(u)$ is globally bounded in $L^1([0, T], H_0^{\veps}(\M))$. In any case, $\T$ is bounded in $C^0([0, T], X^\veps)$ for an appropriate choice of $\veps>0$.

    Without loss of generality, let us now fix $\veps>0$ small enough so we can encompass all the aforementioned cases simultaneously. For $d=2$ and $d=3$, from Proposition \ref{prop:higher-reg-obs}, Assumption \ref{assumCC} is satisfied for $\sigma=0$ and $\veps>0$. Let $b_\omega$ be given by Lemma \ref{lemma:gcc-smaller-subset} and let $\bC$ be as in \eqref{eq:observation-operator-wave}. Then, Proposition \ref{propcommutwave} ensures that the pair $(A, \bC)$ satisfies Assumption \ref{assumcommu} with $s=1$ and if $\veps\leq 1$. We then reproduce the proof of Proposition \ref{prop:abs-prop-reg} to obtain that $U$ is bounded in $C^0([0, T], X^\veps)$.
    
    We can iterate the previous process to obtain that $U$ is bounded in $C^0([0, T], X^{k\veps})$ for $k\in \N$ as long as Proposition \ref{prop:gain-reg-nonlinearity} or simply Proposition \ref{prop:f-assumptions} apply. We obtain finally that $U$ is bounded in $C^0([0, T], X^\upsilon)$ for any $\upsilon\in [0, 1)$.
    
\end{proof}
\begin{remark}
\label{rk:propagsource}
Proposition \ref{prop:propagation-regularity} can be easily extended with the same proof by replacing \eqref{eq:nlw-uc} by equations of the form  $\partial_t^2 u-\Delta_g u+\chi(x)f(u+h_1)=h_2$ with $h_1\in C^0([0,T],H^2_D)$ and $h_1\in C^0([0,T],H^1_D)$.
\end{remark}

\subsubsection{On unique continuation for linear waves}\label{s:UCPword} Based on the works of Tataru, Robbiano-Zuily and H\"ormander, it is established that global unique continuation holds under the framework of partial analyticity and very general geometric assumptions, provided sufficient time has passed so that we do not contradict the finite speed of propagation. To obtain the nonlinear unique continuation property, we will follow \cite{JL13} and aim to treat the nonlinearity as a potential term. Therefore, we must ensure that our framework allows for the application of the result for linear waves. In what follows, we will recall a version of the Robbiano-Zuily-Hormander-Tataru result, which is well suited to our specific context.

We will first introduce some geometric quantities needed to state the result. For $E\subset \M$, we can define the largest distant from $E$ to a point in $\M$ by
\begin{align*}
    \mc{L}(\M, E)=\sup_{x\in\M}\text{dist}(x, E).
\end{align*}
If $E$ is open, the quantity
    \begin{align*}
        T_{UC}(E):=2\mc{L}(\M, E),
    \end{align*}
is the minimal time of unique continuation for the (linear) wave equation from an open set $E$, see Tataru \cite{1995:tataru:ucp-general}. We say that an open set $\omega$ satisfies the \ref{assumGCC} if there exists $T>0$ such that $(\omega, T)$ satisfies the \ref{assumGCC}. One can then define the minimal control time associated with $\omega$ by
    \begin{align*}
        T_{GCC}(\omega)=\inf\{T>0\ |\ (\omega, T)\ \text{satisfies \ref{assumGCC}}\}.
    \end{align*}
It can be proved that $T_{GCC}(\omega)\geq T_{UC}(\omega)$, see Laurent-L\'eautaud \cite[Lemma B.4.]{LL16}.

\begin{remark}
    For the case where $\partial\M=\emptyset$, the critical time $T_{GCC}(\omega)$ is not allowed, since, as shown in \cite[Theorem 1.1]{LL16}, the observability estimate always fails for such time.
\end{remark}

We now state the unique continuation property for linear waves with coefficients analytic in time, due to Tataru \cite{1995:tataru:ucp-general,Tataru:99}, Robbiano-Zuily \cite{RZ:98} and H\"ormander \cite{Hor:97}. We refer to \cite[Theorem 6.1]{2019:laurent-leautaud:quantitative-uc-waves} for a quantitative statement that implies the unique continuation and contains the construction of hypersurfaces that allows to obtain the global result.

\begin{theorem}[Tataru-Robbiano-Zuily-H\"ormander]\label{thm:qucp}
    Let $\M$ be a compact Riemannian manifold with (or without) boundary, $\Delta_g$ the Laplace-Beltrami operator on $\M$, and
    \begin{align*}
        P=\partial_t^2-\Delta_g+W_0\partial_t+W_1\cdot\nabla+V
    \end{align*}
    with $V$, $W_0$, $W_1$, $\text{div}(W_1)$ bounded and depending analytically on the variable $t\in (0, T)$. Let $\omega$ be a nonempty open subset of $\M$ and $T>2\mc{L}(\M, \omega)$. Let $(u_0, u_1)\in H_0^1(\M)\times L^2(\M)$ and associated solution $u$ of
    \begin{align*}
        \left\{\begin{array}{ll}
            Pu=0&\ \text{ in }\ (0, T)\times\Int{\M}\\
            u_{|_{\partial\M}}=0 &\ \text{ in }\ (0, T)\times \partial\M\\
            (u, \partial_t u)(0)=(u_0, u_1). &\ 
        \end{array} \right.
    \end{align*}
    Then, if $u$ satisfies $u=0$ on $[0,T]\times \omega$, then $u=0$ on $[0,T]\times \M$.
\end{theorem}

\subsubsection{Unique continuation for semilinear waves} We now prove Theorem \ref{thm:unique-continuation-nlw} and then Proposition \ref{propR3obstacle}.

\begin{proof}[Proof of Theorem \ref{thm:unique-continuation-nlw}] Let $U(t)=(u(t), \partial_t u(t))$ be a solution of \eqref{eq:nlw-uc}, which, by Proposition \ref{prop:propagation-regularity}, belongs to $\in L^\infty([0, T], X^\sigma)$ for any $\sigma\in (1/2, 1)$. For any $\chi\in C^{\infty}_c(\omega) $, the application $t\in (0, T)\mapsto \chi u(t,\cdot)\in H^{1+\sigma}(\M)\cap H_0^1(\M)$ does not depend on $t$ and is therefore analytic. In particular, Theorem \ref{thm:analytic-prop} applies and we get that $t\in (0, T)\mapsto U(t)\in X^\sigma$ is analytic.

Since $u$ is smooth with respect to $t$ and $f$ is smooth, by writing
\begin{align*}
    \Delta_g u=\partial_t^2 u+f(u)
\end{align*}
we get that $\Delta_g u\in L^2(\M)$ and so $u\in H^2(\M)$. We differentiate the above equation to obtain
\begin{align*}
    \Delta_g^2 u&= \Delta_g (\partial_t^2 u+f(u))=\partial_t^2 \Delta_g u+\Delta_g f(u)\\
    &=\partial_t^4 u+f'(u)\partial_t^2 u+f''(u)(\partial_t^2 u)^2+\Delta_g f(u)
\end{align*}
which shows that $u$ belongs to $H^4(\M)$. This process can be repeated as many times as wanted, so by classical Sobolev embedding we get that $u=u(t, x)$ is smooth with respect to $x$. In particular, $(t, x)\in [0, T]\times \M\mapsto u(t, x)\in C^\infty([0, T]\times \M)$ is bounded, together with all its derivatives.

Set $z=\partial_t u$ and observe that $z$ solves
\begin{align}\label{thm:proof:eq:nlw-unique-cont}
    \left\{\begin{array}{rl}
        \partial_t^2 z-\Delta_g z+f'(u)z=0  &\ (t, x)\in [0, T]\times \text{Int}(\M),\\
        z_{\left|\partial\M\right.}=0            &\ (t, x)\in [0, T]\times\partial\M,\\
        z=0 &\ (t, x)\in [0, T]\times \omega.
    \end{array}\right.
\end{align}
By the previous discussion $(t, x)\mapsto f'(u(t, x))$ is bounded, analytic in $t$ and smooth in $x$. Since $T>T_{GCC}\geq T_{UC}=2\mc{L}(\M, \omega)$, we can apply Theorem \ref{thm:qucp} to get  that $z\equiv 0$ everywhere. This in turn means that $u(t, x)=u(x)$ is constant in time and henceforth it solves
\begin{align}
    \left\{\begin{array}{rl}
        -\Delta_g u+f(u)=0  &\ x\in\text{Int}(\M),\\
        u=0            &\ x\in \partial\M.
    \end{array}\right.
\end{align}
Moreover, multiplying the latter equation by $u$ and integrating by parts leads us to the identity
\begin{align*}
    0\leq \int_\M |\nabla u(x)|^2 dx=-\int_\M u(x)f(u(x)) dx.
\end{align*}
Under the assumption that $sf(s)\geq 0$ for all $s$, the above identity, the connectedness of $\M$ and the boundary condition imply that $u\equiv 0$ everywhere. In the case $\partial\M=\emptyset$, we get
\begin{align*}
    0\leq  \int_\M |\nabla u(x)|^2 dx=-\int_\M u(x)f(u(x)) dx\leq -\gamma\int_\M |u(x)|^2dx,
\end{align*}
which directly implies $u\equiv 0$ everywhere.
\end{proof}

\begin{proof}[Proof of Proposition \ref{propR3obstacle}] Up to increasing $R$, we can assume without loss of generality that $\O\Subset B(0,R)$. Since $\partial_t u$ vanishes on $(0, T)\times\big(\R^3\setminus B(0,R)\big)\subset (0, T)\times \omega $, we have $-\Delta u+f(u)=0$ with $u\in H^1(\R^3\setminus B(0, R))$. Since $f$ is subcritical, we can use elliptic regularity and bootstrap (see, for instance, \cite[Theorem 9.19]{GT01}) to show that $u=u(x)$ belongs to $C^{4, \al}$ for some $\al\in (0, 1)$ in the set $\R^3\setminus B(0, R)$. 

Let us take $R_1>R$. We can easily construct a radial cutoff function $\chi\in C_c^\infty(\overline{\Omega})$ satisfying:
\begin{itemize}
    \item $\chi=1$ in $B(0, R)\setminus \O$;
    \item $\chi=0$ in $\Omega\setminus B(0, R_1)$;
    \item $\partial_{\vec n}\chi=0$ in $\partial\Omega=\partial\O$;
    \item $\supp\nabla\chi\subset B(0, R_1)\setminus B(0, R)$.
\end{itemize}
Observe that $(1-\chi)$ is supported in $\R^3\setminus B(0, R)$ and that $u$ does not depend on time in such a set and is regular in space. Henceforth, $(1-\chi) u$ is analytic as a map from $(0, T)$ into $H^{1+\sigma}(\Omega)\cap H_0^1(\Omega)$, for some $\sigma\in (1/2, 1)$, which we fix from now on. Additionally, pick $\veps>0$ so that $\sigma+\veps<1$. It remains to check that $\chi u$ is analytic on the time variable $t$.

Let us take $\widetilde{R}>R_1$ and consider $\M=B(0,\widetilde{R})\setminus \O$, so that $\M\cap\omega\neq \emptyset$ and $\partial \M=\partial \O\cup S(0,\widetilde{R})$. Since $\chi$ is $0$ outside the ball $B(0, R_1)$, we can consider it as a cutoff function in $C_c^\infty(\overline{\M})$ with $\partial_{\vec n}\chi=0$ on $\partial\M$. Let $\widetilde{\chi}\in C_c^\infty(\overline{\Omega})$ be another cutoff with the same properties as $\chi$ and $\widetilde{\chi}=1$ on $\textnormal{Supp}(\chi)$. Since the operator $f$ is local, we have $\chi f(u)=\chi f(\widetilde{\chi}u)$. The equation satisfied by $z:=\chi u$ is then
\begin{align}
\label{eq:zobstacle}
    \left\{\begin{array}{rl}
        \partial_t^2 z-\Delta_g z+\chi f(z+(1-\chi)\widetilde{\chi}u)-[\Delta_g, \chi]u=0 &\ (0, T)\times \M, \\
        z_{|_{\partial \M}}=0 &\ (0, T)\times \partial\M, \\
        \partial_t z=0 &\ (0, T)\times \widetilde{\omega},\\ 
    \end{array}\right.
\end{align}
where $\widetilde{\omega}=\M\cap\omega$. Observe that $[\Delta_g, \chi]$ is supported in the annulus $B(0, R_1)\setminus B(0, R)\subset \widetilde{\omega}$. Moreover, due to the regularity of $u$ in such a set (recall that it does not depend on time there) and the condition $\partial_{\vec n} \chi=0$ in $\partial\M$, $[\Delta_g, \chi]u$ is an analytic map from $(0, T)$ into $H_D^{2}(\M)\subset H_D^{\sigma+\veps}(\M)$. The same holds for $(1-\chi)\widetilde{\chi}u$ that defines an analytic map from $(0, T)$ into $H_D^{2}(\M)\subset H_D^{1+\sigma}$, where the cutoff $\widetilde{\chi}$ ensures the correct boundary condition on $S(0,\widetilde{R})$. In particular, since Proposition \eqref{prop:propagation-regularity} still holds for the equation \eqref{eq:zobstacle} (see Remark \ref{rk:propagsource}), we obtain that $(z, \partial_t z)\in C^0([0, T], X^\sigma)$.

We are then in the configuration of Theorem \ref{thmabstractanalyticintro}. As we did in the proof of Theorem \ref{thm:analytic-prop}, we get that $t\mapsto z(t)\in H^{1+\sigma}(\M)\cap H_0^1(\M)$ is real analytic.

Summarizing, we have proved that $t\in (0, t)\mapsto u(t)\in H^{1+\sigma}(\Omega)\cap H_0^1(\Omega)$ is analytic. A version of Theorem \ref{thm:qucp} of unique continuation for linear waves for unbounded domains can be applied (see \cite[Corollary 3.12]{JL13}), from which we get that $\partial_t u=0$ in the whole cylinder $(0, T)\times \Omega$. We then conclude as we did for Theorem \ref{thm:unique-continuation-nlw}.
\end{proof}

\subsection{Observability inequality for the nonlinear equation}
Once the unique continuation property has been proved, the proof of the observability estimate follows some ideas from earlier articles. We follow in particular, the scheme introduced in \cite{DLZ03} with the further simplification of \cite{JL13} that replaced the use of microlocal defect measure by the decay of the semigroup. We follow a similar path, except that we want to have the observability in finite time, the one of \ref{assumGCC}. Therefore, we have to use instead the observability estimate as a black box.
\begin{proof}[Proof of Theorem \ref{thmobserintro}]We make the proof for $d=3$ and, without loss of generality, we assume that $p\in (3,5)$. For $d\leq 2$, the proof is the same using different Strichartz norms and the polynomial bound of the nonlinearity. To simplify the notation, we denote $\mc{Z}$ the Banach space of vectors $W=(u,v)$ so that $W\in C^0([0, T],X)$ and $u\in L^{4}([0, T],L^{12}(\M))$, endowed with the natural norm. We argue by contradiction. Assume that \eqref{obsevNLintro} is not satisfied. Then, there exists a sequence $(u_{0,n},u_{1,n})\in H^1_0\times L^2(\M)$, with $\nor{(u_{0,n},u_{1,n})}{H^1_0\times L^2}\leq R_0$, so that the unique solution $(u_{n}, \partial_t u_{n})\in \mc{Z}$ of \eqref{eq:nlw-1} satisfies
\begin{align}\label{obsevNLintronegation}
   \int_0^T \norm{\mathbbm{1}_\omega\partial_t u_{n}(t)}_{L^2(\M)}^2 dt
     \leq \frac{1}{n}\norm{(u_{0,n},u_{1,n})}_{H_0^1(\M)\times L^2(\M)}^2.
\end{align}
Up to taking a subsequence, we can assume $(u_{0,n},u_{1,n}) $ converges weakly to $(u_{0},u_{1})\in H^1_0\times L^2$. The global well-posedness theory for \eqref{eq:nlw-1} and the defocusing assumption states that the sequence $u_{n}$ is globally bounded in $C^{0}([0, T],H_0^1(\M))$ with uniformly bounded Strichartz norms $L^{4}([0, T],L^{12}(\M))\leq R$; see Theorem \ref{thm:cauchy-problem}.
     
In particular, we can extract a subsequence (still denoted $u_{n}$) so that $u_{n}$ converges weakly to some $u\in L^{4}([0, T],L^{12}(\M))$. 

     We will prove that $u$ is solution of the nonlinear equation with initial datum $(u_{0},u_{1})$. More precisely, denoting $U_0=(u_{0},u_{1})$, the well-posedness theory allows to define the nonlinear solution
        \begin{align*}
        V(t)=e^{tA}U_0+\int_0^t e^{A(t-s)}F(V(s))ds=V_{lin}+V_{Nlin},
    \end{align*}
     and we want to prove that $U=V$ where $U=(u,\partial_{t }u)$. The operators $A$ and $F$ are defined as in Section \ref{s:not_wave} with the appropriate choice of $\beta$.

     Observe that the sequence $(u_{n}, \partial_{t }u_{n})$ is bounded in $C^{0}([0, T],H_0^1(\M)\times L^2(\M))$. Using the Aubin-Lions Lemma (see, for instance, \cite[Corollary 4]{S:87}), we obtain that, for any $\eta>0$, still up to a subsequence, we can assume that $u_{n}$ converges strongly to $u$ in $C^{0}([0, T],H^{1-\eta}(\M))$. In particular, choosing $\eta=\frac{5-p}{4}>0$, by Sobolev embedding, we conclude that  $u_{n}$ converges strongly to $u$ in $L^{\infty}([0, T],L^{r}(\M))$ for $r=\frac{12}{7-p}<6$. Using \eqref{hip:nonlinearity-hyp-1}, H\"older estimates for $\frac{1}{2}=\frac{1}{r}+\frac{p-1}{12} $ and $p<5$, we get
    \begin{align*}
    \norm{f(u_{n})-f(u)}_{L^{1}([0, T], L^{2}(\M))}&\leq \norm{u_{n}-u}_{L^{\infty}([0, T], L^{r}(\M))}\norm{1+ u_{n}^{p-1}+u^{p-1}}_{L^{1}([0, T],L^{\frac{12}{p-1}}(\M))}\\
    &\leq C \norm{u_{n}-u}_{L^{\infty}([0,T], L^{r}(\M))}\left(1+\norm{u_{n}}_{L^{p-1}([0, T],L^{12}(\M))}^{p-1}+\norm{u}_{L^{p-1}([0, T],L^{12}(\M))}^{p-1}\right)\\
     &\leq C(T) \norm{u_{n}-u}_{L^{\infty}([0,T], L^{r}(\M))}\left(1+\norm{u_{n}}_{L^{4}([0, T],L^{12}(\M))}^{p-1}+\norm{u}_{L^{4}([0, T],L^{12}(\M))}^{p-1}\right).
    \end{align*} 
We obtain that $f(u_{n})$ converges strongly to $f(u)$ in $L^{1}([0, T], L^{2}(\M))$. The Duhamel formulation gives
    \begin{align*}
        U^n(t)=e^{tA}U_0^n+\int_0^t e^{A(t-s)}F(U^n(s))ds=U_{lin}^n+U_{Nlin}^n,
    \end{align*}
with 
    \begin{align}
    \label{e:cvgNL}
\norm{U_{Nlin}^n-V_{Nlin}}_{\mc{Z}}\underset{n\to +\infty}{\longrightarrow}0.
    \end{align}
    Since the application $U_0 \mapsto e^{\cdot A}U_0$ is linear continuous from $X$ to $\mc{Z}$, it is also continuous for the weak topology on each space. In particular, since $U_0^n$ converges weakly to $U_0$ in $X$, we obtain that $e^{\cdot A}U_0^n$ converges weakly to $e^{\cdot A}U_0=V_{lin}$ in $L^{4}([0, T],L^{12}(\M))$. In particular, $U^n$ converges weakly to $V$  in $\mc{Z}$ and $U=V$ as expected. 

We have obtained that $u\in C^{0}([0, T],H_0^1(\M))\cap L^{4}([0, T],L^{12}(\M))$ is a mild solution of  
     \begin{align*}
    \left\{\begin{array}{rll}
        \partial_t^2 u-\Delta_g u+f(u)=&0  &\ (t, x)\in [0, T]\times \text{Int}(\M),\\
        u_{|_{\partial\M}}=&0            &\ (t, x)\in [0, T]\times\partial\M,\\
            (u, \partial_t u)(0)=&(u_0, u_1)&\ x\in \M .
    \end{array}\right.
\end{align*}
Moreover, using \eqref{obsevNLintronegation} and taking weak limit, we get $\partial_t u=0$ in $[0, T]\times\omega$.
In particular, we are in a position to apply Theorem \ref{thm:unique-continuation-nlw} and we obtain $u=0$. In particular, \eqref{e:cvgNL} can be written
    \begin{align}
    \label{e:cvgNLbis}
\norm{U_{Nlin}^n}_{\mc{Z}}\underset{n\to +\infty}{\longrightarrow}0.
    \end{align}
The observability inequality allows to write
    \begin{align*}
      \norm{U_0^n}_{X}^2  &\leq \mathfrak{C}_{\text{obs}}^2\int_0^T \norm{\mathbbm{1}_\omega\partial_t u_{lin}^n}_{L^2(\M)}^2dt\\
        &\leq 2\mathfrak{C}_{\text{obs}}^2\int_0^T \norm{\mathbbm{1}_\omega\partial_t u^n}_{L^2(\M)}^2dt+2\mathfrak{C}_{\text{obs}}^2\int_0^T \norm{\mathbbm{1}_\omega\partial_t u_{Nlin}^n}_{L^2(\M)}^2dt.
    \end{align*}
When combined with \eqref{obsevNLintronegation} and \eqref{e:cvgNLbis}, we obtain $\alpha_n:=\norm{U_0^n}_{X}\underset{n\to +\infty}{\longrightarrow}0$.

Now that we know that the initial datum converges to zero strongly, we can "linearize" and consider the nonlinear solution as close to the linear one for which the observability is known. More precisely, denote $w_n=u_{n}/\alpha_n$ mild solution of
     \begin{align*}
    \left\{\begin{array}{rll}
        \partial_t^2 w_n-\Delta_g w_n+\alpha_n^{-1}f(\alpha_n w_n)=&0  &\ (t, x)\in [0, T]\times \text{Int}(\M),\\
        (w_n)_{|_{\partial\M}}=&0            &\ (t, x)\in [0, T]\times\partial\M,\\
            (w_n, \partial_t w_n)(0)=&(w_{n,0}, w_{n,1})&\ x\in \M, 
    \end{array}\right.
\end{align*}
with $\norm{(w_{n,0}, w_{n,1})}_{X}=1$.  If we write $f(s)=f'(0)s+h(s)$ (with $f'(0)\geq 0$ thanks to the assumption on $f$), we have the estimate
    \begin{align}\label{hypNL-R}
      \ |h(s)|\leq C(|s|^2+|s|^p)\ \text{and}\ |h'(s)|\leq C(|s|+|s|^{p-1}).
    \end{align}
The new nonlinearity $f_n(s)=\alpha_n^{-1}f(\alpha_n s)$ can be written as  $f_n(s)=f'(0)s+\alpha_n^{-1}h(\alpha_n s)$. In particular, $h_n(s):=\alpha_n^{-1}h(\alpha_n s)$ satisfies, uniformly in $n \in \N$,
    \begin{align}\label{hypNL-Rn}
      \ |h_n(s)|\leq C\alpha_n(|s|^2+|s|^p)\ \text{and}\ |h'_n(s)|\leq C\alpha_n(|s|+|s|^{p-1}).
    \end{align}
    Now, denoting 
    \begin{align*}
      \widetilde A=\begin{pmatrix}
    0 & I \\
    \Delta_g -f'(0)& 0
    \end{pmatrix}\ \text{ and }
  H_n\begin{pmatrix}
    u  \\
     v
    \end{pmatrix}=\begin{pmatrix}
    0  \\
     -h_n(u)
    \end{pmatrix},
\end{align*}
we have
    \begin{align*}
        W^n(t)=e^{t\widetilde A}W_0^n+\int_0^t e^{\widetilde A(t-s)}H_n(W^n(s))ds=W_{lin}^n+W_{Nlin}^n,
    \end{align*}
with
    \begin{align}
    \norm{W_{lin}^n}_{\mc{Z}}&\leq C\\
\label{estimWNlin}\norm{W_{Nlin}^n}_{\mc{Z}}&\leq C\alpha_n \norm{W^n}_{C^0([0, T],X)}(\norm{w^n}_{L^{4}([0, T],L^{12}(\M))}+\norm{w^n}_{L^{4}([0, T],L^{12}(\M))}^{p-1}).
    \end{align}
    This gives
        \begin{align*}
    \norm{W^n}_{\mc{Z}}&\leq C+ C\alpha_n \left( \norm{W^n}_{\mc{Z}}^2+ \norm{W^n}_{\mc{Z}}^p\right).
    \end{align*}
    In particular, a bootstrap argument (see for instance \cite[Lemma 2.2.]{BG:99} for a slightly different case) allows to prove, for $n$ large enough,
    \begin{align*}
    \norm{W^n}_{\mc{Z}}&\leq 2C,
    \end{align*}
    which, after getting back to \eqref{estimWNlin} gives 
    \begin{align*}
\norm{W_{Nlin}^n}_{\mc{Z}}&\leq C\alpha_n .
    \end{align*}
   Now, we are in position to apply the observability estimate of Proposition \ref{prop:higher-reg-obs}, taking into account that the assumptions imply $f'(0)>0$ when $\partial \M=\emptyset$. We then write it for $W^n_{lin}=(w^n_{lin},\partial_tw_{lin}^n)$,
        \begin{align*}
  \norm{(w_{0,n},w_{1,n})}_{H_0^1(\M)\times L^2(\M)}^2&\leq C  \int_0^T \norm{\mathbbm{1}_\omega\partial_t w^{n}_{lin}(t)}_{L^2(\M)}^2 dt\\
  &\leq  C  \int_0^T \norm{\mathbbm{1}_\omega\partial_t w^{n}_{Nlin}(t)}_{L^2(\M)}^2 dt+ C  \int_0^T \norm{\mathbbm{1}_\omega\partial_t w^{n}(t)}_{L^2(\M)}^2 dt.
    \end{align*}
    Concerning the first term, we use energy estimates and get
\begin{align*}
     \int_0^T \norm{\mathbbm{1}_\omega\partial_t w^{n}_{Nlin}(t)}_{L^2(\M)}^2 dt\leq C\norm{W_{Nlin}^n}_{\mc{Z}}^2\leq C\alpha_n^2.
\end{align*}
For the second term, estimate  \eqref{obsevNLintronegation} with the scaling $w_n=u_n/\alpha_n$ and $\alpha_n=\norm{(u_{0,n},u_{1,n})}_{H_0^1(\M)\times L^2(\M)}$ can be written
    \begin{align*}
   \int_0^T \norm{\mathbbm{1}_\omega\partial_t w_{n}(t)}_{L^2(\M)}^2 dt
     \leq \frac{1}{n}.     \end{align*}
The combination of the previous estimates give 
\begin{align*}
      \norm{(w_{0,n},w_{1,n})}_{H_0^1(\M)\times L^2(\M)}^2\leq C\alpha_n^2+\frac{C}{n}.
\end{align*}
Yet, $\norm{(w_{0,n},w_{1,n})}_{H_0^1(\M)\times L^2(\M)}=1$, which is a contradiction.
\end{proof}
 
\section{Applications to the Plate equation}\label{s:Plate}
In this section, we will give another application of the abstract Theorem \ref{thmabstractanalyticintro} concerning the plate equation. The purpose will be to obtain Theorem \ref{thm:analytic-nlp}. We begin by presenting the equation and notations.
\subsection{Semilinear Plate equation} Let $T>0$. Let us consider $\M$ to be a compact connected Riemannian manifold  with smooth boundary $\partial\M$ and the hinged semilinear plate equation
\begin{align}\label{eq:nlp-1}
    \left\{\begin{array}{rl}
        \partial_t^2 u+\Delta_g^2 u+f(u)=0  &\ (t, x)\in [0, T]\times \M,\\
        u_{|_{\partial\M}}=\Delta u_{|_{\partial\M}}=0            &\ (t, x)\in [0, T]\times\partial\M,\\
        (u, \partial_t u)(0)=(u_0, u_1) &\ x\in \M,
    \end{array}\right.
\end{align}
where $(u_0, u_1)\in \big(H^2(\M)\cap H_0^1(\M)\big)\cap L^2(\M)$ and $f: \R\to \R$ is assumed to be analytic and to satisfy $f(0)=0$.

Given the abstract result described in Section \ref{s:abstrIntro} and the strategy we have used to obtain unique continuation property for the wave equation, it is then natural to ask if it also holds for system \eqref{eq:nlp-1} when the nonlinearity is assumed to be analytic. 

\subsubsection{Notation}\label{notPlate} Let $A_0=-\Delta_g$ be the Laplace-Beltrami operator, equipped with Dirichlet boundary conditions if $\partial\M\neq\emptyset$. Recall that $A_0: D(A_0)\to L^2(\M)$ is a self-adjoint and nonnegative operator. In this case, its domain is given by $D(A_0)=H^2(\M)\cap H_0^1(\M)$ on $L^2(\M)$. Set $X=D(A_0)\times L^2(\M)$ and introduce the densely defined operator $\bA: D(\bA)\to X$ given by
\begin{align*}
    \bA=\left(\begin{array}{cc}
        0 & I \\
        -A_0^2 & 0
    \end{array}\right)\ \text{ with } D(\bA)=D(A_0^2)\times D(A_0).
\end{align*}
Since $A_0$ is self-adjoint, then $A_0^2$ is a strictly positive operator and so $\bA$ is skew-adjoint. As for the wave operator, a simple computation then shows that $\bA\bA^*=-\bA^2$ has compact resolvent, henceforth $\bA$ satisfies Assumption \ref{assumAA} on $X$.

For $\sigma\in [0, 1]$, in this section, $X^\sigma$ denotes the space
\begin{align*}
    X^\sigma=D(A_0^{1+\sigma})\times D(A_0^{\sigma})=H_D^{2+2\sigma}\times H_D^{2\sigma},
\end{align*}
where we recall the notation introduced in \eqref{defSob}. By Stone's theorem, $\bA$ generates a unitary $C_0$-group on $X$ and $D(\bA)$. By linear interpolation, so it does on $X^\sigma$ for any $\sigma\in [0, 1]$.

\begin{remark}
    We have chosen to consider the \emph{hinged} boundary conditions for simplicity. Other boundary conditions could have been considered, but they would require a more careful analysis; see \cite{ET:15}.
\end{remark}

For some $\widetilde{\chi}\in C^{\infty}(\M)$ to be chosen later, we set $F\in C^0(X)$ to be the map
\begin{align*}
    F: (u, v)\in X\longmapsto (0, -\widetilde{\chi}f(u))\in X.
\end{align*}
Following the exact same steps of Proposition \ref{prop:f-assumptions} and adapting the spaces to this setup, we have an analogous result for $F$. We only need to notice that for $d\leq 3$, $H^{2}_D\subset L^{\infty}$.

\begin{proposition}\label{prop:f-assumptions-pl}
    Set $\sigma=0$ and $d\leq 3$. If $f$ is real analytic and $f(0)=0$, then $F$ satisfies Assumption \ref{assumFholom} for some $\veps>0$.
\end{proposition}

\subsection{From Schr\"odinger's to Plate's observability}
The objective of this section is to discuss abstractly Lebeau's strategy \cite{1992:lebeau:schrodinger-controle} for deriving an observability inequality for the Plate equation from that of Schr\"odinger. This analysis is based on the observation that, under hinged boundary conditions, the bi-Laplace operator is precisely the square root of the Dirichlet-Laplace operator.

\subsubsection{Abstract framework for transferring observability estimates} Let $H_0$ and $Y$ be Hilbert spaces with norms $\norm{\cdot}_0$ and $\norm{\cdot}_Y$, with $H_0$ separable. Let $A$ be self-adjoint, positive\footnote{In the application, the operator $A$ will sometimes only be nonnegative because of the eigenvalue $0$. Yet, we will easily treat this subspace and this assumption makes the proof easier to write.} and boundedly invertible unbounded operator on $H_0$ with domain $D(A)$. Furthermore, we assume that $A$ has compact resolvents.

We introduce the Sobolev scale of spaces based on $A$. For any $s>0$, let $H_s$ denote the Hilbert space $D(A^{s/2})$ with the norm $\norm{x}_s=\norm{A^{s/2}x}_{0}$ (which is equivalent to the graph norm $\norm{x}_{0}+\norm{x}_s=\norm{A^{s/2}x}_{0}$ since $0\in \rho(A)$). We identify $H_0$ with its dual with respect to its inner product (i.e. we will use it as a pivot space). Let $H_{-s}$ denote the dual of $H_s$. Since $H_s$ is densely continuously embedded in $H_0$, the pivot space $H_0$ is densely continuously embedded in $H_{-s}$, and $H_{-s}$ is the completion of $H_0$ with respect to the norm $\norm{x}_{-s}=\norm{A^{-s/2}x}_0$.  We will still denote by $A$ the restriction of $A$ to $H_s$ with domain $H_{s+2}$. It is self-adjoint with respect to the $H_s$ scalar product. By the spectral theorem, there exists an orthonormal basis of $H_0$ consisting of eigenfunctions of $A$. Let us denote it by $\{(e_k, \ld_k)\}_k$, where $Ae_k=\ld_k e_k$ for each $k$. In such case, we can characterize the $H_s$-norm as follows,
\begin{align*}
    \norm{x}_s^2=\sum_{k\in \N} \ld_k^{s}|x_k|^2,\ \forall s\in \R
\end{align*}
where $x=\sum_{k\in \N} x_ke_k$. We will also make the technical assumption that there exists $N\in \N$ so that 
\begin{align}
  \label{hypWeyl}  \sum_{k\in \N}\ld_k^{-N}<+\infty.
\end{align}

Let $\mc{C}\in \mc{L}(H_2, Y)$. Let us consider the first order system for $A$,
\begin{align}
\big\{\begin{array}{ccc}\label{appendix:first-order-system}
     \dot{\psi}(t)-iA\psi(t)=0, & \psi(0)=\psi_0\in H_0, & y(t)=\mc{C}\psi(t). \\
\end{array}
\end{align}
We assume that $\mc{C}$ is an admissible observation operator for $e^{i\cdot A}$, that means, for some $\tau>0$ (and thus for any times by the group property), there exists $K_\tau\geq 0$ such that
\begin{align*}
    \int_0^\tau \norm{\mc{C}e^{itA}\psi_0}_Y^2dt\leq K_\tau \norm{\psi_0}_{H_0}^2,\ \ \forall \psi_0\in H_2.
\end{align*}
\begin{remark}
    Under the admissibility assumption, the output map $\psi_0\mapsto \mc{C}\psi_0$ from $H_2$ to $L_{loc}^2(\R, Y)$ has a continuous extension to $H_0$.
\end{remark}
\begin{definition}
    Let $-\infty< \tau_1<\tau_2<\infty$. We say that system \eqref{appendix:first-order-system} is weakly observable by $\mc{C}$ in $(\tau_1, \tau_2)$ if there exists $C>0$ such that
    \begin{align*}
        \norm{\psi_0}_{H_0}^2\leq C\left(\int_{\tau_1}^{\tau_2} \norm{\mc{C}e^{itA}\psi_0}_{Y}^2dt+\norm{\psi_0}_{H_{-2}}^2\right),
    \end{align*}
    for any $\psi_0\in H_0$.
\end{definition}

Note that the conservation of $H_0$ and $H_{-2}$-norms and translation invariance in time show that it is equivalent to state it for $\tau_1=0$, the only relevant quantity being $\tau_2-\tau_1$. Moreover, changing $\psi(t)$ to $\psi(-t)$ allows to get the same result for $e^{-itA}$.

Consider the second-order observability system
\begin{align}\label{appendix:second-order-system}
\left\{\begin{array}{cl}
     \ddot{z}(t)+A^2z(t)=0, & \\
      z(0)=z_0\in H_2, & \dot{z}(0)=z_1\in H_0, \\
\end{array}\right.
\end{align}
with the observation being either
\begin{align*}
    y(t)=\mc{C}Az(t)\ \text{ or }\ y(t)=\mc{C}\dot{z}(t).
\end{align*}
System \eqref{appendix:second-order-system} can be recast as a first order system, with $x(t)=(z(t), \dot{z}(t))\in X=H_2\times H_0$, 
\begin{align*}
    \A(z_0, z_1):=(z_1, -A^2z_0)\ \text{ with }\ D(\A)=H_4\times H_2,
\end{align*}
considering as possible observations
\begin{align}
    \label{defobservplate}
    \bC(z_0, z_1)=\mc{C}Az_0\\
    \bC_1(z_0, z_1)=\mc{C}z_1.
\end{align}
Note that $\bC\in \mc{L}(H_4\times H_2, Y)$ is admissible for $e^{\cdot \A}$and the same holds for $\bC_1$.

\begin{proposition}\label{prop:w-plates-schr}
    Let $0 < \widetilde{T} < T < \infty$. If the Schr\"odinger-like equation 
    \eqref{appendix:first-order-system} is weakly observable by $\mc{C}\in \mc{L}(H_2, Y)$ on $(0, \Tilde{T})$, then, the Plate-like equation \eqref{appendix:second-order-system} is weakly observable by $\bC$ or $\bC_1$ on $(0, T)$ defined in \eqref{defobservplate}. That means, there exists $C_1>0$ such that
    \begin{align}\label{ineq:wobs-plates}
        \norm{(z_0, z_1)}_{H_2\times H_0}^2\leq C_1\left(\int_0^T \norm{\bC e^{t\A}(z_0, z_1)}_{Y}^2dt+\norm{(z_0, z_1)}_{H_{0}\times H_{-2}}^2\right),\\
           \norm{(z_0, z_1)}_{H_2\times H_0}^2\leq C_1\left(\int_0^T \norm{\bC_1 e^{t\A}(z_0, z_1)}_{Y}^2dt+\norm{(z_0, z_1)}_{H_{0}\times H_{-2}}^2\right),
    \end{align}
    for all $(z_0, z_1)\in H_2\times H_0$.
\end{proposition}
\begin{proof}
Based on the factorization $\ddot{z}+A^2z=(\partial_t+iA)(\partial_t-iA)z$, we introduce the splitting
\begin{align*}
    z_+=\dfrac{1}{2}(z_0-iA^{-1}z_1),\ \ z_-=\dfrac{1}{2}(z_0+iA^{-1}z_1),
\end{align*}
so that
\begin{align*}
    z_0=z_++z_-,\ z_1=iA(z_+-z_-).
\end{align*}
Consequently $(z(t), \dot{z}(t))=e^{t\A}(z_0, z_1)$, solution of \eqref{appendix:second-order-system}, can be rewritten as
\begin{align*}
    z(t)=e^{itA}z_++e^{-itA}z_-.
\end{align*}
For any $s\in \R$, denote by $\Lambda: H_s\times H_{s-2}\to H_s\times H_s$ the isomorphism corresponding to the previous splitting $\Lambda(z_0, z_1)=(z_+, z_-)$. Observe that it is almost an isometry
\begin{align*}
    \norm{(z_0, z_1)}_{H_s\times H_{s-2}}^2=\norm{z_++z_-}_{H_s}^2+\norm{A(z_+-z_-)}_{H_{s-2}}^2=2\big(\norm{z_+}_{H_s}^2+\norm{z_-}_{H_s}^2\big).
\end{align*}
By considering the above splitting, we can write for the $H_s\times H_{s-2}$-energy of the system \eqref{appendix:second-order-system} as follows
\begin{align*}
    \E_s(z_0, z_1)=\dfrac{1}{2}\norm{(z_0, z_1)}_{H_{s}\times H_{s-2}}^2=\E_s(\Ld^{-1}(z_+, z_-)).
\end{align*}

Let $T',T'' \in (0,T)$ and $\rho \in C^\infty_c(\mathbb{R},\mathbb{R}^+)$ 
be such that $T''-T'>\widetilde{T}$, $\rho \equiv 1$ on $(T',T'')$ and $\text{supp}(\rho) \subset (0,T)$. We have the identity
\begin{multline*}
    \int_0^T \norm{\rho(t)\bC e^{t\A}(z_0, z_1)}_{Y}^2=\int_0^T \norm{\rho(t)\mc{C}Ae^{itA}z_+}_{Y}^2dt+\int_0^T \norm{\rho(t)\mc{C}Ae^{-itA}z_-}_{Y}^2dt\\+2\Re\int_0^T \rho^2(t)\inn{\mc{C}Ae^{itA}z_+, \mc{C}Ae^{-itA}z_-}_{Y}dt.
\end{multline*}
while 
\begin{multline*}
    \int_0^T \norm{\rho(t)\bC_1 e^{t\A}(z_0, z_1)}_{Y}^2=\int_0^T \norm{\rho(t)\mc{C}Ae^{itA}z_+}_{Y}^2dt+\int_0^T \norm{\rho(t)\mc{C}Ae^{-itA}z_-}_{Y}^2dt\\-2\Re\int_0^T \rho^2(t)\inn{\mc{C}Ae^{itA}z_+, \mc{C}Ae^{-itA}z_-}_{Y}dt.
\end{multline*}
On the one hand, since $A$ and $e^{\pm i\cdot A}$ commute, the weak observability inequality for \eqref{appendix:first-order-system} implies
\begin{align*}
    \int_0^T \norm{\rho(t)\mc{C}Ae^{\pm itA}z_{\pm}}_{Y}^2dt\geq \dfrac{1}{C}\big(\norm{Az_\pm}_{H_0}^2-\norm{Az_\pm}_{H_{-2}}^2\big).
\end{align*}
Putting the two inequalities together, we get
\begin{align*}
    \int_0^T \norm{\rho(t)\mc{C}Ae^{itA}z_+}_{Y}^2dt+\int_0^T \norm{\rho(t)\mc{C}Ae^{-itA}z_-}_{Y}^2dt&\geq \dfrac{1}{2C}\big(\E_2(\Lambda^{-1}(z_0, z_1))-\E_{0}(\Lambda^{-1}(z_0, z_1))\\
    &=\dfrac{1}{2C}\big(\E_2(z_0, z_1)-\E_{0}(z_0, z_1)\big).
\end{align*}
On the other hand, to treat the interaction term, we first look at its spectral expansion. By denoting $z_{\pm, k}=\inn{z_{\pm}, e_k}_{H_0}$ the Fourier coefficients of $z_{\pm}$, we have
\begin{align*}
    \int_0^T \rho^2(t)\inn{\mc{C}Ae^{itA}z_+, \mc{C}Ae^{-itA}z_-}_{Y}dt&=\sum_{j, k} z_{+, j}z_{-, k} \inn{\mc{C}Ae_j, \mc{C}Ae_k}_{Y} \left(\int_\R e^{it(\ld_j+\ld_k)}\rho^2(t)dt\right).
\end{align*}
First of all, we use that $\mc{C}$ is a bounded operator from $H_2$ to $Y$ to estimate $\inn{\mc{C}e_j, \mc{C}e_k}_{Y}$ as follows
\begin{align*}
    |\inn{\mc{C}e_j, \mc{C}e_k}_{Y}|\leq \norm{\mc{C}e_k}_{Y}\norm{\mc{C}e_j}_{Y}\leq C\norm{Ae_k}_{H_0}\norm{Ae_j}_{H_0}=C\ld_k\ld_j.
\end{align*}
Since $\rho^2\in C_c^\infty(\R)$ is smooth compactly supported function, for any $N\in \N$, we can find $C_N>0$ such that $|\widehat{\rho^2}(\xi)|\leq C_N \left<\xi\right>^{-N}$ and 
\begin{align*}
    \left|\int_\R e^{it(\ld_j+\ld_k)}\rho^2(t)dt\right|\leq \dfrac{C_N}{(\ld_j+\ld_k)^{N}} \leq\dfrac{C_N}{(\ld_j\ld_k)^{N/2}},\ \forall j, k\in \N_0.
\end{align*}
We then have
\begin{align*}
    \left|\int_0^T \rho^2(t)\inn{A\mc{C}e^{itA}z_+, A\mc{C}e^{-itA}z_-}_{Y}dt\right|&\leq \sum_{j, k\in \N_0} \ld_j\ld_k |z_{+, j}||z_{-, k}|\dfrac{C_N}{(\ld_j\ld_k)^{N/2}}\\
    &\leq C_N\sum_{j, k\in \N_0} \dfrac{1}{(\ld_j\ld_k)^{N/2-1}} \left(|z_{+, j}|^2+|z_{-, k}|^2\right)\\
    &\leq C_N\left(\sum_{j\in \N_0}\dfrac{1}{\ld_j^{N/2-1}}\right)\left(\sum_{k\in \N_0} \dfrac{1}{\ld_k^{N/2-1}} \big(|z_{+, k}|^2+|z_{-, k}|^2\big)\right),
\end{align*}
where we have chosen $N$ large enough so that $\ld_J^{1-N/2}$ is summable by \eqref{hypWeyl}. By adjusting $N$ if necessary, we conclude the proof by observing that
\begin{align*}
    \sum_{k\in \N_0} \dfrac{1}{\ld_k^{N/2-1}} \big(|z_{+, k}|^2+|z_{-, k}|^2\big)\leq \norm{(z_+, z_-)}_{H_{-2}, H_{-2}}^2=2\norm{(z_0, z_1)}_{H_{0}, H_{-2}}^2.
\end{align*}
\end{proof}

Let us make the following unique continuation assumption on the pair $(A, \mc{C})$.

\begin{assump}{UCP}\label{assumAeigen}
    For any eigenvector $\psi$ of $A$ such that $\mc{C}\psi=0$, then $\psi\equiv 0$.
\end{assump}

\begin{proposition}\label{prop:nt-plates}
    Let $0 < \widetilde{T} < T < \infty$. Assume that the pair $(A, \mc{C})$ satisfies Assumption \ref{assumAeigen}. If the Schr\"odinger-like equation 
    \eqref{appendix:first-order-system} is weakly observable by $\mc{C}$ on $(0, \Tilde{T})$, then, for any non zero $(z_0, z_1)\in H_2\times H_0$,
    \begin{align*}
        \bC e^{t\A}(z_0, z_1)\neq 0\ \text{ in }\ L^2([0, T], Y),\\
            \bC_1 e^{t\A}(z_0, z_1)\neq 0\ \text{ in }\ L^2([0, T], Y).
    \end{align*}
\end{proposition}
\begin{proof}
 We consider first the observation by $\bC$. Let us consider
    \begin{align*}
        N_T:=\{(z_0, z_1)\in H_2\times H_0\ |\ \bC e^{t\A}(z_0, z_1)=0,\ \text{ in } L^2([0,T], Y)\}.
    \end{align*}
    Since the equation is linear, it is a linear subspace of $H_2\times H_0$. Proposition \ref{prop:w-plates-schr} implies that for all $(z_0, z_1)\in N_T$,
    \begin{align}\label{prop:ineq:plates_nt}
        \norm{(z_0, z_1)}_{H_2\times H_0}\leq C\norm{(z_0, z_1)}_{H_0\times H_{-2}}.
    \end{align}

    We will argue by contradiction to prove that $N_T=\{0\}$. Let $0<\veps<T-\widetilde{T}$ and $(z_0, z_1)\in N_T$. Let us introduce the sequence
    \begin{align*}
        (z_0^\veps, z_1^\veps)=\dfrac{1}{\veps}\big(e^{\veps\A}(z_0, z_1)-(z_0, z_1)\big),
    \end{align*}
    and note that it belongs to $N_{\widetilde{T}}$. Observe that $(z_0, z_1)\in H_2\times H_0$ implies that $(A^{-1}z_0, A^{-1}z_1)\in H_4\times H_2=D(\A)$. By classical semigroup theory, we have
    \begin{align*}
        (A^{-1}z_0^\veps, A^{-1}z_1^\veps)=\dfrac{1}{\veps}\big(e^{\veps\A}(A^{-1}z_0^\veps, A^{-1}z_1^\veps)-(A^{-1}z_0^\veps, A^{-1}z_1^\veps)\big)\xrightarrow[\veps\to 0]{} \A (A^{-1}z_0, A^{-1}z_1),
    \end{align*}
    where the convergence holds in $H_2\times H_0$. It follows that $\big((z_0^\veps, z_1^\veps)\big)_{\veps}$ is a Cauchy sequence for the norm $\norm{\cdot}_{H_0\times H_{-2}}$ and so it is for the norm $\norm{\cdot}_{H_2\times H_{0}}$, due to inequality \eqref{prop:ineq:plates_nt}. Therefore, the limit $\A(z_0, z_1)$ belongs to $H_2\times H_0$, namely, $(z_0, z_1)\in D(\A)$ and so $N_T\subset D(\A)$. The regularity in time of $e^{t\A}$ allows us to take $\partial_t$ in the observation condition, obtaining
    \begin{align*}
       \partial_t\bC e^{t\A}(z_0, z_1)=\bC \partial_t e^{t\A}(z_0, z_1)=\bC e^{t\A}\A(z_0, z_1).
    \end{align*}
    This means that $N_T$ is stable under $\A$ and it only contains elements of $D(A^k)\times D(A^k)$ for any $k\in \N$.

    By applying inequality \eqref{prop:ineq:plates_nt} to $\A(z_0, z_1)$, we get $\norm{(z_0, z_1)}_{H_4\times H_2}\leq C\norm{(z_0, z_1)}_{H_2\times H_0}$. We deduce that the unit ball of $N_T$ in the $H_2\times H_0$ topology, is bounded in $H_4\times H_2$ and thus it is compact by compact embedding (recall, $A$ has compact resolvent). By Riesz's theorem, $N_T$ is a finite-dimensional subspace of $H_2\times H_0$.  Since $-\A^2$ is self-adjoint positive and sends $N_T$ into itself, it admits an eigenvalue $\ld\geq 0$ associated to the eigenvector $(v_\ld, w_\ld)$.  We readily arrive at the system
    \begin{align*}
       \left\{\begin{array}{cc}
           A^2 v_\ld=\ld v_\ld, & \\
            A^2w_\ld=\ld w_\ld, &  \\
        \end{array}\right.
    \end{align*}
Since $A^2$ is a strictly positive operator, we can write $\lambda=\alpha^2$ for $\alpha> 0$ and get $(A+\alpha I)(A-\alpha I)v_{\lambda}=0$. Using that $A+\alpha I$ is a strictly positive operator, we get $Av_{\lambda}=\alpha v_{\lambda}$ and $Aw_{\lambda}=\alpha w_{\lambda}$. In particular, $e^{t\A}(v_\ld, w_\ld)=(\cos(\alpha t)v_\ld+\alpha^{-1}\sin (\alpha t) w_\ld,-\alpha\sin(\alpha t)v_\ld+\cos (\alpha t) w_\ld)$. Therefore, such eigenvector must satisfy $\bC e^{t\A}(v_\ld, w_\ld)=0$ for $t\in [0,T]$ and therefore $\mc{C}A v_\ld=\mc{C}A w_\ld=0$. Using Assumption \ref{assumAeigen} and $\alpha\neq 0$, we deduce that $(v_\ld,w_{\ld})=(0,0)$. This is a contradiction to the fact that $N_T\neq \{0\}$.

    For $\bC_1$, the proof is the same except that it leads to the unique continuation problem $-\alpha\sin(\alpha t)\mc{C}v_\ld+\alpha \cos (\alpha t) \mc{C}w_\ld=0$, $t\in [0,T]$, for which Assumption \ref{assumAeigen} still applies.
\end{proof}

\begin{proposition}\label{prop:obs-schr-plates}
    Let $0 < \widetilde{T} < T < \infty$. If the Schr\"odinger-like equation 
    \eqref{appendix:first-order-system} is weakly observable on $(0, \Tilde{T})$ and Assumption \ref{assumAeigen} is satisfied, then, there exists $\mathfrak{C}_{obs}>0$ such that
    \begin{align}\label{ineq:obs-plates}
        \norm{(z_0, z_1)}_{H_2\times H_0}^2\leq \mathfrak{C}_{obs}^2\int_0^T \norm{\bC e^{t\A}(z_0, z_1)}_{Y}^2dt\\
    \label{ineq:obs-platesC1}      \norm{(z_0, z_1)}_{H_2\times H_0}^2\leq \mathfrak{C}_{obs}^2\int_0^T \norm{\bC_1 e^{t\A}(z_0, z_1)}_{Y}^2dt
    \end{align}
    for all $(z_0, z_1)\in H_2\times H_0$.
\end{proposition}
\begin{proof}
    The result follows from Proposition \ref{prop:w-plates-schr} and Proposition \ref{prop:nt-plates}. We write it for $\bC$, but it is exactly the same for $\bC_1$.

    Assume that the inequality \eqref{ineq:obs-plates} does not hold. Then we can find a sequence $(z_0^n, z_1^n)$ of norm $1$ in $H_2\times H_0$ such that
    \begin{align}\label{plates-obs-conv}
        \int_0^T \norm{\bC e^{t\A}(z_0^n, z_1^n)}_{Y}^2 dt\xrightarrow[n\to\infty]{} 0.
    \end{align}
    Let $(z_0, z_1)$ be the weak limit of $(z_0^n, z_1^n)$. The above inequality implies that $(z_0, z_1)\in N_T$ and so $(z_0, z_1)=(0, 0)$ by Proposition \ref{prop:nt-plates}. We then have that $(z_0^n, z_1^n)\wto (0, 0)$ weakly in $H_2\times H_0$ and by compact embedding (recall $A$ has compact resolvent) we get $\norm{(z_0^n, z_1^n)}_{H_0\times H_{-2}}\to 0$ as $n\to \infty$. Applying the weak observability estimate \eqref{ineq:wobs-plates} along with \eqref{plates-obs-conv} we get that $\norm{(z_0^n, z_1^n)}_{H_2\times H_0}\to 0$ as $n\to \infty$, which is a contradiction.
\end{proof}

\begin{remark}
    Lebeau's strategy gives a loss on time, harmless for our ends. Following Miller \cite{2012:miller:resolvent-control}, it is possible to avoid such loss in time but it would imply to inherit the geometric setup of the Wave equation, which is known to be more restrictive than the one of Schr\"odinger.
\end{remark}

\subsection{Proof of Theorem \ref{thm:analytic-nlp}}
We follow the notations of Section \ref{notPlate}. We will first need the following observability estimates in order to apply the abstract results. 
\begin{proposition}\label{prop:obs-plate}
 Let $\Tilde{T}>0$ and $\omega$ be an open subset of $\M$ so that the Schr\"odinger equation 
    \eqref{appendix:first-order-system} is weakly observable by $\mc{C}=\mathbbm{1}_{\omega}\in \mc{L}( L^2(\M))$ on $(0, \Tilde{T})$. Then, for any $0<\Tilde{T}<T<+\infty$ and any $b_{\omega}\in C^{\infty}(\M)$ so that $b_{\omega}=1$ on $\omega$, there exists $C>0$ such that for any $s\in [0,2]$, any $(z_0, z_1)\in H_D^{2+s}\times H^s_D(\M)$ and associated solution $z$ of
    \begin{align*}
    \left\{\begin{array}{rl}
        \partial_t^2 z+\Delta_g^2 z=0  &\ (t, x)\in [0, T]\times \M,\\
        z_{|_{\partial\M}}=\Delta_g z_{|_{\partial\M}}=0            &\ (t, x)\in [0, T]\times\partial\M,\\
        (z, \partial_t z)(0)=(z_0, z_1) &\ x\in \M,
    \end{array}\right.
\end{align*}
we have, if $\partial M\neq \emptyset$ 
       \begin{align}\label{thm:ineq:obs-plate-1}
        C\norm{(z_0, z_1)}_{H^{2+s}_D\times H^s_D(\M)}^2\leq \int_0^T \norm{b_\omega \Delta_g z(t)}_{H^s_D(\M)}^2 dt,
    \end{align}
    and if $\partial M= \emptyset$ 
    \begin{align}\label{thm:ineq:obs-plate-1bis}
        C\norm{(z_0, z_1)}_{H^{2+s}\times H^s(\M)}^2\leq \int_0^T \norm{b_\omega z(t)}_{H^{2+s}(\M)}^2 dt.
    \end{align}
\end{proposition}
\begin{proof}
    In the case $\partial M= \emptyset$, the operator $-\Delta_g$ is not strictly positive because of constants. So, we first prove \eqref{thm:ineq:obs-plate-1} in the case $\partial M\neq \emptyset$ or $\partial M= \emptyset$ but $(z_0, z_1)$ are orthogonal to the set of constants. By interpolation, it is enough to consider the cases $s=0$ and $s=2$.

  First of all, the eigenvalue's condition \eqref{hypWeyl} is satisfied, for instance, due to Weyl's law. The unique continuation of eigenfunctions for the laplacian $\Delta_g$ is known to hold in our framework, see \cite[Proposition 5.2.]{LLR:book1}, hence Assumption \ref{assumAeigen} is satisfied.
    
    For $s=0$, this is a consequence of the Proposition \ref{prop:obs-schr-plates} for the case \eqref{ineq:obs-plates}   and the fact that $b_{\omega}=1$ on $\omega$. For the second one, assume that $(z_0, z_1)\in H^4_D\times H^2_D$. Take $w=\partial_t z$ and observe that it is a mild solution of
    \begin{align*}
        \left\{\begin{array}{rl}
            \partial_t^2 w+\Delta_g^2 w=0  &\ \\
            w_{|_{\partial\M}}=\Delta w_{|_{\partial\M}}=0    &\ \text{if } \partial\M\neq\emptyset\\
            (w, \partial_t w)(0)=(z_1, \Delta z_0),
        \end{array}\right.
    \end{align*}
    with $(z_1, \Delta_g z_0)\in H_D^2(\M)\times L^2(\M)$. Applying the observability inequality \eqref{ineq:obs-platesC1} to $w$ and then going back to the $z$ variable, we get
    \begin{align*}
        \norm{(z_1, z_0)}_{H_D^4\times H_D^2}^2\leq C_2\int_0^T \norm{b_\omega \Delta_g^2 z(t)}_{L^2(\M)}^2dt.
    \end{align*}
    Note that
    \begin{align*}
        \int_0^T \norm{b_\omega \Delta_g^2 z(t)}_{L^2(\M)}^2dt\leq \int_0^T \norm{\Delta_g(b_\omega\Delta_g z)}_{L^2(\M)}^2dt+\int_0^T \norm{[b_\omega, \Delta_g]\Delta_g z}_{L^2(\M)}^2dt,
    \end{align*}
    so we now need to estimate the commutator term appearing on the right-hand side of the inequality above. Recall $[b_\omega, \Delta_g]\Delta_g z=2\nabla b_\omega\cdot\nabla(\Delta_g z)+\Delta_g b_\omega\Delta_g z$. In the following estimates, the constant $C>0$ may change from line to line,
    \begin{align*}
        \int_0^T \norm{[b_\omega, \Delta_g]\Delta_g z}_{L^2(\M)}^2dt&\leq C\left(\int_0^T \norm{\nabla b_\omega\cdot\nabla(\Delta_g z)}_{L^2(\M)}^2dt+\int_0^T \norm{\Delta_g b_\omega\Delta_g z}_{L^2(\M)}^2dt\right)\\
        &\leq C\int_0^T \norm{z(t)}_{H_D^3}^2dt\\
        &\leq C\norm{(z_0, z_1)}_{H_D^3\times H_D^1}^2\\
        &\leq C\veps \norm{(z_0, z_1)}_{H_D^4\times H_D^2}^2+\dfrac{C}{\veps}\norm{(z_0, z_1)}_{H_D^2\times L^2(\M)}^2\\
        &\leq C\veps \norm{(z_0, z_1)}_{H_D^4\times H_D^2}^2+\dfrac{C}{\veps}\int_0^T \norm{b_\omega \Delta_g z(t)}_{L^2(\M)}^2 dt,
    \end{align*}
    for $\veps>0$ to be chosen. Observe that we have used energy estimates, an interpolation inequality and the observability inequality \eqref{thm:ineq:obs-plate-1} for $s=0$. By choosing $\veps>0$ small enough, the observability inequality \eqref{thm:ineq:obs-plate-1} for $s=2$ follows once we put all the inequalities above together.

    It only remains to prove \eqref{thm:ineq:obs-plate-1bis} when there is no boundary. Decomposing $(z_0,z_1)=\pi_0(z_0,z_1)+\pi_0^{\perp}(z_0,z_1)$ where $\pi_0$ is the projection on the eigenvalues $0$ of $\Delta_g$, that is the constants, we have obtained up to now, noticing the constant part of the initial data produce some part of the solution with zero Laplacian, 
    \begin{align*}
        \norm{\pi_0^{\perp}(z_0,z_1)}_{H^{2+s}\times H^s(\M)}^2\leq C\int_0^T \norm{b_\omega \Delta_g z(t)}_{H^s(\M)}^2 dt.
    \end{align*}
    By adding the components corresponding to eigenvalue zero and noticing that $[b_{\omega},\Delta_g]$ is a differential operator of order one, we obtain
      \begin{align*}
        \norm{(z_0,z_1)}_{H^{2+s}\times H^s(\M)}^2\leq C\int_0^T \norm{ b_\omega z(t)}_{H^{2+s}(\M)}^2 dt+C\norm{(z_0,z_1)}_{H^{1+s}\times H^{s-1}(\M)}^2.
    \end{align*}
    A compactness-uniqueness argument as in Proposition \ref{prop:obs-schr-plates} allows to conclude.
\end{proof}
\subsubsection{Observability for the Schr\"odinger equation}\label{sssec:obs-schr} Let $\M$ be a compact Riemannian manifold with or without boundary equipped with a metric $g$ and take $\omega\subset \M$. In what follows, we consider an observation operator $\mathcal{C}\psi=\mathbbm{1}_\omega \psi$, unless we specify otherwise.

In any of the situations described in Section \ref{s:plateintro}, an observability inequality at the $ L^2$ level holds for the linear Schr\"odinger equation. We summarize this discussion in the following result. 
\begin{theorem}\label{thm:schr-obs}
    If we are in any of the examples described in Section \ref{s:plateintro}, then for every $T>0$, there exists $C>0$ such that
    \begin{align*}
        \norm{v_0}_{L^2(\M)}^2\leq C\int_0^T \norm{\mc{C} e^{it\Delta_g}v_0}_{L^2(\M)}^2dt,
    \end{align*}
    for all $v_0\in L^2(\M)$.
\end{theorem}
Note that $e^{it\Delta_g}$ is the flow with Dirichlet boundary condition in case that $\partial \M\neq \emptyset$. We now verify that the pair $(\bA, \bC)$ satisfies Assumption \ref{assumcommu}.
\begin{proposition}\label{propcommutplate}
    Let $\sigma\in [0, 1]$, $\sigma \neq 1/4$. Then, if $\bC$ given by $ \bC(\phi, \psi)=(0, b_\omega\Delta_g \phi)$ with $b_{\omega}$ smooth satisfying $\partial_{\vec n}b_{\omega}=0$ on $\partial \M$, then Assumption \ref{assumcommu} is fulfilled with $s=1/2$ as long as $\veps\leq 1/2$, $\sigma+\veps<5/4$ and $\sigma+\veps\neq 1/4$. In the case $\partial M=\emptyset$, the same result holds with $ \bC(\phi, \psi)=(b_\omega\phi,0)$.
\end{proposition}
\begin{proof}
    We compute $[(A^*A)^{1/2},\bC]=\begin{psmallmatrix}
    0 & 0\\
    [b_{\omega},\Delta_g] \Delta_g & 0
    \end{psmallmatrix}$. Therefore, the result is true as long as $[b_{\omega},\Delta_g]\Delta_g=-2\nabla_g b_{\omega}\cdot \nabla_g \Delta_g-(\Delta_g b_{\omega})\Delta_g$ sends $H^{4+2\sigma}_D$ into $H^{2(\sigma+\veps)}_D$, that is, as long as $[b_{\omega},\Delta_g]=-2\nabla_g b_{\omega}\cdot \nabla_g -\Delta_g b_{\omega}$ sends $H^{2+2\sigma}_D$ into $H^{2(\sigma+\veps)}_D$. The assumption $\veps\leq 1/2$ ensures that the loss of derivative is correct, whereas $\sigma+\veps<5/4$ ensures that $H^{2(\sigma+\veps)}_D$ is either $H^{2(\sigma+\veps)}(\M)$ or $H^{2(\sigma+\veps)}_0(\M)$ with the Dirichlet boundary condition (provided we avoid the value $1/2$). Since $\partial_{\vec n}b_{\omega}=0$, this gives the result. This is similar in the other case.
\end{proof}

We now come to the proof of the main result of this section.

\begin{proof}[Proof of Theorem \ref{thm:analytic-nlp}]
    It only remains to check that Theorem \ref{thmabstractanalyticintro} can be applied with the abstract notations in Section \ref{notPlate}, with $\sigma=0$. The proof is very similar to the one of Theorem \ref{thm:analytic-prop}.
    
    We have already established that $\bA$ satisfies Assumption \ref{assumAA}. Using Lemma \ref{lemma:omomtilde-smaller-subset} below, we can construct successively $b_{\omega}$ and $\chi$ smooth, compactly supported in $\widetilde{\omega}$, with $\partial_{\vec n}\chi=\partial_{\vec n}b_{\omega}=0$ on $\partial M$ and so that $b_{\omega}=1$ on $\omega$ and $\chi=1$ on $\supp(b_{\omega})$. Take $\bC(\phi, \psi)=(0,b_\omega \Delta_g \phi)$ or $\bC(\phi, \psi)=(b_\omega \phi,0)$ if $\partial \M=\emptyset$ so that Proposition \ref{prop:obs-plate} applies. Therefore, we get that that $\bA$ satisfies Assumption \ref{assumCC} with $\bC\in \mc{L}(X)$ and $\bC\in \mc{L}(X^\veps)$ for some $\veps>0$. Also, Proposition \ref{propcommutplate} applies so that Assumption \ref{assumcommu} is fulfilled.  
    By writing $u=\chi u+(1-\chi)u$, we want to prove that $t\mapsto (1-\chi)u$ is analytic. Set $\Tilde{\chi}=(1-\chi)$ and consider the new variable $z=\Tilde{\chi} u$. We then have
    \begin{align*}
        \partial_t^2 z+\Delta_g^2 z &=\Tilde{\chi}(\partial_t^2 u+\Delta_g^2 u)-[\Delta_g^2, \Tilde{\chi}]u=-\Tilde{\chi} f(u)+[\Delta_g^2, \chi]u\\
        &=-\Tilde{\chi} f(z+\chi u)+[\Delta_g^2, \chi]u=-\Tilde{\chi} f(z+h_1)+h_2,
    \end{align*}  
     with $h_1=\chi u$ and $h_2=[\Delta_g^2,\chi]u$. From Proposition \ref{prop:f-assumptions-pl}, we see that $F$ satisfies Assumption \ref{assumFholom}. Since the multiplication by $\chi$ maps $H^2_D$ into itself, while $[\Delta_g^2,\chi]$ maps $H^{3+\veps}(\M)$ into $H^{\veps}_D$ for $\veps<1/2$. So, if we choose $\sigma=0$ and $\veps<1/2$, the assumptions imply that $h_1$ and $h_2$ are analytic with value in $H^2_D$ and $H^{\veps}_D$, respectively. Now, the conclusion follows as a direct application of Theorem \ref{thmabstractanalyticintro} in the same way as Theorem \ref{thm:analytic-prop}. We conclude that $t\mapsto (z(t), \partial_t z(t))$ is analytic with value in $X^0$. By assumption, $t\mapsto \chi u(t)$ is analytic with value in $H^{3+\veps}\cap H_0^1$ (and so it is the same for $\chi \partial_t u$) and therefore with value in $H^2_D$. So, $t\mapsto (\chi u(t),\chi\partial_t u(t))$ is analytic with value in $X^0$. 
     
     By summing up, we obtain that $t\mapsto ( u(t),\partial_t u(t))$ is analytic with value in $X^0$. Using the equation again as in Theorem \ref{thm:analytic-prop}, we obtain that $t\mapsto \bA U(t)$ is analytic with value in $X^0$. Hence $t\mapsto U(t)$ is analytic with value in $H^4\cap H_0^1(\M)\cap H^2\cap H_0^1(\M)$, which finishes the proof.
\end{proof}

\appendix
\setcounter{theorem}{0}
\renewcommand{\thetheorem}{\Alph{section}\arabic{theorem}}

\section{} 

\subsection{ODEs in Banach spaces} We now introduce the two different notions of ODEs in Banach spaces used in the present article. Let us consider the framework of Section \ref{sec:abstract-construction} and let $I\subset \R$ be a nonempty interval and take $s_0\in I$. 

For any $s\in \R$, we can easily extend $e^{sA}$ to $C^{0}([0,T],X^{\sigma})$ by the formula
\begin{align*}
\left[e^{sA}V\right](t)=e^{sA}V(t),
\end{align*}
for $V\in C^{0}([0,T],X^{\sigma})$. If $H\in L^{1}\left(I,C^{0}([0,T],X^{\sigma})\right)$, we say that $\xi\in C^{0}\left(I,C^{0}([0,T],X^{\sigma})\right)$ satisfies 
\begin{align}\label{appendix:ode:ode-banach-1}
    \left\{\begin{array}{lc}
        \dfrac{d}{ds}\xi(s)=A \xi(s)+H(s), & s\in I, \\
        \xi(s_0)=\xi_0, & 
    \end{array}\right.
\end{align}
with $\xi_0\in C^{0}([0,T],X^{\sigma})$, if it satisfies
\begin{align}\label{appendix:ode:ode-banach-1-duhamel}
\xi(s)&=e^{(s-s_0)A}\xi_0+\int_{s_0}^{s}e^{(s-w)A}H(w)dw,\ \forall s\in I,
\end{align}
with equality in $C^{0}([0,T],X^{\sigma})$.
\begin{lemma}
\label{lmDuhamelCauchy}
If $H\in L^{1}\left(I,C^{0}([0,T],\P_{n}X^{\sigma})\right)$ and $\xi\in C^{0}\left(I,C^{0}([0,T],\P_{n}X^{\sigma})\right)$ for some $n\in\N$ and satisfies 
$ \dfrac{d}{ds}\xi(s)=A \xi(s)+H(s)$ in the previous sense. Then, it satisfies this equation in the sense of Cauchy-Lipschitz.
\end{lemma}
\begin{proof}
It follows as an application of Duhamel's formula.
\end{proof}
\begin{lemma}\label{lmtranslat}
For $T_1<T_2$, let us consider $T\in (0, T_2-T_1)$, $\eta\in (0, T_2-T-T_1)$ and $I:=[T_1-\eta, T_2-T-\eta]$. Let $G\in C^{0}([T_{1},T_{2}],X^{\sigma})$ and assume that $V\in C^{0}([T_{1},T_{2}],X^{\sigma})$ is a mild solution of 
\begin{align*}
    \left\{\begin{array}{lc}
    \dfrac{d}{dt}V(t)=A V(t)+G(t) &\ \text{ for }\ t\in [T_1, T_2]\\
   V(T_1)=V_0. & 
    \end{array}\right.
\end{align*}
If we define $\xi,H\in C^{0}\left(I, C^{0}([0,T],X^{\sigma})\right)$ by $\xi(s)=V^s$ and $H(s)=G^s$ with 
\begin{align*}
    \xi(s)(t)=V^{s}(t)=V(t+s+\eta)\ \text{ and }\ H(s)(t)=G^{s}(t)=G(t+s+\eta),\quad \forall s\in I, t\in [0,T],
\end{align*}
then, for any $s_0\in I$, $\xi$ is solution in the sense of \eqref{appendix:ode:ode-banach-1-duhamel} of
\begin{align*}
    \left\{\begin{array}{lc}
        \dfrac{d}{ds}\xi(s)=A \xi(s)+H(s), & s\in I, \\
        \xi(s_0)=\xi_0, & 
    \end{array}\right.
\end{align*}
with $\xi_0=V^{s_0}=V(\cdot+s_0+\eta)$.
\end{lemma}
\begin{proof} 
By Duhamel's formula, for all $t\in [T_1, T_2]$ we have
\begin{align*}
V(t)&=e^{(t-T_1)A}V_{0}+\int_{T_1}^{t}e^{(t-\tau)A}G(\tau)d\tau.
\end{align*}
with $V(T_1)=V_0$. Pick $s_0\in I$. Then, for $s\in I$ and $t\in [0, T]$,
\begin{align*}
V^{s}(t)&=V(t+s+\eta)=e^{(t+s+\eta-T_1)A}V_{0}+\int_{T_1}^{t+s+\eta}e^{(t+s+\eta-\tau)A}G(\tau)d\tau\\
&=e^{(s-s_0)A}\left(e^{(t+s_0+\eta-T_1)A}V_{0}+\int_{T_1}^{t+s_0+\eta}e^{(t+s_0+\eta-\tau)A}G(\tau)d\tau\right)+\int_{t+s_0+\eta}^{t+s+\eta}e^{(t+s+\eta-\tau)A}G(\tau)d\tau\\
&=e^{(s-s_0)A}V(t+s_0+\eta)+\int_{s_0}^{s}e^{(s-w)A}G(t+w+\eta)dw\\
&=e^{(s-s_0)A}V^{s_0}(t)+\int_{s_0}^{s}e^{(s-w)A}G^{w}(t)dw.
\end{align*}
So, since this is true for any $t\in [0,T]$, it gives 
\begin{align*}
V^{s}=e^{(s-s_0)A}V^{s_0}+\int_{s_0}^{s}e^{(s-w)A}G^{w}dw,\ \forall s\in I.
\end{align*}
By hypothesis, this equality holds in $C^0([0, T], X^\sigma)$, and is exactly \eqref{appendix:ode:ode-banach-1-duhamel}, as we wanted to prove.
\end{proof}

\subsection{Complex analysis in Banach spaces}\label{s:appcomplex} Let $E$ and $F$ be Banach spaces over the same field $\mathbb{K}$, with $\mathbb{K}$ being either $\R$ or $\mathbb{C}$. Along this appendix, we will introduce several notions of differentiability and analyticity needed to unify the different results used throughout the present article.

We first start with a notion of differentiability in Banach spaces, often referred to as Fr\'echet differentiability.

\begin{definition}
\label{def-Kdifferentiable}
    Let $U$ be an open subset of $E$. A mapping $f: U\to F$ is said to be $\mathbb{K}$-differentiable (or just differentiable) if for each point $x\in U$ there exists a mapping $A\in \mc{L}(E, F)$ such that
    \begin{align*}
        \lim_{h\to 0} \dfrac{\norm{f(x+h)-f(x)-Ah}_F}{\norm{h}_E}=0.
    \end{align*}
    Such map $A$ is called the derivative of $f$ at $x$ and it is denoted by $Df(x)$.
\end{definition}

We can rephrase the above definition as follows: for each $x\in U$ there exists a mapping $A\in \mc{L}(E, F)$ such that
\begin{align*}
    f(x+h)=f(x)+Ah+o(h)
\end{align*}
where $o(h)/\norm{h}_E\to 0$ as $h\to 0$. With this formulation at hand, we state the Chain rule in this setting.

\begin{theorem}{\cite[Theorem 13.6]{1986:mujica:complex-analysis}}\label{appendix:thm:chain-rule} (Chain rule)
    Let $E$, $F$ and $G$ be Banach spaces over $\mathbb{K}$. Let $U\subset E$ and $V\subset F$ be two open sets and let $f: U\to F$ and $g: V\to G$ be two differentiable mappings with $f(U)\subset V$. Then the composite mapping $g\circ f: U\to G$ is differentiable as well and $D(g\circ f)(x)=Dg(f(x))\circ Df(x)$ for every $x\in U$.
\end{theorem}

We now come to introduce the different notions of holomorphic or analytic maps that have been used throughout the present article. A mapping $P: E\to F$ is said to be an $k$-homogeneous polynomial if there exists a $k$-linear mapping $A: E^k\to F$ such that $P(x)=A(x,\ldots, x)$ for every $x\in E$. We represent by $\mc{P}(^kE, F)$ the Banach space of all continuous $k$-homogeneous polynomials from $E$ into $F$ under the norm
\begin{align*}
    \norm{P}_{\mc{P}(^kE, F)}=\sup\{\norm{P(x)}_F\ |\ x\in E,\ \norm{x}_E\leq 1\}.
\end{align*}
A series $\sum_{k=0}^\infty f_k$ of homogeneous polynomials $f_k\in \mc{P}(^k E, F)$ will shortly be called a formal series from $E$ to $F$. The space of all formal series with continuous terms will be denoted by $S(E, F)$. We say that a formal series $\sum_{j=0}^\infty f_j$ converges in a set $U\subset E$ if for every $x\in U$ the series $\sum_{j=0}^\infty f_j(x)$ is convergent.

\begin{definition}
    Let $U$ be an open subset of $E$ and $\mathbb{K}=\mathbb{C}$ (resp. $\R$). A continuous mapping $f: U\to F$ is said to be \emph{holomorphic} (resp. \emph{analytic}) if for each $x\in U$ there exist a series $\sum_{j=0}^\infty f_j\in S(E, F)$ such that
    \begin{align*}
        f(x+h)=\sum_{j=0}^\infty f_j(h)
    \end{align*}
    for all $h$ in a neighborhood of $0\in E$. We shall denote by $\H(U, F)$ the vector space of all holomorphic mapping from $U$ into $F$.
\end{definition}

\begin{remark}
    The sequence $(f_j)$ which appears in the above definition is uniquely determined by $f$ and $x$. We then shall write $f_j=f_j(x)$ for every $j\in \N_0$.
\end{remark}

The previous definition has been taken from \cite{BS:71-analytic} and \cite{1986:mujica:complex-analysis}. Observe that here we have reserved the concept \emph{holomorphic} for the complex case and \emph{analytic} for the real case. When going through the literature, it is often the case that holomorphicity is introduced with a different definition. We will introduce these notions and then we will establish that they are equivalent. From now on, assume that $\mathbb{K}=\mathbb{C}$, unless we say otherwise.

\begin{definition}
A mapping $f: U\to F$ is said to be:
\begin{enumerate}
    \item \emph{weakly holomorphic} if $\psi\circ f$ is holomorphic for every $\psi\in F^*$, where $F^*$ is the dual space of $F$.
    \item \emph{G-holomorphic} if for all $x\in U$ and $h\in E$, the mapping $\zeta\mapsto f(x+\zeta h)$ is holomorphic on the open set $\{\zeta\in \mathbb{C}\ |\ x+\zeta h\in U\}$.
\end{enumerate}
\end{definition}

The following theorem shows that one of the most important features of the complex analysis still holds when working with functions between complex Banach spaces.

\begin{theorem}{\cite[Theorem 8.12, Theorem 8.7, Theorem 13.16]{1986:mujica:complex-analysis}}\label{appendix:thm:equiv-holom}
    Let $U$ be an open subset of $E$, and let $f: U\to F$. The following statements are equivalent:
    \begin{enumerate}
        \item $f$ is $\mathbb{C}$-differentiable;
        \item $f$ is holomorphic;
        \item $f$ is weakly holomorphic;
        \item $f$ is continuous and $G-$holomorphic.
    \end{enumerate}
\end{theorem}

For a given $x\in U$ and $h\in E$, let us denote by $\rho(x, h)$ the supremum of all numbers $\rho$ such that $|\zeta|\leq \rho$ implies $x+\zeta h\in U$.

\begin{theorem}{\cite[Theorem 7.1, Corollary 7.3]{1986:mujica:complex-analysis}}\label{appendix:prop:cauchy-form} (Cauchy integral formula)
    Let $U$ be an open subset of $E$, and let $f\in \H(U, F)$. Let $x\in U$, $h\in E$ and $r<\rho(x, h)$. Then for each $\ld\in \mathbb{D}(0, r)$ we have
    \begin{align*}
        f(x+\ld h)=\dfrac{1}{2\pi i}\int_{|\zeta|=r} \dfrac{f(x+\zeta h)}{\zeta-\ld}d\zeta,
    \end{align*}
    where $|\zeta|=r$ denotes a circle of radius $r$ an center at the origin in the complex plane. Moreover, for each $j\in \N$ we have
    \begin{align*}
        f_j(x)(h)=\dfrac{1}{2\pi i}\int_{|\zeta|=r} \dfrac{f(x+\zeta h)}{\zeta^{j+1}}d\zeta.
    \end{align*}
\end{theorem}

Let $f\in \H(U, F)$. We can expand $f(x+\ld h)$ as
\begin{align*}
    f(x+\ld h)=\sum_{j=0}^\infty f_j(x)(\ld h)=\sum_{j=0}^\infty \ld^j f_j(x)(h),
\end{align*}
which holds uniformly for $|\ld|\leq r$ with $0\leq r<\rho(x, h)$. For $x\in U$ we may define the $n$th variation $\delta^n f(x, h)$ of $f(x)$ with increment $h$ as
\begin{align*}
    \delta^n f(x, h)=\left[\dfrac{d^n}{d\zeta^n} f(x+\zeta h)\right]_{\zeta=0}.
\end{align*}
It can be seen that $\delta^n f(x, h)$ is homogeneous of degree $n$ in $h$. Moreover, looking at the Taylor development of the holomorphic map $\ld\in \mathbb{D}(0, r)\mapsto f(x+\ld h)\in F$, in view of the previous result, it follows that
\begin{align}\label{appendix:eq:cauchy-formula-derivatives}
    \delta^n f(x, h)=\dfrac{n!}{2\pi i}\int_{|\zeta|=r} \dfrac{f(x+\zeta h)}{\zeta^{m+1}}d\zeta.
\end{align}
\begin{remark}
    Formula \eqref{appendix:eq:cauchy-formula-derivatives} does not depend on the $r<\rho(x, h)$ chosen.
\end{remark}

The above discussion leads us to the classical Cauchy estimates.

\begin{proposition}{\cite[Corollary 7.4]{1986:mujica:complex-analysis}}\label{appendix:prop:cauchy-est} (Cauchy estimates)
    Let $U$ be an open subset of $E$, and let $f\in \H(U, F)$. Let $x\in U$, $h\in E$ and $r<\rho(x, h)$. Then for each $n\in \N$ we have
    \begin{align*}
        \norm{\delta^n f(x)(h)}\leq r^{-n}\sup_{|\zeta|=r}\norm{f(x+\zeta h)}.
    \end{align*}
\end{proposition}

\begin{remark}
Actually, it is possible to have the Cauchy estimates locally around any point $x\in U$ or even uniformly in a ball (by assuming that $f$ is bounded there). Let us argue for the former case, the latter being similar. By continuity, there exists $r_x>0$ such that $\norm{f(z)}\leq M$ for all $z\in U$ such that $\norm{z-x}\leq r_x$, where $M=M(x)>0$ is a bound that depends on $x$. Let $h\in E$. Thus, for $z$ such that $\norm{z-x}\leq r_x/2$, we have $z+\zeta h\in U$ for any $|\zeta|\leq \tfrac{r_x}{2\norm{h}}$, since
\begin{align*}
    \norm{z+\zeta h-x}\leq \norm{z-x}+\norm{h}\leq \dfrac{r_x}{2}+\dfrac{r_x}{2\norm{h}} \norm{h}=r_x
\end{align*}
and so $z+\zeta h\in B(x, r_x)\subset U$. Due to the Cauchy estimates
\begin{align*}
    \norm{\delta^n f(z)(h)}&\leq \left(\dfrac{2\norm{h}}{r_x}\right)^n \sup_{|\zeta|=\tfrac{r_x}{2\norm{h}}}\norm{f(z+\zeta h)}\leq M \left(\dfrac{2\norm{h}}{r_x}\right)^m.
\end{align*}
The previous estimate holds uniformly on $\norm{z-x}\leq r^*$, for any $r^*<\tfrac{r_x}{2}$.
\end{remark}

We characterize holomorphic mappings whose domain is an open set in a product of Banach spaces. 

\begin{proposition}{\cite[Proposition 8.10]{1986:mujica:complex-analysis}}\label{appendix:thm:holom-several-variables}
    Let $E_1,\ldots, E_n$ and $F$ Banach spaces, and let $U$ be an open subset of $E_1\times\ldots \times E_n$. Then a mapping $f: U\to F$ is holomorphic if and only if $f$ is continuous and $f(\zeta_1,\ldots, \zeta_n)$ is holomorphic in each $\zeta_j$ when the other variables are held fixed.
\end{proposition}

We can say the following in regards to the regularity of the $n$-th variation of $f$.

\begin{proposition}\cite[Proposition 6.4]{BS:71-analytic}\label{prop:differential-analytic}
    Assume $\mathbb{K}=\mathbb{C}$ (resp. $\R$). If $f: U\to F$ is holomorphic (resp. analytic), then for every $n\in \N$ the function
    \begin{align*}
        \delta^n f: (x, h)\in U\times E\longmapsto \delta^n f(x)(h)\in F
    \end{align*}
    is holomorphic (resp. analytic).
\end{proposition}

\subsubsection{Some analysis results} Here we state two useful results that are used throughout the article in a complex-analytic context.

The following result, known as the Uniform Contraction Principle, elucidates the regularity that can be obtained for a parameter-dependent fixed point.

\begin{theorem}{\cite[Theorem 2.2]{1982:chow-hale:bifurcation-theory}}\label{app:thm:uniffixedpoint}
    Let $U$, $V$ be open sets in Banach spaces $X$, $Y$, let $\overline{U}$ be the closure of $U$, $T:\overline{U}\times V\to \overline{U}$ a uniform contraction on $\overline{U}$ and let $g(y)$ be the unique fixed point of $T(\cdot, y)$ in $\overline{U}$. If $T\in C^k(\overline{U}\times V, X)$, $0\leq k<\infty$, then $g(\cdot)\in C^k(V, X)$. If there is a neighborhood $U_1$ of $\overline{U}$ such that $T$ is analytic from $U_1\times V$ to $X$, then the mapping $g(\cdot)$ is analytic from $V$ to $X$.
\end{theorem}

Observe that the definition of analyticity introduced by \cite{1982:chow-hale:bifurcation-theory} combines $G$-analyticity and weakly analyticity. From Theorem \ref{appendix:thm:equiv-holom}, we deduce that these notions of analyticity are equivalent.

The following theorem permits to treat a real analytic function as a restriction of some holomorphic function.

\begin{theorem}\cite[Theorem 7.2]{BS:71-analytic}\label{app:thm:h-ext}
    Assume $\mathbb{K}=\R$. For any analytic function $f: U\to F$ one may find an open subset $V$ of $E_{\mathbb{C}}$ and a holomorphic function $\Tilde{f}: V\to F_\mathbb{C}$ such that $U\subset V$ and $\Tilde{f}_{|_U}=f$.
\end{theorem}

\subsection{Geometric fact}\label{s:geom}
In this section, we describe briefly the compressed cotangent bundle and prove a geometric Lemma that has been used several times in the article. We refer to Melrose-Sj\"ostrand \cite{MS:78} H\"ormander \cite[Section 18.3 and 24.3]{H:07} for more precisions and \cite[Section 2.2]{BL:01} in the more specific context of the wave equation.

Let $\M$ be a smooth compact Riemannian manifold of dimension $d$ with boundary.  Denote $^bT\M$ the bundle of rank $d$ whose sections are the vector fields tangent to $\partial \M$, by $^bT^*\M$ the dual bundle (Melrose's compressed cotangent bundle), and by $\mathfrak{j}:T^*\M\rightarrow~ ^bT^*\M$ the canonical map. This is the restriction map,  dual to the embedding $^bT\M\hookrightarrow T\M$. Its image can be identified with $T^* (\text{Int}\M) \sqcup T^* \partial\M$ with an appropriate topology.

We set $^bS^*\M=(^bT^*\M\setminus \M_0)/\R_+^*$ to be the cosphere bundle of $^bT^*\M$ the compressed cotangent bundle. Here $\M_0\approx \M$ is the zero section. The map $\mathfrak{j}$ can be defined on the quotient and allows to define $B:=\mathfrak{j}(S^*\M)\subset ~^bS^*\M$. It can be identified with the image of $S^*\M$ by the continuous map $\mathfrak{j}$ and is therefore a compact space when it is equipped with the natural topology of vector bundle inherited from $~^bS^*\M$.

The bicharacteristic flow is usually defined for nonelliptic operators, but the link with generalized geodesics can, for instance, be made as follows, see Lebeau \cite[Section A.3]{L:96}. Let $P=\partial_t^2-\Delta_g$ be the wave operator defined on the manifolds with boundary $\X=\R_t\times \M$. $p=|\xi_x|_g^2-\xi_t^2$, the principal symbol of $P$, is well defined on $T^*\X\approx T^*\R_t\times T^*\M$ and $p^{-1}$ is conical and therefore well defined in $T^*\X$ and $S^*\X$. We denote $Z=\mathfrak{j}(p^{-1}(0))\subset~ ^bT^*\X$ and $SZ=(Z\setminus \X_0)/\R_+^*\subset ~ ^bS^*\X$. $SZ$ has actually two connected components corresponding to $\xi_t>0$ and $\xi_t<0$ where $\xi_t$ is the variable dual to $t$, that we denote $Z^+$ and $Z^-$. We can see that $SZ^+$ can be identified with $(Z\setminus \X_0)\cap \{\xi_t=1/2\}$ and to $\R_t\times~ \mathfrak{j}(S^*\M)$. In that context, since $\xi_t$ is invariant by the flow, if we denote $G$ the bicharacteristic flow of Melrose-Sj\"ostrand, we see that it can be written
\begin{align*}
    G(s)(t,x,\xi)=(t+s,\phi_s (x,\xi)),
\end{align*}
where $\phi_s$ is a well defined flow on $B=\mathfrak{j}(S^*\M)\subset ~^bS^*\M$, the generalized geodesic flow.

We denote $\pi$ the natural projection from $^bS^*\M$ to $\M$. Both are continuous with the natural topologies given. For $\phi_t$, this is a consequence of the continuity of the bicharacteristic flow $G$, see \cite[Lemma 3.31]{MS:78}. 
\begin{lemma}\label{lemma:gcc-smaller-subset}
    Suppose that $(\omega, T)$ satisfies \ref{assumGCC}. Then there exist $\chi\in C^{\infty}_c(\omega)$ with non negative values so that 
    \begin{itemize}
        \item   there exists an open set $\widetilde{\omega}\Subset \omega$ and a time $0<\widetilde{T}<T$ such that $(\widetilde{\omega}, \widetilde{T})$ satisfies \ref{assumGCC},
        \item $\chi=1$ on $\widetilde{\omega}$,
        \item $\partial_{\vec n} \chi=0$ on $\partial \M$ where $\partial_{\vec n}$ is the normal derivative to the boundary.
    \end{itemize}

\end{lemma}
\begin{proof}
    Let $\rho=(x, \xi)\in B$. By \ref{assumGCC}, there exists $t=t(\rho)\in (0, T)$ such that $x_t:=\pi\circ\phi_t(\rho)\in \omega$.
    
    We first assume $x_t\in \partial \M$, the case $x_t\in \text{Int}(\M)$ being simpler.

    In a sufficiently small neighborhood of $x_t$, we can find some geodesic normal coordinates $(x_1,x')\in [0,2\eta)\times B_{\R^{d-1}}(0,2\eta)$ satisfying the following properties.
The point $x_t$ is $(0,0)$ in these coordinates, $\partial \M=\{x_1=0\}$, $\M=\{x_1\geq 0\}$ and the metric $g$ can be written in a diagonal type form $\left(\begin{array}{cc}
          1   & 0 \\
          0   &g'(x_1,x') 
        \end{array}\right)$ where $g'$ is a metric on $\R^{d-1}$ depending smoothly on $(x_1,x')$. In particular, in these coordinates, the normal vector at the boundary pointing inside is $\vec n=\frac{\partial}{\partial x_1}$.

By assumption, $x_t\in \omega$, so up to diminishing $\eta$, we can assume that $[0,2\eta)\times B_{\R^{d-1}}(0,2\eta)\subset \omega$.

Let $\varphi\in C^{\infty}_c((-2,2),[0,1])$ so that $\varphi=1$ on $(-1,1)$. In these coordinates, we define $\chi_{\rho}(x_{1},x')= \varphi(x_{1}/\eta)\varphi(|x'|/\eta)$. We verify that $\chi_{\rho}$ satisfies $\chi_{\rho}\geq 0$, $\chi_{\rho}=1$ in $[0,\eta)\times B_{\R^{d-1}}(0,\eta)$ which contains $(0,0)$, $\partial_{x_1}\chi_{\rho}=0$ on $\{x_1=0\}$ and $\text{Supp}(\chi_{\rho})\subset \omega$. Independently on the coordinates in $\M$, the properties of $\chi_{\rho}$ can be written
    \begin{itemize}
       \item $\chi_{\rho}\geq 0$,
        \item $\chi_{\rho}=1$ in an open set $\omega_{\rho}\subset \M$ (open for the topology of a manifold with boundary) which contains $x_t$, and we can select another open set $\widetilde{\omega}_{\rho}\Subset \omega_{\rho}$ with $x_t\in \widetilde{\omega}_{\rho}$,
        \item $\partial_{\vec n}\chi_{\rho}=0$ on $\partial \M$,
        \item $\chi_{\rho}\in C^{\infty}_c(\omega)$.
    \end{itemize}
    In the case $x_t\in \Int{\M}$, a cutoff function $\chi_{\rho}$ can be found with the same properties.

 Since $x_t:=\pi\circ\phi_t(\rho)\in \widetilde{\omega}_{\rho}$ with $\widetilde{\omega}_{\rho}$ open and for fixed $t$, $\pi\circ\phi_t$ is continuous from $B$ to $\M$, we can find a neighborhood $\mc{V}_{\rho}\subset B$ of $\rho$ such that
    \begin{align}
    \label{inclusflow}
        \pi\circ\phi_t(\mc{V}_\rho)\subset\widetilde{\omega}_{\rho}.
    \end{align}

    The open sets $\mc{V}_{\rho}$ form an open covering of  $B$ by such neighborhoods. The compactness of $^bS^*\M$ allows to extract a finite subcovering of them $\mc{V}_{\rho_1},\ldots, \mc{V}_{\rho_k}$ so that 
    \begin{align}
    \label{finitecover}B=\bigcup_{i=1}^{k}\mc{V}_{\rho_{i}}.
    \end{align} 
    
    We define $\widetilde{\chi}=\sum_{i=1}^{k}\chi_{\rho_{i}}$, $\widetilde{\omega}=\bigcup_{i=1}^{k}\widetilde{\omega}_{\rho_i}$, $F=\bigcup_{i=1}^{k}\overline{\widetilde{\omega}_{\rho_i}}$ and $W=\bigcup_{i=1}^{k}\omega_{\rho_i}$. Note that we have $ \widetilde{\omega}\subset F\subset W\subset \omega$. For $\widetilde{\chi}$, we have the following properties        \begin{itemize}
       \item $\chi\geq 0$ on $\M$,
        \item $\widetilde{\chi}\geq 1$ in the open set $W$,
        \item $\partial_{\vec n}\widetilde{\chi}=0$ on $\partial \M$,
        \item $\widetilde{\chi}\in C^{\infty}_c(\omega)$.
    \end{itemize}
    We state the following Claim that we will prove later.
    \medskip
    
\textbf{Claim:} Let $F\subset W\subset \M$ with $F$ closed and $W$ open. Then, there exists $h\in C^{\infty}(\M)$ so that $h\geq 0$ on $\M$, $h=0$ on $F$, $h\geq 1$ on $\M\setminus  W$ and $\partial_{\vec n}h=0$ on $\partial \M$.

\medskip

 We now define $\chi=\frac{\widetilde{\chi}}{\widetilde{\chi}+h}$. $\widetilde{\chi}+h\geq 1$ on $\M$ so $\chi\in C^{\infty}(\M)$. It satisfies $\partial_{\vec n}\chi=0$ on $\partial \M$ by composition since it is the case for $\widetilde{\chi}$ and $h$. Moreover, since $h=0$ on $F$, we have $\chi=1$ on $F$.
    
    We now check that $(\widetilde{\omega},\widetilde{T})$ satisfies \ref{assumGCC} for some $\widetilde{T}\in (\max\{t(\rho_j),j=1,\dots, k\}, T)$.    
    
    Indeed, let $\rho \in B$. By the finite covering property \eqref{finitecover}, there exists $j\in \llbracket 1,k \rrbracket $ so that $\rho\in \mc{V}_{\rho_j}$. \eqref{inclusflow} implies then $\pi\circ\phi_{t(\rho_{j})}(\rho)\in \widetilde{\omega}_{\rho_{j}}$ and then $\pi\circ\phi_{t(\rho_{j})}(\rho)\in\widetilde{\omega}$ as expected.
 
  The proof is now complete except for the proof of the Claim.   
  
  Let $x\in \M\setminus W$. In particular, $x\in \M \setminus F$ which is open. By following the same method as in the first part of the proof, we can construct $h_{x} \in C^{\infty}_c(\M \setminus F)$ so that $h_{x}\geq 0$, $h_{x}=1$ in an open set $\omega_{x}\subset \M\setminus F$ which contains $x$. We can also assume $\partial_{\vec n}h_{x}=0$ on $\partial \M$. The $\omega_{x}$ form a covering of the compact set $\M\setminus W$. So, we can select a finite covering $\M\setminus W=\bigcup_{i=1}^{k} \omega_{x_{i}}$ and define $h=\sum_{i=1}^{k}h_{i}$ which satisfies the expected properties.
\end{proof}
With similar arguments, we can also prove the following result.
\begin{lemma}\label{lemma:omomtilde-smaller-subset}
    Suppose that $\omega$ and $\widetilde{\omega}$ are two open subsets of $\M$ so that $\omega \Subset \widetilde{\omega}$. Then there exists $\chi\in C^{\infty}_c(\widetilde{\omega})$ with non negative values so that 
    \begin{itemize}
          \item $\chi=1$ on $\omega$,
        \item $\partial_{\vec n} \chi=0$ on $\partial \M$.
    \end{itemize}
\end{lemma}
\begin{proof}
Since $\omega \Subset \widetilde{\omega}$ where both sets are open, we can find $\omega_1$ open so that $\omega \Subset \omega_1 \Subset \widetilde{\omega}$. We apply the claim with $F=\M\setminus \widetilde{\omega}$, $W=\M\setminus \overline{\omega_1}$ to get $h\in C^{\infty}_c(\widetilde{\omega})$ so that $h\geq 1$ on $\overline{\omega_1}$. If we apply again the claim with $F=\overline{\omega}$, $W=\omega_1$ to get $m\in C^{\infty}(\M)$ so that $m=0$ on $\overline{\omega}$ and $m\geq 1$ on $\M\setminus\omega_1$. In both cases, the functions are nonnegative and we have $\partial_{\vec n} h=\partial_{\vec n} m=0$ on $\partial \M$. The function $\chi=\frac{h}{m+h}$ is regular and has the desired properties.
\end{proof}

\bibliographystyle{alpha}
\bibliography{bibliography.bib}

\end{document}